\documentclass[11pt]{article}

\font\notsosmall=cmr8

\usepackage{lipsum} 

\usepackage{amssymb}
\usepackage{amsthm}
\usepackage{amsfonts}
\usepackage[latin2]{inputenc}
\usepackage{amsmath}
\usepackage{anysize}
\usepackage{hyperref}
\usepackage{bbm}
\usepackage{comment}
\usepackage{enumitem}

\newtheorem{theorem}{Theorem}[section]
\newtheorem{lemma}[theorem]{Lemma}
\newtheorem{prop}[theorem]{Proposition}
\newtheorem{proposition}[theorem]{Proposition}
\newtheorem{corollary}[theorem]{Corollary}
\newtheorem{conjecture}[theorem]{{Conjecture}}

\newtheorem{question}[theorem]{{Question}}
\newtheorem{exer}[theorem]{{Exercise}}
\newtheorem{claim}[theorem]{{Claim}}

\theoremstyle{remark}
\newtheorem{remark}[theorem]{Remark}

\theoremstyle{definition}
\newtheorem{example}[theorem]{{Example}}

\newtheorem{definition}[theorem]{{Definition}}
\newtheorem{hypothesis}[theorem]{{Hypothesis}}
\newtheorem{notation}[theorem]{{Notation}}

\def\bclaim{\begin{claim}}
\def\eclaim{\end{claim}}
\def\bexer{\begin{exer}}
\def\eexer{\end{exer}}
\def\bquestion{\begin{question}}
\def\equestion{\end{question}}
\def\bdefin{\begin{definition}}
\def\edefin{\end{definition}}
\def\bcor{\begin{corollary}}
\def\ecor{\end{corollary}}
\def\bthm{\begin{theorem}}
\def\ethm{\end{theorem}}
\def\bconj{\begin{conjecture}}
\def\econj{\end{conjecture}}
\def\blem{\begin{lemma}}
\def\elem{\end{lemma}}
\def\blemma{\begin{lemma}}
\def\elemma{\end{lemma}}
\def\bprop{\begin{prop}}
\def\eprop{\end{prop}}
\def\bremark{\begin{remark}}
\def\eremark{\end{remark}}
\def\bhyp{\begin{hypothesis}}
\def\ehyp{\end{hypothesis}}
\def\bnot{\begin{notation}}
\def\enot{\end{notation}}
\def\bexample{\begin{example}}
\def\eexample{\end{example}}

\def\MA{Monge--Amp\`ere }
\def\K{K\"ahler }
\def\i{\sqrt{-1}}
\def\del{\partial}
\def\dbar{\bar\partial}
\def\ddbar{\del\dbar}
\def\ra{\rightarrow}
\def\eps{\epsilon}
\newcommand{\RR}{\mathbb{R}}
\newcommand{\CC}{\mathbb{C}}
\newcommand{\NN}{\mathbb{N}}

\newcommand{\PP}{\mathbb{P}}
\def\del{\partial}
\newcommand{\calH}{\mathcal{H}}

\def\sm{\setminus}

\def\vol{\hbox{vol}}
\DeclareMathOperator{\Ric}{Ric}

\def\sm{\setminus}
\def\w{\wedge}
\def\o{\omega}

\newcommand{\dcal}{\mathcal{D}}
\newcommand{\ecal}{\mathcal{E}}
\newcommand{\hcal}{\mathcal{H}}

\newcommand{\lcal}{\mathcal{L}}\newcommand{\mcal}{\mathcal{M}}

\newcommand{\rcal}{\mathcal{R}}

  \def\calR{\rcal}
\def\calM{\mcal}  \def\calL{\lcal}
 \def\calH{\hcal} \def\H{\hcal}
  \def\calE{\ecal}

\def\a{\alpha} \def\be{\beta}
\def\o{\omega} 
 
\def\vp{\varphi} \def\eps{\epsilon}

\def\K{K\"ahler } 
\def\KE{K\"ahler--Einstein } \def\KEno{K\"ahler--Einstein}

\def\polishl{\char'40l}
 \def\Blocki{B\polishl{}ocki}
\def\Holder{H\"older }

\def\Ric{\hbox{\rm Ric}\,} 

\def\ovpn{\o^n_{\vp}}
\def\gM{g_{\hbox{\sml M}}}

\def\h#1{\hbox{#1}}
\def\text{\textstyle}
\def\dis{\displaystyle}
\def\q{\quad} \def\qq{\qquad}

\def\PSH{\mathrm{PSH}}

\def\Aut{{\operatorname{Aut}}}
\def\id{{\operatorname{id}}}
\def\Isom{{\operatorname{Isom}}}

\def\h#1{\hbox{#1}}
\def\MA{Monge--Amp\`ere }

\def\bpf{\begin{proof}}
\def\epf{\end{proof}}
\def\beq{\begin{equation}}
\def\eeq{\end{equation}}
\def\beqno{\begin{equation*}}
\def\eeqno{\end{equation*}}
\def\eaeq{\end{aligned}}
\def\baeq{\begin{aligned}}

\newcommand\blfootnote[1]{%
  \begingroup
  \renewcommand\thefootnote{}\footnote{#1}%
  \addtocounter{footnote}{-1}%
  \endgroup
}

\def\Ric{\hbox{\rm Ric}\,}
\def\Rico{\Ric\,\!\o}

\def\opcit{\underbar{\phantom{aaaaa}}}

\def\ovp{\omega_{\varphi}}

\def\ovpn{\omega^n_{\varphi}}

\def\on{\omega^n}
\def\fo{f_\omega}

\def\Ho{\calH_\omega}
\def\ginv{g^{-1}}

\def\fs{\mathfrak{s}}

\def\lb{\label}

\def\Ent{\h{\rm Ent}}

\def\aut{\h{\rm aut}}
\def\AutMJ{\Aut(M,\JJJ)}
\def\AutMJz{\Aut(M,\JJJ)_0}
\def\autMJ{\aut(M,\JJJ)}

\def\E{\calE}

\def\vpt{\vp_t}

\def\beq{\begin{equation}}
\def\eeq{\end{equation}}
\def\beqno{\begin{equation*}}
\def\eeqno{\end{equation*}}
\def\eaeq{\end{aligned}}
\def\baeq{\begin{aligned}}

\def\bpf{\begin{proof}}
\def\epf{\end{proof}}
\def\Rico{\Ric\,\!\o}

\def\a{\alpha} \def\be{\beta}
\def\o{\omega} 
 
\def\vp{\varphi} \def\eps{\epsilon}

\def\ovp{{\o_{\vp}}}
\def\on{\omega^n}
\def\ovpn{\omega_{\vp}^n}

\def\intM{\int_M}
\def\Ric{\hbox{\rm Ric}\,}

\def\ovpn{\o^n_{\vp}}
\def\gM{g_{\hbox{\sml M}}}

\def\gE{g_{\hbox{\sml E}}}
\def\gC{g_{\hbox{\sml C}}}

\def\gM{g_{\hbox{\sml M}}}
\def\dM{d_{\hbox{\sml M}}}

\def\dC{d_{\hbox{\smlsev C}}}
\def\dCpq{d_{\hbox{\smlsev C,p,q}}}
\def\dMinput#1{d_{\hbox{\sml M,#1}}}
\def\dCinput#1{d_{\hbox{\sml C,#1}}}
\def\ellC{\ell_{\hbox{\smlsev C}}}
\def\ellM{\ell_{\hbox{\smlsev M}}}

\def\dD{d_{\hbox{\smlsev D}}}
\def\dDG{d_{\hbox{\smlsev D},G}}
\def\ellD{\ell_{\hbox{\smlsev D}}}

\def\h#1{\hbox{#1}}
\def\mathoverr#1#2{\buildrel #1 \over #2}

\def\strutdepth{\dp\strutbox}
\def\specialstar{\vtop to \strutdepth{
    \baselineskip\strutdepth
    \vss\llap{$\star$\ \ \ \ \ \ \ \ \  }\null}}
\def\marginalstar{\strut\vadjust{\kern-\strutdepth\specialstar}}

\def\marginal#1{\strut\vadjust{\kern-\strutdepth
    {\vtop to \strutdepth{
    \baselineskip\strutdepth
    \vss\llap{{ \small #1 }}\null}
    }}
    }

\def\fo{f_\omega}
\def\ga{\gamma}

\def\JJJ{\h{\rm J}}
\def\JJJsml{\h{\notsosmall J}}

\def\isom{\h{\rm isom}}
\def\V{V^{-1}}

\font\sml=cmr6
\font\smlsev=cmr7

\font\Bbb=msbm10 
\font\Bbbfootnote=msbm7 scaled \magstephalf

\def\outlinfootnote#1{\hbox{\Bbbfootnote #1}}

\def\sseq{\subseteq}
\def\tpsi{\tilde\psi}
\def\CCfoot{{\outlinfootnote C}}

\def\intm{\int_M}
\def\qqq#1{\hskip#1em\relax}
\def\D{\Delta}

\title{
Tian's properness conjectures: \\ an introduction to K\"ahler geometry
}

\author{Yanir A. Rubinstein
\blfootnote{2010 Mathematics subject classification 
32Q20, 
58E11 (primary); 
53C25, 
53C55, 
14J50, 
32W20, 
32U05 (secondary).} 
}
\date{}

\begin{document}

\maketitle

\def\thhit#1{\hbox{${\hbox{#1}}^{\,{\hbox{\itnotsosml th}}}$}}
\font\itnotsosml=cmti7

\vglue-0.3cm
\centerline{\it Dedicated to Gang Tian
on the occassion of
his
6\thhit{0}
 birthday }

\begin{abstract}
This manuscript served as lecture notes for a minicourse in the 2016 Southern California Geometric Analysis Seminar Winter School.
The goal is to give a quick introduction to K\"ahler geometry by describing the recent resolution of Tian's three influential properness conjectures
in joint work with T. Darvas. These results---inspired by and analogous to 
work on the Yamabe problem in conformal geometry---give an  analytic characterization for the existence of K\"ahler--Einstein metrics and other important canonical 
metrics in  complex geometry, as well as strong borderline Sobolev type inequalities referred to as the (strong) Moser--Trudinger inequalities. 

\end{abstract}

\tableofcontents


\section*{Prologue}
\addcontentsline{toc}{section}{Prologue}

Harmonic functions are special. They enjoy a high degree of regularity, and in some vague sense are considered to be more aesthetically pleasing than an arbitrary function.
In geometry, one similarly seeks aesthetically pleasing structures on a given space. A typical example is that of an Einstein structure. Among all Riemannian structures Einstein structures are special in many ways; the interested reader is referred to the book by Besse \cite{Besse}.

Harmonic functions can be defined as solutions to the Laplace equation. 
A fundamental result in analysis is that harmonic functions are also characterized as minimizers of the Dirichlet energy. This result is fundamental in many ways.
First, the Dirichlet  energy makes sense for functions whose gradient is merely square integrable while Laplace's equation requires two derivatives to exist pointwise.
Second, it gives an approach to actually constructing harmonic functions.

K\"ahler--Einstein metrics can be defined as solutions to a fully nonlinear 
analogue of the Laplace equation. A nonlinear analogue of the Dirichlet 
energy was introduced by Mabuchi 30 years ago. 
A basic aspect of this analogy is that the Euler--Lagrange equation for the Mabuchi functional is precisely the K\"ahler--Einstein equation.
One possible way to view these lectures is as an attempt to explain some aspects of this analogy in more detail. In doing so, we strive to give a quick---and at least partly introductory---course in K\"ahler geometry.

\section*{A second prologue}
\addcontentsline{toc}{section}{A second prologue}

The K\"ahler--Einstein problem is also strongly motivated by an analogy with the Yamabe problem, that is, of course, itself motivated by the classical Dirichlet problem described above. The following table serves as an overall guidance to the K\"ahler--Einstein problem, especially for those familiar with the resolution of the Yamabe problem. 
Our goal in these lectures is to describe the right column of this table,
culminating in a complete analytic description of 
``Tian's properness conjectures" and ``Tian's
Moser--Trudinger conjecture" at the bottom right.

A few remarks are in place. 
First, this table is highly schematic, and its main purpose
is to highlight some possible analogies between the two analytic problems. Second, the infimum in the definition of
$
\mu_{[\o]}
$
is, of course, solely for the analogy, since 
$
\mu_{[\o]}
=\frac12\int_M R_{g_{\ovp}}\ovpn/\int_M\ovpn
$
is a cohomological invariant of the \K class.
Third, an alternative sufficient and necessary condition
that appeared very recently and after these lectures were delivered
can now be described in terms of
an invariant that is different but related to 
$\alpha_{[\omega]}$ coming from algebraic geometry
and K-stability  \cite{FujitaOdaka,cds,T15}. 
Yet, this last characterization is purely algebraic, and so it is less pertinent to the analogy with the Yamabe problem.
Finally, one may also discuss the more general problem of constant scalar curvature (csc) \K metrics. While this is beyond the scope of the lectures,
it is worth mentioning briefly the state-of-the-art on this problem at the time the lectures were given. Indeed, in \cite{DR2} aside from solving the \KE case we also
were able to reduce the general csc problem to the regularity of weak minimizers of the Mabuchi energy.  Shortly afterwards, our techniques played a r\^ole 
in the resolution of this regularity problem, and hence of the analytic characterization of constant scalar curvature \K metrics \cite{BDL2,CC,CC2}.
Some of these important developments are already described in the survey \cite{Dsur},
while others just appeared and seem to involve important new ideas beyond
the scope of these lectures.

\vfill\eject

\vglue-1cm
\hglue-1.8cm
\begin{tabular}{|l|l|l|}
\hline
\qqq1&
Yamabe problem\qqq{7} & K\"ahler--Einstein problem\qqq{4} \\
\hline
{\it structure} & Riemannian manifold $(M^n,g)$
& \K manifold $(M^{2n},\hbox{\rm J},\omega)$ \\
&&\\
{\it class} & conformal class $(M^n,[g])$ & \K class $(M^{2n},[\omega])$\\
& $=\Big\{Fg\,:\, C^\infty(M)\ni F>0$ \Big\} & $=\Big\{\o+\i\ddbar\vp>0\,:\,
\vp\in C^\infty(M)\Big\}$ \\
&&\\
{\it problem} & (non)-existence of constant & (non)-existence of a \KE \\
& scalar curvature metric in & metric in a \K class $(M^{2n},[\omega])$\\
& a conformal class $(M^n,[g])$ &\\
&&\\
{\it equation} &
$u^{N-2}g$ has constant scalar
&
$\ovp:=\o+\i\ddbar\vp$ has constant
\\
& curvature $\mu$ (here $N:=\frac{2n}{n-2}$) 
&Ricci curvature $\mu$\\
&&\\
&\qqq6$\Updownarrow$&\qqq8$\Updownarrow$\\
&&\\
&$\mu=\dis u^{N-1}\big(R_gu-(N+2)\D_gu\big)$&
$\mu=\dis-\log\frac{\ovpn}{\on e^{f_\o}}\Big/\vp$\\
&&\\
&(here $R_g=\;$scalar curvature of $g$)&
(here $\D_\o f_\o =\frac12(R_{g_\o}-2n\mu_{[\o]})$ where $\mu_{[\o]}$ \\
&&
is defined below and $g_\o(\,\cdot\,,\,\cdot\,):=\o(\cdot,\JJJ\cdot)$)\\
&& \\
{\it functional} &
Yamabe energy&
Mabuchi energy\\
&&\\
&
$u^{N-2}g\mapsto Y(u):=$
&
$\ovp\mapsto E(\vp):=\dis\int_M\log\frac{\ovpn}{\on e^{f_\o}}\ovpn$
\\
& $\dis\frac{\int_M\Big((N+2)|\nabla_gu|^2+R_gu^2\Big)dV_g}
{\Big(\int_Mu^NdV_g\Big)^{2/N}}$
&
$\dis-\mu\!\!\int_M\!\!\i\del\vp\w\dbar\vp\w\sum_{k=0}^{n-1}
\frac{\o^{n-1-k}\w \ovp^{k}}{{(n+1)}/{(k+1)}}$\\
&(invariant under $u\mapsto cu, \; c>0$)&
(invariant under $\vp\mapsto \vp+c, \; c\in\RR$)\\
&&\\
{\it sign} &
$\mu_{[g]}:=\dis\inf_{u>0}$ $\int_M R_{ug}dV_{ug}/(\int_MdV_{ug})^{2/N}$
&
$\mu_{[\o]}:=\dis\inf_{\vp>0}$ $\frac12\int_M R_{g_{\ovp}}\ovpn/\int_M\ovpn$
\\
\raise8pt\hbox{\it invariant}&&\\
{\it solution ($\mu\le  $0)} &
always exists (Aubin,Trudinger,Yamabe)
&
always exists (Aubin, Yau)
\\
&&\\
{\it sufficient crit-} & 
exists if $\mu_{[g]}<\mu_{[g_{S^n}]}$\qq (Aubin)
&
exists if $\alpha_{[\o]}>\dis\frac{n\mu_{[\o]}}{n+1}$\qq (Tian)
\\
\raise4pt\hbox{\it 
erion ($\mu>$0)} & where $\mu_{[g_{S^n}]}=n(n-1)\vol(S^n(1))^{\frac2n}$
& where
$\alpha_{[\o]}:=\dis\sup_{c>0} 
\{c\,:\,
\!\!\!\!\!\!\!\!
\sup_
{
\mathoverr{\;\;\;\;\;\ovp>0}{\sup\!\!\!\!\vp=-1}
}
$
$
\!\!\!\!\!\!
\int_M
e^{-c\vp}\o^n< \infty
 \}$
\\
&&\\
{\it necessary }& Aubin's criterion always holds&
Tian's properness conjectures?\\
{\it condition } & (if $[g]\neq[g_{S^n}]$) (Aubin, Schoen)
&\\
&&\\
{\it strong border-} &Aubin's strong Moser--Trudinger & Tian's Moser--Trudinger conjecture?\\
{\it line Sobolev} &inequality on ${S^n}$&\\
\hline
\end{tabular}

\vglue1cm

\section{Introduction}

The main motivation for these lectures 
are three conjectures:
Tian's properness conjectures and 
 Tian's Moser--Trudinger conjecture.
Consider the space 
\begin{equation}
\label{HEq}
\textstyle\calH
=
\{\omega_\vp:=\o+\i\ddbar\vp \,:\, \vp\in C^{\infty}(M), \,  \omega_\vp>0\}
\end{equation}
of all \K metrics representing a fixed cohomology class  
on a compact \K manifold $(M,\JJJ,\o)$. 

Motivated by results in conformal geometry 
and the direct method in the calculus of variations,
in the 90's Tian  introduced the notion of ``properness on $\calH$"
\cite[Definition 5.1]{T1994}
in terms of the Aubin nonlinear 
energy functional $J$ \cite{Aubin84}  and the Mabuchi K-energy $E$ 
\cite{Mabuchi87} as follows (both functionals are defined \S\ref{MabSec} below, see 
 \eqref{AubinEnergyEq} and  \eqref{EEq}).

\bdefin
The functional $E:\calH\ra\RR$ is said to be proper if 
\beq
\label{PropernessEq}
\forall\,\o_j\in\calH,\q
\lim_jJ(\o_j)\ra \infty 
\q\Longrightarrow\q
\lim_jE(\o_j)\ra \infty. 
\eeq
\edefin
\noindent
Tian made the following influential conjecture 
\cite[Remark 5.2]{T1994},
\cite[Conjecture 7.12]{Tianbook}.
Denote by $\AutMJz$ the identity component of the group of
automorphisms of $(M,\JJJ)$, and denote by
$\aut(M,\JJJ)$ its Lie algebra, consisting of holomorphic vector fields.

\bconj
\lb{TianConj} {\rm (Tian's first properness conjecture)}
Let $(M,\JJJ,\o)$ be a  Fano manifold. 
%
Let $K$ be a maximally compact subgroup of $\AutMJz$. Then 
$\H$ contains a K\"ahler--Einstein  
metric
if and only if $E$ is proper on the subset $\H^K\subset\H$
consisting of $K$-invariant metrics.
\econj

Tian's conjecture is central in \K geometry
since it predicts an analytic characterization of 
\KE manifolds. 
Appropriate analogues of this conjecture in conformal
geometry are known and were crucial in the solution of the famous Yamabe problem
concerning the existence of constant scalar curvature metrics in conformal classes.
We briefly discuss this in Section \ref{SecondSec}.

The conjecture has attracted
much 
attention
over the past two and a half decades including motivating 
much work on equivalence between 
algebro-geometric notions of stability and existence of 
canonical metrics, as well as on the interface
of pluripotential theory and \MA equations. 
While the algebraic-geometric characterization of \KE manifolds
has been finally obtained \cite{cds,T15}, the analytic characterization
of Conjecture \ref{TianConj} has remained open. 
We refer to the surveys \cite{Thomas,PhongSturmSurvey,Tian2012,PhongSongSturm,R14}.

Conjecture \ref{TianConj} (which we refer to as the {\it Tian's first properness conjecture}) gives a characterization of K\"ahler--Einstein manifolds in terms of the Mabuchi energy. Thus, it can be seen as the analogue of the properness of the Yamabe energy which led to the resolution of the Yamabe problem. Another central theorem in conformal geometry is
Aubin's strong Moser--Trudinger inequality on spheres.
Tian's Moser--Trudinger conjecture suggests a \K geometry analogue of this inequality on any \KE manifold,
that we now turn to describe.

Denote by $\Lambda_1$ the real eigenspace of
the smallest positive eigenvalue of $-\Delta_\o$, and set
$$
\calH_\o^\perp:=\{\vp\in\H\,:\, \int \vp\psi\on=0,\;\forall \psi\in\Lambda_1\}.
$$
When $\o$ is \KEno, 
it is well-known that $\Lambda_1$ is in a one-to-one correspondence with
holomorphic gradient vector fields \cite{ga}.
Tian made the following conjecture in the 90's 
\cite[Conjecture 5.5]{Tian97},
\cite[Conjecture 6.23]{Tianbook},\cite[Conjecture 2.15]{Tian2012}.

\bconj
\lb{TianConj2}
{\rm (Tian's Moser--Trudinger conjecture)}
Suppose $(M,\JJJ,\o)$ is Fano K\"ahler-Einstein.
Then for some $C,D>0$,
$$
E(\vp) \geq C J(\vp) - D, \qq  \vp \in \H_\o^\perp.
$$
\econj

By the end of these lectures we will present results that resolve both Conjectures
 \ref{TianConj} and \ref{TianConj2}. 
For Conjecture  \ref{TianConj}, the special case 
when $K$ is trivial
has already been known for almost 20 years from the work of Tian and Tian--Zhu \cite{Tian97,TZ00}. Treating the general case  has remained open since. Somewhat surprisingly, Conjecture
 \ref{TianConj} was actually disproved by Darvas and the author recently \cite{DR2}. More precisely, 
Theorem \ref{KEGexistenceIntroThm} establishes precisely for which manifolds Conjecture \ref{TianConj} holds, giving a converse
to a result of Phong et al. \cite{pssw}.
Next, and this is the second main result of \cite{DR2}, an alternative formulation to
Conjecture \ref{TianConj} is established---which we refer to
as {\it Tian's second properness conjecture}---giving the sought after
analytic characterization for  K\"ahler--Einstein  metrics. 
This is stated in Theorem \ref{KEexistenceIntroThm}.
Finally, the strong Moser--Trudinger inequality for \KE manifolds is established, confirming 
Conjecture \ref{TianConj2} \cite{DR2}. This is stated in Corollary \ref{TianConj2Cor}

We leave out a few relevant topics due to space and time limitations.
Notably, we mostly do not delve into the the pluripotential theoretic
and Bergman kernel aspects of the proof of Theorem \ref{KEexistenceIntroThm},
for which we refer to Darvas' survey that has appeared in the meantime \cite{Dsur}.
On the other hand, our treatment is rather self-contained 
and reviews most of the basics, giving an opportunity to the interested reader
for a rapid introduction to current research in \K geometry.


\section {K\"ahler and Fano manifolds} 

In these lectures all manifolds will be assumed to be complex. Complex manifolds are just like topological or differentiable manifolds except that the transition functions between the different charts in the atlas are required to be holomorphic in both directions (i.e., biholomorphic). Thus, all our manifolds will be of even dimension. In dimension two all complex manifolds are also {\it  K\"ahler}. In higher dimensions, however, the latter condition is rather subtle and reflects the existence of a Riemannian structure highly compatibile with the given complex structure--- we explain this next.

Let $(M,\JJJ , g)$ denote a complex manifold together with a Riemannian metric $g$ on $M$
that is compatible with $\JJJ $
in the sense that 
$$
g(x,y)=g(\JJJ x,\JJJ y), \q \forall\, x,y\in \Gamma(M,TM),
$$
where $\Gamma(M,TM)$ denote smooth vector fields on $M $.
Since $\JJJ ^2=-I$, the formula
$\omega:=\o=g(\JJJ \,\cdot\,,\,\cdot\,)$ 
defines a 2-form on $M$. 
We call $(M,\JJJ , g)$ a  K\"ahler manifold if the form $\o$ 
is a closed 2-form,
$$
d\omega=0.
$$  

\def\bdz#1{\overline{dz^{#1}}}
\def\b#1{\bar{#1}}

K\"ahler manifolds have many other equivalent characterizations; we refer the reader to \cite[\S2.1.4]{RThesis}.

In these lectures we will be interested in the curvature of  K\"ahler manifolds. In particular, we will be interested in trying to understand 
{\it when  K\"ahler manifolds admit Einstein metrics.} Just as the Riemannian metric can be transformed using the complex structure into a skew-symmetric form, so can the Ricci curvature tensor
$\Ric\, g$. We denote the Ricci form by
$$
\Rico=\Ric\, g(\JJJ \,\cdot\,,\,\cdot\,).
$$ 
Thus, on a K\"ahler manifold the Einstein equation
$$
\Ric\, g= cg,
$$
transforms to
\begin{equation}
\begin{aligned}
\label{KE1Eq}
\Rico= c\o.
\end{aligned}
\end{equation}
Let $z_1,\ldots,z_n$ be local holomorphic coordinates on a neighborhood in $M $. In those coordinates express the form $\o$,
$$
\o=g_{i\bar j}dz^i\w \bdz j.
$$
As discovered by 
Schouten \cite{Schouten1},
Schouten and van Dantzig \cite{SchoutenVanDantzig,SchoutenVanDantzig2}, and 
\K \cite{Kahler} 
(see \cite[p. 35]{RThesis} for more references)
the Ricci form then has the following expression
\beq
\label{RicciFormEq}
\Rico=
-
{\i}
\ddbar\log\det [g_{i\b j}].
\eeq
In fact, the proof is not hard. First, recall some useful formul\ae:

\bexer
\lb{DetExer}
{\rm
Let $D $ be a constant coefficient first-order operator defined on 
some domain in $\RR^m$ and let $A $ be a matrix-valued function on the same domain. Then,
$$
D\log\det A=A^{ij} DA_{ij},
$$
and
$$
D A^ {ij} = -A^ {i t} DA_{t  s} A^{s j},
$$
where $A^ {ij} $ is the coefficient in the $i $-th row
and $j $-th column of the inverse matrix of $A$. 
}\eexer

 Therefore,
$$
\ddbar\log\det [g_{i\b j}]
=
-g^{i\b t}g_{t\b s,k}g^{s\b j}g_{i\b j, \b l}
+g^{i\b j}g_{i\b j,k\b l}.
$$
Now, $d\o=0$ implies that both $\del \omega =\b\del \omega= 0 $. Thus, $g_{i\b j,k}
= g_{k\b j,i}$ and $g_{i\b j,\b k}
= g_{i\b k,\b j}$.

\begin{exer} {\rm
\label{}
Complete the proof of  \eqref{RicciFormEq}.
} \end{exer}


Thus, if $\eta$ is any  K\"ahler form such that locally
 $\eta = h_{i\bar j}\bdz i\w dz^j$ then
\begin{equation}
\begin{aligned}
\label{RicdiffEq}
\Rico-\Ric\,\eta=
{\i}{}\ddbar\log\frac{\det[h_{i\b j}]}{\det[g_{i\b j}]}
={\i}\ddbar\log\frac{\eta^n}{\on}
\end{aligned}
\end{equation}
is an exact two form on $M $ since $\log\frac{\eta^n}{\on}$ is a globally defined smooth function on $M $
as $\frac{\eta^n}{\on}>0$. Therefore, the Ricci form of any  K\"ahler metric is not only a closed two-form (as is evident from  \eqref{RicciFormEq}), but also lies in a fixed cohomology class.
Up to a constant factor, this class is called the first Chern class of $M $ and is denoted by 
$2\pi c_1 (M) $.

The point of this discussion is that Einstein metrics on a  K\"ahler manifold
 can exist only  if the equality of cohomology classes
\beq\label{c1}
\mu[\o]= 2\pi c_1(M) 
\eeq
holds. Now, as a rule of thumb,
Einstein metrics of negative Ricci curvature exist in abundance on Riemannian manifolds, while Einstein metrics of positive Ricci curvature are quite harder to come by \cite{Besse}. Somewhat analogously, it is easier to prove existence of K\"ahler--Einstein metrics of negative Ricci curvature, i.e., when $\mu <0 $: a
fundamental theorem of Aubin and Yau states that  then \eqref{c1} also implies that there exists a unique K\"ahler--Einstein metric whose cohomology class is $[\omega] $.
Around the same time, Yau also showed that the same is true when $\mu = 0 $. 
However, when $\mu>0$ it was shown by Matsushima already in the 50's that  \eqref{c1} is not sufficient. \K manifolds for which \eqref{c1} holds with $\mu>0$ are 
called Fano manifolds.
Thus, it is natural to ask:

\bquestion
When does a Fano manifold admit a  
K\"ahler--Einstein  metric?
\equestion
As just explained, if such a \KE metric exists it must
have positive Ricci curvature. (Conversely, if a \K manifold
admits a \K metric of positive Ricci curvature it is Fano.)
In these lectures we describe an answer to this question. The key player will be the Mabuchi energy, which we now turn to describe.

\section {The Mabuchi  energy}
\lb{MabSec}

Before defining the Mabuchi energy we introduce several other basic  functionals.

The two most basic functionals, introduced by Aubin
\cite{Aubin84}, are defined by
\beq
\label{AubinEnergyEq}
\begin{aligned}
J(\vp)&=J(\o_\vp):=
V^{-1}\int_M\vp\o^n
-
\frac{V^{-1}}{n+1}\int_M
\vp\sum_{l=0}^{n}\o^{n-l}\w\o_{\vp}^{l},
\cr
I(\vp)&=I(\o_\vp):=
V^{-1}\int_M\vp(\o^n-\ovpn).
\end{aligned}
\eeq
Here,
$$
V =\int^{}_{} \ovpn,
$$
is a constant independent of $ \o_\varphi \in\H $.

\begin{exer} {\rm
\label{}
Show that $V $ is $n! $ times the volume of $M $ with respect to the Riemannian metric $g $.
(See the end of the proof of Proposition 2.1 in \cite{ClarkeR} for a solution.)

} \end{exer}

The notation $J(\vp)=J(\o_{\vp})$ (and similarly for $I $)  is justified by the fact that $J(\vp)=J(\vp+c)$
for any $c\in\RR$.
These two functionals, as well as their difference,  are mostly equivalent, in the sense that,
\begin{equation}
\label{IJEq}
\frac1{n^2}(I-J)\le \frac1{n(n+1)}I\le\frac1n J\le I-J\le \frac{n}{n+1} I\le nJ.
\end{equation} 

\bremark
We will be rather sloppy and often say ``$\vp\in\H$" when we really mean
$\ovp\in\H$. However, see Remark \ref{HHoRemark} where we start being more
precise.
\eremark

A closely related functional is the Aubin--Mabuchi functional, introduced by Mabuchi \cite[Theorem 2.3]{Mabuchi87},
\begin{equation}
\label{AMdef}
\h{\rm AM}(\vp):= V^{-1}\int_M\vp\o^n-J(\vp)=
\frac{V^{-1}}{n+1}\sum_{j=0}^{n}\int_M \vp\, \o^j \wedge \o_\vp^{n-j},
\end{equation}

\begin{exer} {\rm
\label{}
 Prove the integration by parts formula
$$
\int  g\i\ddbar f\wedge \alpha^ j\wedge\beta ^ {n  - j -1}=
\int  f\i\ddbar g\wedge \alpha^ j\wedge\beta ^ {n  - j -1},
$$
whenever $\alpha,\beta $ are smooth closed (1, 1)-forms and
$f,g\in C^2(M)$.
Then, show that if $\delta \mapsto  \varphi (\delta) $ denotes a \hbox{$C^1$} curve in $\H, $ (in the sense that $\delta\mapsto \vp(\delta)(x)$ is $C^1$ map for each $x\in M$,
and $\o_{\vp(\delta)}\in\H$ for each $\delta$), 
\begin{equation}
\begin{aligned}
\label{AMVarEq}
\frac{d}{d\delta}\h{\rm AM}(\vp(\delta)) = V^{-1} \int^{}_{} \frac{d}{d\delta}\varphi(\delta)
\o_{\vp(\delta)}^n.
\end{aligned}
\end{equation}
} \end{exer}

Denote by
\beq
\Ent(\nu,\chi)=
\frac{1}{V}\int_M\log\frac{\chi}{\nu}\chi,
\eeq
the entropy of the measure $\chi$ with respect to the measure $\nu$ (where
here $V=\int_M\chi=\int_M\nu$).

The Mabuchi energy (sometimes also called the K-energy as in 
Mabuchi's  original article) 
$$E:\H\ra\RR,$$ 
is defined by
\cite[(5.27)]{R14},\cite{Mabuchi87}, 
\beq
\label{EEq}
E(\ovp)= E(\vp) 
:=
\Ent(e^{f_\o} \o^n,\o^n_\vp)
-\mu  \h{\rm AM}(\vp)
+ \mu\V\int_M \vp\ovp^n.
\eeq
Here, $\fo$ is a smooth function depending on $\omega$ that we define next.
\begin{definition}
\label{foDef}
The Ricci potential of $\o$,
$f_\o$, satisfies
\beq\label{foEq}
\i\ddbar\fo=\Ric\,\o-\mu\o,
\eeq
where it is
convenient to require the normalization $\int e^{\fo}\on=\int \on$.
\end{definition}

\begin{exer} {\rm
\label{IJAMExer}
Show that
\beq
\lb{IJIEq}
\h{\rm AM}(\vp)
=(I-J)(\vp)+\V\int \vp\ovp^n,
\eeq
and therefore the last two terms in \eqref{EEq} equal $-\mu(I-J)(\o,\ovp)$, so
\beq
\label{E2ndEq}
E(\ovp)=
\Ent(e^{f_\o} \o^n,\o^n_\vp)
-\mu (I-J)(\o,\ovp).
\eeq

}\end{exer}

From this formula, we see that understanding the K-energy essentially means understanding the interplay between the entropy and the Aubin functional $I-J$
(or, the equivalent functionals $I$ or $J$, recall  \eqref{IJEq}).
This is in some sense the holy grail, the difficult analytical question at the heart of Tian's conjecture.
A first (and fundamental) result in this direction is Theorem \ref{EntThm} below, however only after much more work do we obtain a clearer picture of this relationship,
culminating in Theorem \ref{KEexistenceIntroThm}.

\bexer
{\rm Show that indeed $E(\ovp)= E(\vp)$, i.e., that
$E(\vp+C)= E(\vp)$ for any $C\in\RR$.
}\eexer

There is another way to write the K-energy:
\begin{equation}
\label{Kendef}
E(\vp):= 
\Ent(\o^n,\o^n_\vp)
+  
s_0 \h{\rm AM}(\vp) 
- 
\frac{1}{V}\sum_{j=0}^{n-1}\int_M \vp \Ric \o \wedge \ovp^{j} \wedge \o^{n-1-j},
\eeq
where $s_0 = V^{-1}\int_M s_\o \o^n$ is the average scalar curvature. 
Of course, 
$$V^{-1}\int_M s_\o \o^n=V^{-1}\int_M n\Rico\w \o^{n-1}=n\mu$$
so $s_0=n\mu$. 
\bexer
{\rm
Show that  \eqref{Kendef} coincides with   \eqref{EEq} when $\mu[\omega] = 2\pi c_1 (M) 
\; (\,=[\Rico])$.
}\eexer

The point of \eqref{Kendef} is that it actually makes sense on any \K manifold. We will however stick to the first formula in these lectures for simplicity.

\subsection {The K-energy when $\mu<0$}

Using Exercise \ref{IJAMExer}, Conjecture \ref{TianConj} is seen to hold in the case $\mu <0$ as follows.
First, convexity of the exponential function implies that
$$
\int \log f d\nu \le \log \int fd\nu,
$$
whenever $d\nu$ is a probability measure (so $\int d\nu =1$), so
\begin{equation}
\label{}
\Ent(\nu,\chi)=
-\int\log\frac{\nu}{\chi}\frac\chi V
\ge
-\log\int\frac{\nu}{\chi}\frac\chi V=0,
\end{equation}
i.e., the entropy is always nonnegative.
Therefore,
\beq
\label{EnegEq}
E(\vp) 
=
\Ent(e^{f_\o} \o^n,\o^n_\vp)
-\mu(I-J)(\vp)\ge
\frac{|\mu|}{n}J(\vp),
\eeq
as desired.

\subsection {The K-energy when $\mu\ge0$}

We now describe a technique, due to Tian \cite[\S7]{Tianbook}, 
to treat some of the cases when $\mu\ge0$ by showing the entropy itself is always proper.
Since
\begin{equation}
\label{EntProperEq}
\Ent(e^{f_\o} \o^n,\o^n_\vp)\ge \Ent(\o^n,\o^n_\vp)-\sup f_\o,
\end{equation}
it suffices to estimate 
$
\Ent(\o^n,\o^n_\vp).
$
Since the functionals $I $ and $J $ are interchangeable as far as properness goes
(recall \eqref{IJEq}), what we would like to show is: 
\begin{theorem}
\label{EntThm}
There exists
a positive $\beta,C $ such that
$$
\Ent(\o^n,\o^n_\vp)\ge
\beta I(\vp)-C=-\beta 
V^{-1}\int_M\Big(\vp-V^{-1}\int_M\vp\on\Big)\ovpn-C.
$$
\end{theorem}
Rewriting the functional $I $ in this way is useful for the following reason:
\begin{equation}
\baeq
\label{}
\beta I(\vp)-
\Ent(\o^n,\o^n_\vp)
&=
\int\log e^{\log\frac{\on}{\ovpn}-\be(\vp-V^{-1}\int_M\vp\on)}\ovpn/V
\cr
&\le
\log\int e^{\log\frac{\on}{\ovpn}}e^{-\be(\vp-V^{-1}\int_M\vp\on)}\ovpn/V
\cr
&=
\log\int e^{-\be(\vp-V^{-1}\int_M\vp\on)}\on/V,
\eaeq
\end{equation}
and so the question reduces to whether there exists a positive $\beta $ such that the functional
$$
\vp\mapsto \int e^{-\be(\vp-V^{-1}\int_M\vp\on)}\on
$$
is uniformly bounded on $\H $. Observe that we have managed to eliminate the dependence on the measure $\ovpn$. To be more precise, the question is now about integrability properties of functions in $\H $ with respect to a {\it fixed} measure.  
We treat this question in the next subsection. 
Before doing so, observe that an affirmative answer implies the K-energy $E$ is proper whenever $\mu = 0 $. When $\mu > 0$, an affirmative answer implies using  \eqref{IJEq},
\begin{equation}
\begin{aligned}
\label{}
E(\vp)\ge (\be - n\mu/(n+1)) I(\vp).
\end{aligned}
\end{equation}
Thus, if $\beta $ can be taken larger than $n\mu/(n+1)$ then the K-energy is proper even when $\mu > 0$.

\subsection {Tian's invariant}

The preceding question is equivalent to the following:

\begin{question}
\label{TianInvQuestion}
Is
\beq\label{SecondDefAlpha}
\alpha(M,[\o])
=
\sup\Big\{\; \be \,:\,
\sup_{\vp\in\H}\int_M
e^{-\be(\vp-V^{-1}\int_M\vp\on)}\o^n< C(\be)
\h{\rm \ for some constant $C(\be)>0$}
\Big\}
\eeq
positive? 

\end{question}

By definition, the number $\alpha(M,[\o])$ is an invariant of the  K\"ahler class $[\omega] $.
It was introduced by Tian, who answered Question \ref{TianInvQuestion} affirmatively
 \cite[Proposition 2.1]{Tian1987}.

\begin{theorem}
\label{TianalphaThm}
$\alpha(M,[\o])>0$.
\end{theorem}
As explained above, Theorem \ref{TianalphaThm} implies Theorem \ref{EntThm}.  

Before going into the detailed proof of this theorem we observe that the last statement of the previous subsection can be stated as follows.

\begin{theorem}
\label{TianbigalphaThm}
Suppose \eqref{c1} holds. The K-energy is proper whenever $\alpha(M,[\o])>n\mu/(n+1)$.
\end{theorem}

Thanks to Theorem \ref{TianalphaThm}, Theorem \ref{TianbigalphaThm} treats in a unified fashion the negative, zero, and some positive cases.

\begin{remark} {\rm
\label{}
Using  \eqref{Kendef} instead of  \eqref{EEq} one may generalize Theorem \ref{TianbigalphaThm} to cohomology classes nearby $c_1(M)/\mu$, as shown recently by Dervan \cite[Theorem 1.3] {Dervan}. 
} \end{remark}

We now turn to proving Theorem  \ref{TianalphaThm}.
The key is an elementary, but by no means trivial, result  on subharmonic functions 
in the plane from H\"ormander's book \cite[Theorem 4.4.5]{H}.
Denote  
the ball of radius $R $ about the origin  in $\CC$ by
$$
B_R(0).
$$

\begin{theorem}
\label{HThm}
Let $R>0$
and let $\psi$ be a smooth subharmonic function defined on 
$B_R(0)\subset\CC$, satisfying
\begin{equation}
\begin{aligned}
\psi(0)& \ge -1,\cr
\psi & \le 0, \h{\ on\ } B_R(0).
\end{aligned}
\end{equation}
Then 
for every $\rho\in[R/2,e^{-1/2}R)$ 
there exists a constant $C$ depending only on $R,\rho$ 
such that
\begin{equation}
\begin{aligned}
\label{Integral}
\int_{B_{\rho}(0)} e^{-\psi}\i dz\w\overline{dz} \le C.
\end{aligned}
\end{equation}
\end{theorem}

\bpf
Let $\tpsi:=\psi+1$ (so that $\tpsi\le1$ and $\tpsi(0)\ge0$).
We prove  \eqref{Integral} for $\tpsi$ which is the same thing as  \eqref{Integral} for 
$\psi$ up to a factor of $e$. 

The Riesz (or Poisson) representation of a smooth function $f:B_R(0)\ra\RR$ takes the form
\begin{equation}
\begin{aligned}
\label{Riesz}
f(z)=\frac1{2\pi}\int_{B_R(0)}\log\Big|\frac{Rz-R\zeta}{R^2-z\bar\zeta}\Big|\Delta f(\zeta)
\frac\i2 d\zeta\w d\bar\zeta
+\int_{0}^{2\pi}\frac{R^2-|z|^2}{|z-Re^{\i\theta}|^2} f(Re^{\i\theta})\frac{d\theta}{2\pi}.
\end{aligned}
\end{equation}

Now we consider   \eqref{Riesz} for $f=\tpsi$ and try to obtain bounds for each of the terms.

{\noindent \it Second term:}
First we show that the second term in  \eqref{Riesz} for $f=\tpsi$ is actually itself uniformly bounded
when $z\in B_{R/2}(0). $ 
Indeed, putting $z=0$ and $f=\tpsi$ in  \eqref{Riesz},
$$
0\le\tpsi(0)
=
\frac1{2\pi}\int_{B_R(0)}\log\frac{|\zeta|}{R}\Delta \tpsi(\zeta)
\frac\i2 d\zeta\w d\bar\zeta
+\int_{0}^{2\pi} \tpsi(Re^{\i\theta})\frac{d\theta}{2\pi}.
$$
Hence
\begin{equation}
\begin{aligned}
\label{Eqtwo}
2\pi\ge 2\pi-2\pi\tpsi(0)
=
\int_{B_R(0)}\log\frac {R}{|\zeta|}\Delta \tpsi(\zeta)
\frac\i2 d\zeta\w d\bar\zeta
+\int_{0}^{2\pi} (1-\tpsi(Re^{\i\theta}))d\theta.
\end{aligned}
\end{equation}
Since $\tpsi\le1$ the second integrand is nonnegative. So is the first, since 
 $\Delta\tpsi\ge0$.
So each of the terms is nonnegative and hence bounded from above by $2\pi$. Therefore,
$$
\int_{0}^{2\pi}|\tpsi(Re^{\i\theta})|d\theta\le
\int_{0}^{2\pi} (1-\tpsi(Re^{\i\theta}))d\theta
+\int_{0}^{2\pi} 1\cdot d\theta\le 2\pi+2\pi=4\pi,
$$
hence
\begin{equation}
\begin{aligned}
\label{6Eq}
\Big|
\int_0^{2\pi}
\frac{R^2-|z|^2}{|z-\zeta|^2} 
\tpsi(Re^{\i\theta})\frac{d\theta}{2\pi}
\Big|
\le
\sup_{z\in B_{r}(0), \,|\zeta|=R}
\frac{R^2-|z|^2}{|z-\zeta|^2}
\int_0^{2\pi}
|\tpsi(Re^{\i\theta})|\frac{d\theta}{2\pi}
\le C(r,R)\cdot\frac1{2\pi}4\pi=6,
\end{aligned}
\end{equation}
where $C(r,R)$ is some constant depending only on $r,R$. 

{\noindent \it First term:}
The first term in  \eqref{Riesz} is not uniformly bounded, however we will show it is uniformly exponentially integrable in the sense of the statement of the theorem.
We split this first term into two parts one of which will be actually uniformly bounded
(all we really need is a uniform bound from below):
\begin{claim}
\label{}
For each $z$ such that $|z|<r<\rho$ 
one has
$$
\bigg|
\frac1{2\pi}\int_{B_R(0)\setminus B_{\rho}(0)}\log\Big|\frac{Rz-R\zeta}{R^2-z\bar\zeta}\Big|
\Delta \tpsi \frac\i2 d\zeta\w d\bar\zeta
\bigg|
\le  C,
$$
for some constant $C> 0 $ depending only on $r,\rho, R$.
\end{claim}

\bpf
Recall that in  \eqref{Eqtwo} each of the terms was bounded by $2\pi $,
so 
$$
\int_{B_R(0)}\log\frac {R}{|\zeta|}\Delta \tpsi(\zeta)
\frac\i2 d\zeta\w d\bar\zeta
\le
2\pi.
$$
Thus, as $\log(1+b)\ge Cb$ for some constant $C=C(\eps)\in (0,1)$ when $b\in(0,\eps)$, 
$$
\int_{B_R(0)\sm B_{R(1-\eps)}(0)}(C\frac{|R-|\zeta||}{|\zeta|})\Delta \tpsi(\zeta)
\frac\i2 d\zeta\w d\bar\zeta
\le 2\pi.
$$
We have,
\begin{equation*}
C\int_{B_R(0)}\frac{|R-|\zeta||}{|\zeta|}\Delta \tpsi(\zeta)
\frac\i2 d\zeta\w d\bar\zeta
\le 2\pi,
\end{equation*}
in particular,
\begin{equation}
\label{1sttotalmassEq}
C\int_{B_R(0)}{|1-|\zeta|/R|}\Delta \tpsi(\zeta)
\frac\i2 d\zeta\w d\bar\zeta
\le 2\pi.
\end{equation}
Now,
$$
\Big|\frac{R^2-z\bar\zeta}{Rz-R\zeta}\Big|\in \del B_1(0), \q\forall \zeta\in\del B_R(0),
$$
as can be checked from the fact that for each $z$ such that $|z|<1$  
the map $\zeta\mapsto \frac{1-z\bar\zeta}{z-\zeta}$ is a M\"obius map, i.e., sends $\del B_1(0)$
to itself and then scaling.
Thus, if $z\in B_r(0)$ with $r<1$, there exists $C>0$ possibly depending on $r,\rho,R$  such that
$$
\bigg|\log\Big|\frac{Rz-R\zeta}{R^2-z\bar\zeta}\Big|\bigg|
\le 
\left\{\aligned
&C \big|1-|\zeta|/R\big| &\qquad \h{for $\zeta\in(R(1-\eps),R)$},\\
&C                       &\qquad \h{for $\zeta\in(\rho,R(1-\eps))$.}
\endaligned
\right.
$$
Then,
$$
\baeq
\bigg|
\frac1{2\pi}
\int_{B_R(0)\setminus B_{\rho}(0)}
\log
&
\Big|\frac{Rz-R\zeta}{R^2-z\bar\zeta}\Big|
\Delta \tpsi \frac\i2 d\zeta\w d\bar\zeta
\bigg|
\cr
& 
\le
C
\frac1{2\pi}\int_{B_R(0)\setminus B_{R(1-\eps)}(0)}
\big|1-|\zeta|/R\big|
\Delta \tpsi \frac\i2 d\zeta\w d\bar\zeta
\cr
&\quad
+
C\frac1{2\pi}\int_{B_{R(1-\eps)}(0)\setminus B_\rho(0)}
\Delta \tpsi \frac\i2 d\zeta\w d\bar\zeta.
\eaeq
$$
The first term on the right hand side is uniformly bounded by  \eqref{1sttotalmassEq}.
The second term is uniformly bounded by Claim \ref{harmonicmassClaim}  below.
\epf

Since $\rho<e^{-1/2}R$ from now and on we write
$$\rho=e^{-1/2-\eps}R,
$$
with $\epsilon > 0 $ small, say $\eps=1/500$, where
$500$ could be replaced by a generous quantity
(cf. \cite[Proposition 8.1]{Bamler}).
In order to estimate $e^{-\tpsi}$ we only need to estimate for each $z$ such that
$|z|<r$ 
the exponential term (note the minus sign)
\begin{equation}
\begin{aligned}
\label{ExponEq}
\exp\Big(   
-\frac1{2\pi}
\int_{B_{e^{-1/2-\eps}R}(0)}\log\Big|\frac{Rz-R\zeta}{R^2-z\bar\zeta}\Big|
\Delta \tpsi \frac\i2 d\zeta\w d\bar\zeta
\Big).
\end{aligned}
\end{equation}
This can be done using Jensen's inequality, or just the arithmetic mean-geometric mean inequality.
For that need to normalize the meausure so that it integrates to 1 (i.e., becomes a probability
measure).
That is need to divide by 
$$
a:=\frac1{2\pi}\int_{B_{e^{-1/2-\eps}R}(0)} \Delta \tpsi \frac\i2 d\zeta\w d\bar\zeta,
$$
 i.e., the mass of the harmonic measure on this ball. It is well-known that the mass of the harmonic measure on a compact subset of a ball is uniformly bounded by a constant depending on the distance to the boundary of the ball whenever the function is uniformly bounded from above on the whole ball and its value is fixed at one point. Moreover,
$$
a<2\; !
$$
Indeed, this is the reason to choose $\rho=e^{-1/2-\epsilon}R$:
$$
a=\frac1{2\pi}\int_{B_{\rho}(0)} \Delta \tpsi \frac\i2 d\zeta\w d\bar\zeta
\le
\frac1{2\pi}
\int_{B_{\rho}(0)}\frac{\Big(2\log\frac {R}{|\zeta|}\Big)}
{1+2\epsilon} \Delta \tpsi \frac\i2 d\zeta\w d\bar\zeta
\le \frac2{1+2\epsilon}\cdot \frac{2\pi}{2\pi}=\frac2{1+2\epsilon},
$$
since earlier we proved the first term of   \eqref{Eqtwo} is bounded by $2\pi$
(note all we did was insert a term between the large parenthesis which is bigger than 1).

For an earlier reference we state the following claim whose proof is identical to the computation of $a $.
\begin{claim}
\label{harmonicmassClaim}
For every $\epsilon\in (0, 1) $ there is a constant $C = C (\epsilon) $
such that 
\begin{equation*}
\begin{aligned}
\label{}
\frac1{2\pi}\int_{B_{R(1-\eps)}(0)} \Delta \tpsi \frac\i2 d\zeta\w d\bar\zeta\le C.
\end{aligned}
\end{equation*}  
\end{claim}

\begin{exer} {\rm
\label{}
Compute the constant $C (\epsilon) $ in the previous claim.
} \end{exer}

So we come back to   \eqref{ExponEq}, and apply the arithmetic mean-geometric mean inequality:

$$
\baeq
\hbox{ \eqref{ExponEq}} & =
\exp\Big(   
\int_{B_{e^{-1/2-\eps}R}(0)} -a\cdot\log\Big|\frac{Rz-R\zeta}{R^2-z\bar\zeta}\Big|
\Delta \tpsi \frac\i2 d\zeta\w d\bar\zeta/(2\pi a)
\Big)
\cr
& \le
\int_{B_{e^{-1/2-\eps}R}(0)} \Big|\frac{R^2-z\bar\zeta}{Rz-R\zeta}\Big|^a
\Delta \tpsi \frac\i2 d\zeta\w d\bar\zeta/(2\pi a)
\cr
& \le
C
\int_{B_{e^{-1/2-\eps}R}(0)} 
\frac1{|z-\zeta|^a}\Delta \tpsi \frac\i2 d\zeta\w d\bar\zeta/(2\pi a).
\eaeq
$$

Now this itself may not be bounded, however, it is in $L^1$ in $z$---more precisely in
$L^1(B_r(0))$---and this is what we want to show. It is crucial here that $a<2$ or
in other words to chose $\rho=e^{-1/2-\eps}R$ earlier.
To be precise, we integrate now in $z$ to get 
$$
\baeq
&
\int_{B_{1/2R}(0)} 
\int_{B_{e^{-1/2-\eps}R}(0)} 
\frac1{|z-\zeta|^a}\Delta \tpsi(\zeta) \frac\i2 d\zeta\w d\bar\zeta/(2\pi a)
\w \frac\i2 dz\w d\bar z/2\pi
\cr
&\le
\int_{B_{(1/2+e^{-1/2-\eps})R}(0)} 
\int_{B_{e^{-1/2-\eps}R}(0)} 
\frac1{|\xi|^a}\Delta \tpsi(\zeta) 
\frac\i2 d\zeta\w d\bar\zeta/(2\pi a)
\w \frac\i2 d\xi\w d\bar\xi/2\pi
\cr
&=\frac{2\pi}{2-a}[(1/2+e^{-1/2-\eps})R]^{2-a}
\int_{B_{e^{-1/2-\eps}R}(0)}\Delta \tpsi(\zeta) 
\frac\i2 d\zeta\w d\bar\zeta/(2\pi a)
\cr
&\le\frac{2\pi}{2-a}[(1/2+e^{-1/2-\eps})R]^{2-a}=C(\eps).
\eaeq
$$
Note Fubini's theorem applies thanks to the last estimate, so the change
of order of integration is justified, and so the original integral is 
bounded, concluding the proof of Theorem \ref{HThm}.
\epf

\begin{exer} {\rm
\label{6Exer}
Show that the fraction in   \eqref{6Eq} is bounded above by a constant depending only on  $r/R$ as claimed
and blows up as $r $ approaches 0. 
When $r = R/2 $ show that this constant is equal to $3$. 
(it is even simpler to see it must be $\le4$).
} \end{exer}

\begin{exer} {\rm
\label{}
Compute the Green kernel of $B_R(0)$ and then 
derive the Riesz representation formula  \eqref{Riesz}  starting with the identity (cf. \cite[\S2.4]{GT}, \cite{Wanby})
$$
f(x)=-\int_{B_R(0)} G(x,y)\Delta f(y) dy+\int_{\del B_R(0)} \del_r G(x,y)f(y)dy.
$$
} \end{exer}

\begin{exer} {\rm
\label{}
Show that Theorem  \ref{HThm} holds for any subharmonic function $f $ by using the fact that the Riesz representation   \eqref{Riesz} holds with the same expression by interpreting 
$\Delta f(\zeta) \frac\i2 d\zeta\w d\bar\zeta$
as the harmonic measure associated to  $f$ (a positive measure with respect to which the Green kernel is integrable) \cite[Theorem 4.5.1]{Ransford}.  
} \end{exer}

Theorem \ref{HThm} can be extended to any dimension as follows  \cite[Theorem 4.4.5]{H}.

\begin{corollary}  
\label{HorCor}
The same result holds in $\CC^n$ with constants that might additionally depend on $n$. In other words, if $\psi$ is a smooth plurisubharmonic function on $B_R(0)\subset\CC^n$, satisfying
\begin{equation}
\begin{aligned}
\psi(0)& \ge -1,\cr
\psi & \le 0, \h{\ on\ } B_R(0),
\end{aligned}
\end{equation}
then 
for every $\rho\in[R/2,e^{-1/2}R)$ 
there exists a constant $C$ depending only on $R,\rho,n$ 
such that
\begin{equation}
\begin{aligned}
\label{Integral}
\int_{B_{\rho}(0)} e^{-\psi}\i dz_1\w\overline{dz_1}\w\cdots\w\i dz_n\w\overline{dz_n} \le C.
\end{aligned}
\end{equation}
\end{corollary}

\def\la{\lambda}

\bremark
In both Theorem \ref{HThm} and Corollary \ref{HorCor} one may drop the smoothness assumption on $\psi$: indeed convolve $\psi$ with a smooth mollifier and 
then observe the integrals on the left had side converge in the limit.
\eremark

\bpf
Write
\begin{equation}
\begin{aligned}
\label{HormSperical}
\int_{B_r(0)\sseq\CCfoot^n} e^{-\psi}
=
\int_{\del B_1(0)\sseq \CCfoot^n} dV_{S_r^{2n-1}}(\la)
\int_{B_r(0)\sseq\CCfoot} |w|^{2n-2}e^{-\psi(\la w)}\frac\i2 dw\w d\bar w/2\pi,
\end{aligned}
\end{equation}
and $|w|^{2n-2}\le r^{2n-2}$ is  uniformly bounded, so we can apply the previous result for
$n=1$. To obtain   \eqref{HormSperical}, note that we are integrating over a $2n+1$ dimensional manifold
on the right-hand side and on a $2n$-dimensional one on the
left-hand side. We normalize by $2\pi$,
the area of $S^1$ since each point is counted ``$S^1$ times", since if wish to write
$z=\la w$ with $w\in\CC$, $\la\in S^{2n-1}$, and $|z|=|w|, \, |\lambda|=1$ then $w$ is only determined in $\CC$ up to multiplication by a number in $S^1$! This is not quite precise since
we should normalize by an $S^1$ of varying radius! So need to actually divide by $2\pi|w|$, however
the change of variables introduces a factor $|w|^{2n-1}$; 
so we get  \eqref{HormSperical}.
\epf

\bpf[Proof of Theorem \ref{TianalphaThm}] Let the injectivity radius of $(M,\o)$ be $6r$ (so at each point exists a geodesic ball
of that radius). Choose an $r$-net of $M$, that is a colle
ction of points $\{x_j\}_{j=1}^N$ such
that $M=\cup_j B_r(x_j)$. For each $\vp\in\H$ one has $n+\Delta_\o\vp>0$. 
Hence Green's formula  says that \cite[Theorem 4.13 (a), p. 108]{Aubinbook}
\begin{equation}
\begin{aligned}
\label{MeanVal}
\vp(x)=\V\intm\vp\on+\V\int-G(x,y)\Delta_\o\vp(y)\on(y)\le \V\intm\vp\on+nA_\o,
\end{aligned}
\end{equation}
where $G(x,y)\ge-A_\o$ and $\intm G(x,y)\on(y)=0$ for each $x\in M$.
Note that $A_\o$ is a constant depending only on $(M,\omega) $, in other words, the Green kernel is uniformly bounded from below \cite[Theorem 4.13 (d), p. 108]{Aubinbook}. 
Since we are also assuming $\sup\vp=0$, we obtain (the RHS of   \eqref{MeanVal} is independent of $x$)
$$
\V\intm\vp\on+nA_\o\ge0.
$$
Hence, since $\vp$ is non-positive,
$$
\int_{B_{r}(x_j)}\vp\on \ge \intm\vp\on\ge-nA_\o V,
$$
that is,
\begin{equation}
\begin{aligned}
\label{MinValBall}
\sup_{B_{r}(x_j)}\vp\ge\frac{-nA_\o V}{\vol({B_{r}(x_j)}}.
\end{aligned}
\end{equation}
Choose a local \K potential $\psi_j$ satisfying $\i\ddbar\psi_j=\o|_{B_{5r}(x_j)}$.
Let $C_2:=\sup_j\sup_{B_{5r}(x_j)}\psi_j$. Therefore, since
$\vp\le0$,
$$
\psi_j+\vp\le C_2, \h{\ on \ } {B_{5r}(x_j)}.
$$
Also, let $y_j\in {B_{r}(x_j)}$ be such that (using   \eqref{MinValBall})
$$
\vp(y_j) \ge\frac{-nA_\o V}{\vol_\o({B_{r}(x_j))}}.
$$
We may also assume without loss of generality that $\psi_j(y_j)=0$, otherwise we add a constant
to $\psi_j$ (and then $C_2$ possibly increases).

Now the function $\vp+\psi_j$ is plurisubharmonic (psh for short) on $B_{5r}(y_j)$ (recall that $\vp$ is not psh only $\o$-psh so we
add to it a local potential for $\o$ in order to be able to apply H\"ormander's result).
Also, since $B_{4r}(y_j)\sseq B_{5r}(x_j)$, we have that on the set $B_{4r}(y_j)$
it holds
$$
\baeq
\vp+\psi_j-C_2 & \le 0\cr
(\vp+\psi_j-C_2)(y_j) & \ge -C_2-\frac{nA_\o V}{\vol({B_{r}(x_j)}}.
\eaeq
$$
Therefore can apply H\"ormander's result to 
$f_j:=(\vp+\psi_j-C_2)/(C_2+\frac{nA_\o V}{\vol({B_{r}(x_j)}})$
on $B_{4r}(y_j)$
(note $C_2\ge0$, $A_\o>0$, so we are not dividing by zero), 
namely obtain that
$$
\int_{B_{2r}(y_j)} e^{-f_j}\;\on\big|_{B_{2r}(y_j)}<C_j
$$
(instead of $2r$ could have taken any number in the range $[2r,e^{-1/2}4r)$).
Patching these up, using the fact that $B_{2r}(y_j)\supseteq B_r(x_j)$ and $M$ 
is covered by the latter, we obtain that
regardless of $\vp$, one has
$$
\intM e^{-a\vp}\on< C,
$$
where $a:=\min_j 1/(C_2+\frac{nA_\o V}{\vol({B_{r}(x_j)}})$, and
consequently $\alpha(M,[\o])\ge a>0$.
\epf


\section {The  K\"ahler--Einstein  equation}

The K\"ahler--Einstein equation  \eqref{KE1Eq} is a fourth order equation in terms of the  K\"ahler potential.
The remarkable formula  \eqref{RicciFormEq} for the Ricci form, however, allows to integrate it to a second order equation. Indeed, subtracting $\Rico$ from both sides of the equation and using  \eqref{RicdiffEq} yields
$$
\Ric\,\ovp-\Rico=
{\i}\ddbar\log\frac{\on}{\ovpn}
=\mu\ovp-\Rico
=\mu\i\ddbar\vp-\i\ddbar f_\o,
$$
where $\fo$ is called the Ricci potential of $\o$,
and satisfies $\i\ddbar\fo=\Ric\,\o-\mu\o$, where it is
convenient to require the normalization $\int e^{\fo}\on=\int \on$.
We thus obtain the \KE equation,
\begin{equation}
\label{KEEq}
\ovpn=\on e^{\fo-\mu\vp}, \quad \h{ on } M
\end{equation}
for a global smooth function $\vp$ (called the \K potential of $\ovp$ relative to $\o$).
The function $\fo$, in turn, is given in terms of the reference geometry and is thus known.
Observe that, strictly speaking, the right-hand side of   \eqref{KEEq} should be
$$
\on e^{\fo-\mu\vp+C}
$$
For some constant $C  $; whenever $\mu\not = 0 $ we can incorporate the constant $C $ by subtracting $C/\mu $ from $ \varphi $ since the left-hand side of  \eqref{KEEq} is invariant under this. When $\mu = 0 $, the constant $C $ must be zero by  \eqref{foEq}. 

Note also that the solution $ \varphi $ is only determined up to a constant when $\mu = 0 $, while it is uniquely determined when $\mu\not = 0 $ by  \eqref{foEq}.

We close this section by noting a relationship between the K-energy and  K\"ahler--Einstein  metrics: the Euler--Lagrange equation of the K-energy is precisely the  K\"ahler--Einstein  equation. Indeed,
\beq
\lb{EderivEq}
\baeq
\frac{d}{d\eps}\Big|_{\eps=0}E(\vp(\eps)) 
&= 
\frac{d}{d \eps}\Big|_{\eps=0} \Ent(e^{f_\o} \o^n,\o_{\vp(\eps)}^n)
\cr
&\q-\mu \V\int_M \dot\vp\o_{\vp}^n
+ \mu\V\int_M \dot\vp\o_{\vp}^n
+ \mu\V\int_M \vp\Delta\dot\vp\o_{\vp}^n
\cr
&=
\V\int_M \big(\Delta_{\ovp}\dot\vp+\log\frac{\o_{\vp}^n}{e^{f_\o} \o^n} \Delta\dot\vp\big)\o_{\vp}^n
+ \mu\V\int_M \vp\Delta\dot\vp\o_{\vp}^n
\cr
&=
\V\int_M \big(\log\frac{\o_{\vp}^n}{e^{f_\o} \o^n} \Delta\dot\vp\big)\o_{\vp}^n
+ \mu\V\int_M \vp\Delta\dot\vp\o_{\vp}^n
\cr
&=
-\V\int_M \dot\vp\Delta_\ovp f_\ovp \o_{\vp}^n,
\eaeq
\eeq
where we used Exercise \ref{DetExer} and the following exercise.

\begin{exer} {\rm
\label{}
Show that the Ricci potential satisfies the following equation
\begin{equation}
\begin{aligned}
\label{RicciPotEq}
f_{\ovp}=\log\frac{e^{f_\o} \o^n}{\o_{\vp}^n}-\mu\vp-\log\V\int_M e^{f_\o-\mu\vp} \o^n.
\end{aligned}
\end{equation}
} \end{exer}
Thus, the Euler--Lagrange equation of the K-energy is precisely
$$
\Delta_\ovp f_\ovp=0,
$$
i.e., 
$$
f_\ovp= \h{const},
$$
that, recalling Definition \ref{foDef}, means that  
$\Ric\,\ovp=\mu\ovp.$

\section {Properness implies existence}
\lb{PropExistSec}

In this section we prove the easier part of Conjecture  \ref{TianConj}:

\begin{theorem}
\label{propernessexistence}
If the Mabuchi energy $E$ is proper on $\H^K$ then there exists  a $K$-invariant K\"ahler--Einstein metric in $\H$.
\end{theorem}

This result is due to Tian \cite{Tian97}, even though it is in some sense already implicit in Ding--Tian  \cite{dt}. The proof we give follows the same ideas as in the original proof, with some simplifications in the presentation.

In particular, in combination with Theorem \ref{TianbigalphaThm}, we obtain as a corollary 
a theorem of Tian \cite[Theorem 2.1]{Tian1987}:

\begin{corollary}
\label{TianbigalphaKEThm}
Let  $\mu > 0 $ and suppose that  \eqref{c1} holds. 
Whenever $\alpha(M,2\pi c_1(M)/\mu)>n\mu/(n+1)$ there exists K\"ahler--Einstein metric
cohomologous to $2\pi c_1(M)/\mu$. 
\end{corollary}

We also  obtain as a corollary the theorems of Aubin and Yau:

\begin{corollary}
\label{AubinYauCor}
Let  $\mu \le 0 $ and suppose that  \eqref{c1} holds. Then there exists a unique K\"ahler--Einstein metric
cohomologous to $\o$. 
\end{corollary}

\begin{remark} {\rm
\label{}
The proof of Corollary  \ref{AubinYauCor} we will give will not directly use the fact that the K-energy is proper whenever $\mu \le 0 $ (which holds according to Theorem \ref{TianbigalphaThm}).
In fact, the proof of Corollary  \ref{AubinYauCor} will be more or less a step in the proof of Theorem \ref{propernessexistence}. 
} \end{remark}

\subsection {A two-parameter continuity method}

We will give a somewhat nonstandard proof of Theorem  \ref{propernessexistence} using a two-parameter continuity method instead of the more standard proofs that use one-parameter continuity methods or the Ricci flow equation. Namely, we consider the two-parameter family of equations,
\begin{equation}
\begin{aligned}
\label{TwoParamCMEq}
\o_{\vp}^n=e^{tf_\o+c_t-s\vp}\on,
\quad
c_t:=-\log\frac1V\int_M e^{tf_\omega}\omega^n, \quad (s,t)\in A,
\end{aligned}
\end{equation}
where 
$$
A:=(-\infty,0]\times[0,1]\;\cup [0,\mu]\times\{1\}
$$
is the parameter set---the union of a semi-infinite rectangle and
 an interval,
and show that there exists a unique solution $\vp(s,t)$ for each $(s,t)\in A$
once we require that the solution be continuous in the parameters  $s, t$. 
Of course, we are really looking to show the existence of $ \varphi (\mu, 1) $. Thus, the strategy is to first construct $ \varphi (s, t) $ for other values of  $(s, t) $ for which existence is easier to show and then perturbing the equation and eventually arriving at the equation for the values $(\mu, 1) $. Hence, the name `continuity method.'

To show existence for the sub-rectrangle $(-\infty,0]\times[0,1]$ is easier and is precisely 
what proves Corollary  \ref{AubinYauCor}. Also, one could restrict to the sub-rectangle
$(-S,\mu]\times[0,1]$ for any value $S> 0$ as far as proving the existence theorems is concerned. Working on $A$ requires no more work and is somewhat more natural and canonical, since the value  $s = -\infty  $ corresponds, in a sense that can be made precise \cite[Proposition 7.3]{Wu},\cite[\S9]{JMR},\cite{Berman-zerotemp}, to the initial reference metric $\omega$. Thus, one can view this continuity method as starting with the given reference metric and deforming it to the  K\"ahler--Einstein  metric. In fact, the one-parameter continuity method with  $t = 1 $
fixed and $s$ varying between $-  \infty $ and $\mu $ can be viewed as the continuity method analogue of the K\"ahler--Ricci  flow, and is called the Ricci continuity method, introduced in \cite{R08} and further developed in \cite{JMR}. One of the reasons we work with the two-parameter family in these lectures is that then the existence of solutions for some parameter values is automatic. 
Indeed,
\begin{equation}
\begin{aligned}
\label{s0solEq}
 \varphi (s, 0) = 0, \qquad s\in (- \infty ,0].
\end{aligned}
\end{equation}
Working with the one-parameter Ricci continuity method is harder since it is nontrivial to show the existence of solutions for some parameter value  $(s, 1) .$ The full strength of the Ricci continuity method however goes beyond that of the two-parameter family in that the former can be used to show existence of K\"ahler--Einstein edge metrics, which are a natural generalization of K\"ahler--Einstein metrics that allows for a singularity along a complex submanifold of codimension one. In that context the two-parameter continuity method does not always seem to work.

\begin{exer} {\rm
\label{Ricovpst}
Show that for each $(s,t)\in A$,
\begin{equation}
\begin{aligned}
\label{RicovpstEq}
\Ric\o_{\vp(s,t)}=(1-t)\Ric\o+s\o_{\vp(s,t)}+(\mu t-s)\o.
\end{aligned}
\end{equation}
Note that this implies that if indeed, as claimed,  $\vp(s,1)\ra 0$  as $s\ra - \infty$
then $f_\o-s\vp$ should be small, i.e.,  $ \varphi \approx f_\o/s$ in that regime.   
} \end{exer}

\subsection {Openness}
\lb{OpenSubSec}

Let
$$
\PSH(M,\o)=\{ \vp \in L^1(M,\o^n)\,:\, \h{$\vp$ is upper semicontinuous and } 
\ovp\geq 0 \}
$$
denote the set of $\o$-plurisubharmonic functions on $M$.
Define $M_{s,t}:C^{2,\gamma}\cap\PSH(M,\o)\ra C^{0,\gamma}$ by 
$$
M_{s,t}(\vp):= \log \frac{\o_{\vp}^n}{\on}-tf_\o+s\vp-c_t,
\quad (s,t)\in A.
$$
If $\vp(s,t)\in C^{2,\gamma}\cap\PSH(M,\o)$ is a solution of \eqref{TwoParamCMEq}, 
we claim that its linearization 
\begin{equation}
\left. DM_{s,t}\right|_{\vp(s,t)} =\Delta_{\vp(s,t)}+s :C^{2,\gamma}\ra C^{0,\gamma},
\quad (s,t)\in A,
\label{oxf}
\end{equation}
is an isomorphism when $s \neq 0$ and  $s< \mu$ . If $s = 0$, this map is an isomorphism if we restrict on each side 
to the codimension one subspace of functions with integral equal to $0$ with respect to $\o^n_{\vp(0,t)}$.  Furthermore, we also 
claim that $ C^{2,\gamma}\cap\PSH(M,\o)\times A \ni (\vp,s,t) \mapsto M_{s,t}(\vp) \in 
C^{0,\gamma}$ is a $C^1$ mapping. Given these claims, the Implicit Function Theorem then 
guarantees the existence of a solution $\vp(\tilde s,\tilde t) \in C^{2,\gamma}$ 
for all $(\tilde s,\tilde t)\in A$ sufficiently close to $(s,t)$. This solution must necessarily be contained in $\PSH(M,\o)$ since $M_{s,t}(\vp (s, t)) = 0 $ means that  \eqref{TwoParamCMEq} holds, so in particular 
$\o^n_{\vp(s,t)}>0$, and by the continuity of $\o_{\vp(s,t)}$ in the parameters it follows that all the eigenvalues of the metric stay positive along the deformation.

We concentrate on the first claim, since the second claim is easy to check.

Now, $DM_{s,t}$ is an elliptic operator and there is a classical and well-developed theory for those kind of operators acting on H\"older spaces \cite{GT}. 
In particular, such an operator has a generalized inverse, or Green kernel.
Also, it is Fredholm of index 0. 
Using the existence of a Green kernel shows that $C^{2,\gamma}$ decomposes as a direct sum $\{Gf\,:\, f\in C^{0,\gamma}\}\oplus K_{s,t}$, where $K_{s, t} $ denotes the kernel of the operator $DM_{s,t}$. 
Thus, whenever $K_{s,t}=\{0\}$ the operator is an isomorphism.

The nullspace of $DM_{s,t}$ is clearly trivial when $s < 0$ since the spectrum of 
$\Delta_{\vp(s,t)}$ is contained in $(-  \infty , 0]$. When $s=0$ the nullspace consists of the constants and so is an isomorphism when restricted to functions of zero average. Thus, the claim is verified whenever $s\le 0$. To deal with the case when $s$  is positive the following lemma is needed.

\begin{lemma}
\label{PoincareLemma}
Suppose that $M_{s,1}(\vp (s, 1)) = 0 $ and that $s\in(0,\mu)$. 
Then the spectrum of $\Delta_{\vp(s,1)} $ is contained in  
$(-  \infty , -s)$.
\end{lemma}

\bpf
 Let $\psi$ be an eigenfunction of $\Delta_{\o_{\vp(s,t)}}$ with eigenvalue $-\lambda_1$.
By standard theory, $\psi $ is smooth.
The Bochner--Weitzenb\"ock formula states that
\[
\frac12\Delta_g|\nabla_g f|_g^2 = \Ric(\nabla_g f,\nabla_g f)+|\nabla^2 f|^2_g+\nabla f \cdot \nabla (\Delta_g f).
\]
Since $\Delta_g=2\Delta_\o$ and $|\nabla^2 f|^2_g=2|\nabla^{1,0}\nabla^{1,0}f|^2+2(\Delta_\o f)^2$, this becomes
\begin{equation}
\label{WeitzCxEq}
\Delta_\o|\nabla^{1,0} \psi|_g^2 = 2\Ric(\nabla^{1,0} \psi,\nabla^{0,1} \psi) + 2|\nabla^{1,0}\nabla^{1,0}\psi|^2
+ 2\lambda_1^2\psi^2 - 4\lambda_1|\nabla^{1,0} \psi|_\o^2.
\end{equation}
Integrating (\ref{WeitzCxEq}) and using that $\Ric\o(s)>s\o(s)$ when $s<\mu$ by   \eqref{RicovpstEq},
\begin{equation*}
\begin{aligned}
\label{}
\int \Big((2s- 4\lambda_1)|\nabla^{1,0} \psi|_\o^2
+ 2\lambda_1^2\psi^2\Big)\on <0.
\end{aligned}
\end{equation*}
Now,
\begin{equation*}
\begin{aligned}
\label{}
\int 2\lambda_1^2\psi^2\on=
-\int 2\lambda_1\psi\Delta_\o\psi\on=
\int 2\lambda_1|\nabla^{1,0} \psi|_\o^2.
\end{aligned}
\end{equation*}
Thus,
\begin{equation*}
\begin{aligned}
\label{}
\int (2s- 2\lambda_1)|\nabla^{1,0} \psi|_\o^2\on <0,
\end{aligned}
\end{equation*}
so we see that $\lambda_1>s$.
\epf

\begin{remark} {\rm
\label{}
Here we see why we cannot use the the rectangle
$$
(-\infty,\mu]\times[0,1]
$$
containing $A$: we run into trouble with openness. If we had chosen $\o$ to have nonnegative Ricci curvature we could have also worked on the larger trapezoid
$$
(-\infty,0]\times[0,1]\cup\{(s,t)\in[0,\mu]\times[0,1]\,:\, \mu t\ge s\}.
$$
Producing such an $\o$ is possible by applying Corollary  \ref{AubinYauCor} with $\mu=0$ (whose proof does not require these arguments).
} \end{remark}

\subsection {An $L ^\infty$ bound in the sub-rectangle }
\lb{LinfsubrectSubSec}

The following two lemmas will be sufficient for our purposes.
\begin{lemma}
\label{NegCzeroLemma}
Suppose that $ \varphi (s, t) $ is a solution of  \eqref{TwoParamCMEq}. Then whenever $s <0 $,
$$
 \varphi (s, t)<C(1+1/|s|),
$$
for some uniform constant $C $ independent of $s $ and  $t $.  
\end{lemma}

\begin{proof}
Let $p $ be a point where the maximum of $ \varphi $ is achieved. Then,   $ \i\ddbar  \varphi(p) \le 0 $. Thus,
$$
\frac{\ovpn}{\on}(p)\le 1, 
$$  
 i.e., 
$$
tf_\o(p)+c_t-s\vp(s,t)(p)\le 0,
$$
so,
$$
\max\vp(s,t)=\vp(s,t)(p)\le(- c_t - t\min f_\omega)/|s|.
$$
Similarly, if $q $ is a point where the minimum is achieved,
$$
\min\vp(s,t)=\vp(s,t)(q)\ge(- c_t - t\max f_\omega)/|s|.
$$
concluding the proof.
\end{proof}

\begin{lemma}
\label{CYLemma}
Suppose that $ \varphi (0, t) $ is a solution of  \eqref{TwoParamCMEq}. Then,
$$
 \max_M|\varphi (0, t)|<C,
$$
for some uniform constant $C $ independent of $t $.  
\end{lemma}

\begin{proof}
As remarked earlier, the solution in this case is a priori only unique up to a constant. 
However, we fixed the normalization by requiring that
the solution be continuous in the parameters  $s, t$. 
We will eventually show that there are solutions $ \varphi (s, t) $ for all $s $ less than 0
and all $t\in[0,1]$. Therefore, $\vp(s,t)$ converges pointwise to the solution  $ \varphi (0, t) $  for each fixed $t $ as $s$ tends to 0, and so this solution is actually unique. In particular, since the latter change sign, so must the former. Thus, it is enough to estimate the oscillation of $ \varphi (0, t) $  in order to estimate the $L^\infty$ norm of $ \varphi (0, t) $ , i.e., it suffices to estimate the minimum of
$$
\varphi (0, t) - \max\varphi (0, t)-1.
$$ 
This bound, due to Yau \cite{Yau1978}, then follows just as in  \cite[\S5]{Tianbook}.
\end{proof}

\subsection {An $L ^\infty$ bound in the interval }
\lb{LinfintSubSec}

\begin{lemma}
\label{}
Let $t=1$. The K-energy is monotonically decreasing in $s $. 
\end{lemma}

\begin{proof}
When $t = 1 $, $c_t = 0 $. Then, 
$$
\baeq
\frac{d}{d s} E ( \varphi (s, 1)) 
&= 
\frac{d}{d s} \Ent(e^{f_\o} \o^n,e^{f_\o-s\vp} \o^n)
\cr
&\q-\mu \V\int_M \dot\vp\ovp^n
+ \mu\V\int_M \dot\vp\ovp^n
+ \mu\V\int_M \vp\Delta\dot\vp\ovp^n
\cr
&=
\V\int_M (-\vp-s\dot\vp-s\vp\Delta\dot\vp)\ovp^n
+ \mu\V\int_M \vp\Delta\dot\vp\ovp^n.
\eaeq
$$
Differentiating \eqref{TwoParamCMEq} yields

\begin{equation}
\begin{aligned}
\label{TwoParamVarCMEq}
\varphi  (s, t) = - (\Delta + s)\dot  \varphi(s, t).
\end{aligned}
\end{equation}
 Thus,
$$
\frac{d}{d s} E ( \varphi (s, 1)) 
=
\V\int_M (\Delta\dot\vp-(\mu-s)\Delta\dot\vp(\Delta+s)\dot\vp)\ovp^n
=
-(\mu-s)\V\int_M \dot\vp\Delta(\Delta+s)\dot\vp\ovp^n<0,
$$
since $\dot  \varphi $ is not constant as can be seen from   \eqref{TwoParamVarCMEq} and  \eqref{TwoParamCMEq}, while $\Delta + s $ is a negative operator for $s<\mu$ thanks to Lemma \ref{PoincareLemma}. 
\end{proof}

By the previous lemma, the K-energy actually decreases along the interval. By properness this implies that
the functional $I - J $ stays uniformly bounded along the interval once we know $ \varphi (0, 1) $ exists.
This will indeed be the case as we will show (in the first step of the proof) existence in the sub-rectangle for all values $s\le 0$. 
Now, the explicit formula  \eqref{EEq} for the K-energy hence implies that the entropy is bounded from above along the interval,
\begin{equation}
\begin{aligned}
\label{}
\Ent(e^{f_\o} \o^n,\o_{\vp(s,1)}^n)<C,
\end{aligned}
\end{equation}
thus,
\begin{equation}
\begin{aligned}
\label{}
\int \big((t-1)f_\o-s\vp(s,1)+c_t\big) \ovpn < C,
\end{aligned}
\end{equation}
or
\begin{equation}
\begin{aligned}
\label{}
\int -\vp(s,1)\o_{\vp(s,1)}^n < C(1+1/s).
\end{aligned}
\end{equation}
Observe that here we may assume that $s>\epsilon> $0 since by openness about the value $(0, 1) $ we have existence for small positive values of $s $.
Going back to   \eqref{AubinEnergyEq} now shows that
\begin{equation}
\begin{aligned}
\label{AverageBndEq}
\int \vp(s, 1)\on < C(1+1/s),
\end{aligned}
\end{equation}
the mean value inequality  \eqref{MeanVal} shows that
\begin{equation}
\begin{aligned}
\label{supBoundEq}
\max \varphi (s, 1)< C(1+1/s).
\end{aligned}
\end{equation}
It remains to estimate $\min \varphi (s, 1) $. 
Now, as in the proof of Lemma \ref{CYLemma}, we set 
$$
\alpha(s):=\max\varphi (s, 1)-\varphi (s, 1)  +1.
$$ 
A standard Moser iteration argument   now yields an estimate \cite[\S5]{Tianbook}
$$
||\alpha(s)||_{L^\infty(M,\o_{\vp(s,1)}^n)}
\le C\big(||\alpha(s)||_{L^1(M,\o_{\vp(s,1)}^n)}\big), 
$$
but
$$
||\alpha(s)||_{L^1(M,\o_{\vp(s,1)}^n)}=\max\varphi (s, 1)+1+\int -\vp(s,1)\o_{\vp(s,1)}^n<C(1+1/s).
$$
Thus, we have proven the following.
\begin{lemma}
\label{CzeroposLemma}
Suppose that $ \varphi (s, 1) $ exists for $s\in(0,\eps]$
for some $\epsilon> 0$. Then for every $s >\epsilon$, 
$$
\max_M|\varphi (s, 1)| < C(1+1/s).
$$
\end{lemma}

\subsection {Second order estimates}

\def\D{\Delta}
\def\tr{\h{tr}}

The reference for this subsection is  \cite[\S7.2--7.4,7.7]{R14}. 

We say that $\o,\ovp$ are uniformly equivalent
if
\begin{equation}
\begin{aligned}
\label{MetricEquivEq}
C_1\o\le \ovp \le C_2\o,
\end{aligned}
\end{equation}
for some
constants $C_2\ge C_1>0$.
We start with a simple result which shows that a Laplacian estimate can be interpreted geometrically. Denote
$$
\tr_\o\eta:=\tr\big([g_{i\b j}]^{-1}[h_{k\b l}]\big),
$$
where 
$
\o=g_{i\bar j} dz^i\w \bdz j,\,
\eta=h_{i\bar j}dz^i\w \bdz j$ in local coordinates. Similarly, denote
$$
\h{det}_\o\eta:=\det\big([g_{i\b j}]^{-1}[h_{k\b l}]\big).
$$

\begin{exer} {\rm
\label{FirstMetricEquivExer}
Show that  \eqref{MetricEquivEq} is implied by either
\beq
\label{FirstC2Eq}
n+\D_\o\vp=\tr_\o\ovp\le C_2,\q
\h{ and\ }
\h{det}_\o\ovp\ge C_1C_2^{n-1}/(n-1)^{n-1},
\eeq
or,
\beq
\label{SecondC2Eq}
n-\D_\ovp\vp=\tr_{\ovp}\o\le 1/C_1,\q
\h{ and\ }
\h{det}_\o\ovp\le C_1C_2^{n-1}(n-1)^{n-1}.
\eeq
} \end{exer}

\begin{exer} {\rm
\label{}
Conversely, show that   \eqref{MetricEquivEq} implies
$$
\tr_\o\ovp\le nC_2, \q
\h{ and\ }
\h{det}_\o\ovp\ge C_1^n,
$$ 
as well as
$$
\tr_{\ovp}\o\le n/C_1, \q 
\h{ and\ }
\h{det}_\o\ovp\le C_2^{n}.
$$
(Indeed, $\sum(1+\la_j)\le A$, and
$
\Pi(1+\la_j)\ge B
$
implies $1+\la_j\ge (n-1)^{n-1}B/A^{n-1}$;
conversely,
$
\Pi(1+\la_j)\ge \big(\frac1n\sum\frac 1{1+\la_j} \big)^{-n}
\ge C_1^n.
$)
} \end{exer}

The quantities $\tr_\o\ovp$ and $\h{det}_\o\ovp$ 
have a nice geometric interpretation.
To see that, we will study the geometry of the identity map $\iota:M\ra M$!
Consider $\del\iota^{-1}$
either as a map from $T^{1,0}M$ to itself, or
as a map from $\Lambda^n T^{1,0}M$ to itself.
Alternatively,
it is section of $T^{1,0\,\star}M\otimes T^{1,0}M$,
or of $\Lambda^n T^{1,0\,\star}M\otimes \Lambda^n T^{1,0}M$, and we may
endow these product bundles with the product metric
induced by $\o$ on the first factor, and by $\ovp$ on the second factor.
Then,  \eqref{FirstC2Eq} means that the norm squared
of $\del\iota^{-1}$, in its two guises above,
is bounded from above by $C_2$, respectively
bounded from below by $C_1C_2^{n-1}/(n-1)^{n-1}$.
Similarly,  the quantities $\tr_\ovp\o$ and $\h{det}_\ovp\o$ 
and
\eqref{SecondC2Eq} can be interpreted in
terms of $\del\iota$.

Now, for us the quantities
$$
\h{det}_\o\o_{\vp(s,t)}= e^{tf_\o+c_t-s\vp(s,t)}
$$
and 
$$
\h{det}_{\o_{\vp(s,t)}}\o= e^{-tf_\o-c_t+s\vp(s,t)}
$$
are already uniformly bounded thanks to the uniform estimate on 
$||\vp(s,t)||_{L^\infty}$ obtained in \S\ref{LinfsubrectSubSec}--\ref{LinfintSubSec}. Thus, 
according to Exercise \ref{FirstMetricEquivExer}, it remains to
find an upper bound for either
$|\del\iota^{-1}|^2$ 
or $|\del\iota|^2$
(from now on we just consider maps on $T^{1,0}M$). 

The standard way to approach this is by using the maximum
principle, and thus involves computing the
Laplacian of either one of these two quantities.
The classical approach, due to Aubin \cite{Aubin1970,Aubin1976,Aubin1978}
and Yau \cite{Yau1978}, is to estimate the first, while
a more recent approach is to estimate the second
\cite{R08,JMR}, and this builds on using and finessing older
work of Lu \cite{Lu68} and Bando--Kobayashi \cite{BK88}.
Both of these approaches are explained in a unified manner in  \cite[\S7]{R14}.
The result we need is \cite [Corollary 7.8 (i)] {R14}.

\def\osc{\operatorname{osc}}

\begin{lemma}
\label{LaplacianLemma}
Let $\vp\in C^4(M)\cap \PSH(M,\o)$.
Suppose that 
\begin{equation}
\begin{aligned}
\label{RicInEq}
\Ric\ovp\ge -C_1\o- C_2\ovp,
\end{aligned}
\end{equation}
and
\begin{equation}
\begin{aligned}
\label{biseqEq}
\max_M\h{\rm Bisec}_\o\le C_3.
\end{aligned}
\end{equation}
Then
\beq\label{FirstLaplacianEstimate}
-n<\Delta_\o\vp\le
(C_1+n(C_2+2C_3+1))e^{(C_2+2C_3+1)\osc\vp}-n.
\eeq

\end{lemma}

To see that this result is applicable, observe first  that  \eqref{biseqEq} holds simply because $M$ is compact and $\omega $ is smooth. Second, according to  \eqref{RicovpstEq} 
\begin{equation*}
\baeq
\Ric\o_{\vp(s,t)}
&=
(1-t)\Ric\o+s\o_{\vp(s,t)}+(\mu t-s)\o
\cr
&\ge
s\o_{\vp(s,t)}+(\mu t+(1-t)C_4-s)\o,
\eaeq
\end{equation*}
where $C_4 $ is a lower bound for the Ricci curvature of $\omega $, i.e., satisfying
$$
\Rico\ge C_4\o.
$$
Therefore, Lemma \ref{LaplacianLemma} holds with
$$
C_1 =\max\{0, |\mu t+(1-t)C_4-s|\},\q
C_2 =\max\{0, |s|\},\q
C_3=C_3(\o).
$$

\subsection {Higher order compactness via Evans--Krylov's estimate}

\def\calD{\dcal}

To show our solutions along the continuity method are smooth, it 
suffices to improve the Laplacian estimate to a $C^{2,\gamma}$ estimate
for some $\gamma>0$. Indeed, then it is standard to see that the solutions
automatically have uniform $C^{k,\gamma}$ estimates for each $k$
(Exercise \ref{allkExer}).	
This is obtained via the standard Evans--Krylov estimate, adapted to the 
complex setting. The standard references for this are
the lecture notes of Siu \cite{Siu} and \Blocki~\cite[\S5]{Bl}, 
as well as the treatment of the real \MA equation by Gilbarg and Trudinger \cite{GT}, with the modification by Wang--Jiang \cite{WJ}, \Blocki~\cite{Bl2000}.

\begin{lemma}
\label{CtwogammaLemma}
Let $\psi\in C^4(M)\cap \PSH(M,\o)$ be a solution to 
$\o_\psi^n=\on e^F$. Then
\begin{equation}
\label{HolderTwoEstimateEq2}
||\psi||_{C^{2,\gamma}}\le C,
\end{equation}
where $\gamma > 0$ and $C$ depend only on $M,\o,$  
$||\Delta_\o\psi||_{C^0},||\psi||_{C^0}$, and $||F||_{C^2}$.
\end{lemma}

\bpf
For concreteness, we carry out the proof for out particular
$F$ in \eqref{TwoParamCMEq}.
For each pair $(s,t)$ define $h=h(s,t)$ by 
\begin{equation}
\label{DefinehEq}
\log h:= tf_\o-s\vp+c_t+\log\det[\psi_{i\bar j}],
\end{equation}
where $\psi$ is a local \K potential for $\o$ on some fixed neighborhood
(we will obtain our estimate only on this neighborhood, but then cover $M$
with finitely many such).
Set $u:=\psi+\vp$. Each $C^{k,\gamma}$ norm 
of $\psi$ is bounded by a constant $C=C(k,\gamma,\o)$, so to get the desired
bound on $\vp$ is tantamount to bounding $u$.

Let $\eta=(\eta_1,\ldots,\eta_n)\in\CC^n$ be a unit vector, and consider $u$ as a function of $(z_1,\ldots,z_n)\in\CC^n$. Then,
$$
(\log\det[u_{i\bar j}])_{\eta\bar\eta} = -u^{i\bar l}u^{k\bar j}u_{\eta i\bar j}u_{\bar \eta k\bar l} +u^{i\bar j}u_{\eta\bar \eta i\bar j}.
$$
(repeated differentiation is justified since by assumption $\vp$ and hence $u$ belong to $C^4(M)$). Since
$$
\log\det[u_{i\bar j}]=
\log\det[\psi_{i\bar j}+\vp_{i\bar j}]=\log h,
$$
and letting 
\begin{equation}
\label{wfunctionEq}
w:=u_{\eta\bar \eta}, 
\end{equation}
we thus have
\begin{equation}
\label{SubsolIneq}
u^{i\bar j}w_{i\bar j}\ge (\log h)_{\eta\b\eta}=\frac{h_{\eta\b\eta}}h-\frac{|h_\eta|^2}{h^2},
\end{equation}
which can be rewritten in divergence form,
\begin{equation}
\label{SubsolSecondIneq}
(hu^{i\bar j}w_i)_{\bar j}
\ge 
\eta^{\b l}(\eta^k h_k)_{\b l}
-g,\qquad 
g:=\frac{|h_\eta|^2}{h}.
\end{equation}

\begin{theorem} \cite[Theorem 8.18]{GT}
\label{GTThm8.18}
Let $\Omega\subset\RR^m$, and assume $B_{4\rho}=B_{4\rho}(y)\subset \Omega$.
Let $L = D_i (a^{ij}D_j + b^i) + c^iD_i + d$ 
be strictly elliptic, $\lambda I<[a^{ij}]$, 
with $a^{ij},b^i,c^i,d\in L^\infty(\Omega)$, satisfying
$$
\sum_{i,j} |a^{ij}|^2 < \Lambda^2,
\quad
\lambda^{-2} \sum (|b^i|^2+|c^i|^2) + \lambda^{-1}|d| \le \nu^2.
$$
Then if $U\in W^{1,2}(\Omega)$ is nonnegative and satisfies
$LU \le g + D_i f^i$, with
$f^i \in L^q, g\in L^{q/2}$ with $q > m$, 
then for any $p\in[1,\frac m{m-2})$, 
$$
\begin{aligned}
\rho^{-m/p} || U ||_{L^p(B_{2\rho})}
&<
C(
\inf_{B_\rho} U + \rho^{1-m/q}||f||_{L^q(B_{2\rho})} + \rho^{2-2n/q}||g||_{L^{q/2}(B_{2\rho})}
)
\cr
&<
C(
\inf_{B_\rho} U + \rho||f||_{L^\infty(B_{2\rho})} + \rho^{2}||g||_{L^\infty(B_{2\rho})}
)
,
\end{aligned}
$$
with $C=C(n,\Lambda/\lambda,\nu \rho,q,p)$.
\end{theorem}

\begin{lemma}
Let $w$ be defined by \eqref{wfunctionEq}.
Suppose that $s>S$. 
One has
\begin{equation}
\label{WeakHarnackIneqComplexOne}
\sup_{B_{2\rho}}w-\frac1{|B_\rho|}\int_{B_\rho}w\o^n
\le
C\big(\sup_{B_{2\rho}}w  -\sup_{B_{\rho}}w+\rho(\rho+1)\big),
\end{equation}
with $C=C(M,\o,S,||\vp(s,t)||_{C^0(M)},
||\Delta_\o\vp(s,t)||_{C^0})$. 

\end{lemma}

\begin{proof}
By (\ref{SubsolSecondIneq}) $v:=\sup_{B_{2\rho}}w-w$ satisfies 
\begin{equation}
\label{SupersolSecondIneq}
(hu^{i\bar j }v_i)_{\bar j }
\le 
g-\eta^{\b l}(\eta^k h_k)_{\b l},
\end{equation}
where $g:=\frac{|h_\eta|^2}{h}$. In general, there are positive bounds on $[u_{i\bar j }]$ and $[u^{i\bar j }]$,
depending only on $||\Delta_\o\vp(s,t)||_{C^0}$, and hence 
similar positive bounds on $[a^{i\bar j }]:=h[u^{i\bar j }]$ and its inverse, in these coordinates, depending only on 
$S,M,\o,||\Delta_\o\vp(s,t)||_{C^0}$, and $||\vp(s,t)||_{C^0}$
(since the latter two quantities control $||\vp(s,1)||_{C^{0,1}}$ by interpolation).
This, together with Theorem \ref{GTThm8.18}, gives the desired inequalities
provided that $v\in W^{1,2}(M,\on)$, which is automatic as $v$ and $\on$ are smooth
and $M$ is compact.  
The lemma follows.
\end{proof}

Now, let $\{V_j\}_{j=1}^n$ be smooth vector fields on $M$ 
that span $T^{1,0}M$ over $M$ and that
on a local chart are given by 
$V_{k}:=\frac{\del}{\del z_k},
\; k=1,\ldots,n$, and denote
$$
M(\rho):=\sup_{|\zeta|,|Z|\in(0,\rho)}\sum_{j=1}^nV_{j}\overline{V_{j}} u,
\quad
m(\rho):=\inf_{|\zeta|,|Z|\in(0,\rho)}\sum_{j=1}^nV_{j}\overline{V_{j}} u
$$
Our goal is to show that
$
\nu(\rho):=M(\rho) - m(\rho)
$
is \Holder continuous with respect to $g_\o$, i.e.,
$
\nu(\rho)\le C \rho^{\gamma'},
$
for some $\gamma'>0$, or equivalently that
$
\nu(\rho)\le (1-\eps)\nu(2\rho)+\sigma(\rho),
$
for some $\eps\in(0,1)$ and some non-decreasing function $\sigma$
\cite[Lemma 8.23]{GT}.
Let 
$$
M_\eta(\rho):=\sup_{|\zeta|,|Z|\in(0,\rho)}u_{\eta\b\eta},
\quad
m_\eta(\rho):=\inf_{|\zeta|,|Z|\in(0,\rho)}u_{\eta\b\eta},
\quad
\nu_\eta(\rho):=M_{\eta}(\rho)-m_{\eta}(\rho).
$$
Equation (\ref{WeakHarnackIneqComplexOne}) implies
\begin{equation}
\label{WeakHarnackIneqComplexTwo}
\sup_{B_{2\rho}}w-\frac1{|B_\rho|}\int_{B_\rho}w
\le
C\big(\nu_\eta(2\rho)-\nu_\eta(\rho)+\rho(\rho+1)\big),
\end{equation}
and so it remains to obtain a similar inequality for $w-\inf_{B_{2\rho}}w$.

Note that
$
DF|_A.(A-B)\le F(A)-F(B),
$ by concavity of $F(A):=\log \det A $ on the space of positive Hermitian matrices.
Since $DF|_{\nabla^{1,1}u}=(\nabla^{1,1}u)^{-1}$, we have 
\begin{equation}
\label{FConcavityIneq}
\begin{aligned}
u^{i\bar j }(y)(u_{i\bar j }(y)-u_{i\bar j }(x)) \le \log\det u_{i\bar j }(y)-\log\det u_{i\bar j }(x)
\le |h|_{C^{0,1}}|y-x|.
\end{aligned}
\end{equation}
We now decompose $(u^{i\bar j })$ as a sum of rank one matrices.
This will result in the previous equation being the sum
of pure second derivatives for which we can apply our estimate
from the previous step. By uniform ellipticity this decomposition can be done
uniformly in $y$ \cite[p. 103]{Siu},\cite{Bl}. 
Namely, we can fix a set $\{\gamma_k\}_{k=1}^N$
of unit vectors in $\CC^n$ (which we can assume contains $\gamma_1=\eta$ 
as well as a unitary frame of which $\eta$ is an element) and write
$$
(u^{i\bar j }(y)) = \sum_{k=1}^N\beta_k(y)\gamma_k^*\gamma_k,
$$
with $\beta_k(y)$ uniformly positive depending only on $n,\lambda$
and $\Lambda$. Thus (\ref{FConcavityIneq}) 
gives
$$
\begin{aligned}
w(y)-w(x)
&\le C|y-x|-\sum_{k=2}^N \be_k(y)(u_{\gamma_k\bar{\gamma_k}}(y)-u_{\gamma_k\bar{\gamma_k}}(x)) 
\cr
&\le C|y-x|+\sum_{k=2}^N \be_k(y)(\sup_{B_{2\rho}}u_{\gamma_k\bar{\gamma_k}}-u_{\gamma_k\bar{\gamma_k}}(y)). 
\end{aligned}
$$
Now let $w(x)=\inf_{B_{2\rho}}w$, and average over $B_\rho$ to get,
using (\ref{WeakHarnackIneqComplexTwo}),
\begin{equation}
\label{ReverseWeakHarnackIneq}
\frac1{|B_\rho|}\int_{B_\rho}w-\inf_{B_{2\rho}}w
\le
C\big(\sum_{k=2}^N\nu_{\eta_k}(2\rho)-\nu_{\eta_k}(\rho)+\rho(\rho+1)\big).
\end{equation}
Combining this with (\ref{WeakHarnackIneqComplexTwo}), and summing
over $k=1,\ldots,N$ we thus obtain an estimate on $\nu(\rho)$
of the desired form. Hence $\Delta_\o\vp(s)\in C^{0,\gamma'}$ for some
$\gamma'>0$. In fact our proof actually showed that $\vp_{\eta\b\eta}\in C^{0,\gamma'}$ 
for any $\eta$. Hence, by polarization we deduce that also $\vp_{i\bar j }\in C^{0,\gamma'}$, 
for any $i,j$. Hence, $|\Delta_\o\vp(s,t)|_{C^{0,\gamma'}}\le C=C(M,\o,S,
||\Delta_\o\vp(s,t)||_{C^0(M)},
||\vp(s,t)||_{C^0(M)})$. This concludes the proof of Lemma 
\ref{CtwogammaLemma}.
\epf

\bexer
\lb{allkExer}
Suppose that $\vp\in C^\infty(M)$ satisfies $\ovpn=e^F\on$ 
and that 
$$
||\vp|_{C^{2,\gamma}}\le C.
$$
Show that there exists $C'$ such that
$$
||\vp|_{C^{3,\gamma}}\le C'=C'(M,\o,||F||_{C^{1,\alpha}}).
$$
{\rm(Hint: Let $D$ be a first order operator with constant coefficients
in some holomorphic coordinate chart. Write the \MA equation in those
coordinates as 
$$
\log\det[u_{i\bar j}]=\log\det[\psi_{i\bar j}]+F=:\tilde F,
$$ 
as 
in the proof of Lemma \ref{CtwogammaLemma} and apply $D$ to this equation.
By Exercise \ref{DetExer} this then gives a Poisson type equation for
$Du$,
$$
u^{i\bar j}(Du)_{i\bar j}=D\tilde F.
$$
This is not quite a Poisson equation since the Laplacian on the left hand
side depends on $u$ itself! However, since we already have uniform
$C^{0,\gamma}$ estimates on $[u^{i\bar j}]$ and $[u_{i\bar j}]$  the usual
Schauder estimates \cite{GT} give
$$
||Du||_{C^{2,\gamma}}\le C\big(||Du||_{C^{0,\gamma}}+||\tilde F||_{C^{0,\gamma}}\big).
$$ 
Since this holds for 
$$
D\in\bigg\{
\frac{\del}{\del z^1},\ldots,\frac{\del}{\del z^n},
\frac{\del}{\overline{\del z^1}},
\ldots,
\frac{\del}{\overline{\del z^n}}
\bigg\},
$$
we are done.)
}\eexer

Applying the previous exercise repeatedly yields the following
improvement of Lemma \ref{CtwogammaLemma}:
\begin{corollary}
\label{CkgammaCor}
Let $\psi\in C^{k+1}(M)\cap \PSH(M,\o)$ be a solution to 
$\o_\psi^n=\on e^F$. Then
\begin{equation}
\label{HolderTwoEstimateEq2}
||\psi||_{C^{k,\gamma}}\le C,
\end{equation}
where $\gamma > 0$ and $C$ depend only on $M,\o,$  
$||\Delta_\o\psi||_{C^0},||\psi||_{C^0}$, and $||F||_{C^{k-1}}$,
and $C$ depends additionally also on $k$.
\end{corollary}

\subsection {Properness implies existence}

We now complete the proof of one direction of Conjecture \ref{TianConj}. 
Let $K $ be a connected compact subgroup of the automorphism group.
Recall that
$$\H^K\subset\H $$
consists of all $K$-invariant elements of $\H$.
We denote by
$$
 \h{$C_K^{k,\gamma}$} 
$$
the subset of \h{$C^{k,\gamma}$}  consisting of $K $-invariant  functions.
 Denote by
$$
B\subset A
$$
the subset of parameter values $(s,t)$ for which there exists a $K$-invariant 
$C^{2,\ga}$ solution $\vp(s,t)$ to  
 \eqref{}. Note that $(-  \infty , 0]\times \{0\}\subset B $
since $ \varphi (s, 0) = 0 $ and we can always assume that $\omega$ is $K $-invariant, for instance by taking an arbitrary  K\"ahler metric and averaging it with respect to the Haar measure of $G$ \cite[p. 88]{Helg}.

Next, observe that the openness arguments of \S \ref{OpenSubSec} run through unchanged for $K $-invariant solutions. This is because
$M_{s,t}:C^{2,\gamma}\cap\PSH(M,\o)\ra C^{0,\gamma}$ defined by 
$$
M_{s,t}(\vp):= \log \frac{\o_{\vp}^n}{\on}-tf_\o+s\vp-c_t,
\quad (s,t)\in A,
$$
actually maps $C^{2,\gamma}_K\cap\PSH(M,\o)$ to $C_K^{0,\gamma}$, and therefore
$$ 
DM_{s,t}|_{\vp(s,t)} =\Delta_{\vp(s,t)}+s,
\quad (s,t)\in A,
$$
maps $C^{2,\gamma}_K$ to $C_K^{0,\gamma}$.
In conclusion then, $B $ is a nonempty open subset of $A $.
Moreover, if
\begin{equation}
\begin{aligned}
\label{ASEq}
A_{S}:=(- S ,-1/S]\times[0,1],
\end{aligned}
\end{equation}
we have that
$B\cap A_{S}$ is a nonempty open subset of $A_{S} $
for any value $S> 1$.

First, we show that $A_{S}\subset B $. Indeed, let $(s, t)\in \del(B\cap A_{S})$,
and let $\{(s_j, t_j)\}_{j}\subset B\cap A_S$ be a subsequence converging to $(s, t) $. 
According to Lemma \ref{NegCzeroLemma},
$$
\sup_j\max_M|\vp(s_j,t_j) | <C(1+S).
$$ 
Then, according to Lemma \ref{LaplacianLemma},
$$
\sup_j\max_M|\Delta_\omega\vp(s_j,t_j) | <C=C(M,\o,S).
$$
Thus, according to Lemma \ref{CtwogammaLemma},
$$
\sup_j\max_M||\vp(s_j,t_j)||_{ C^{2,\gamma} } <C=C(M,\o,S).
$$
Therefore, for every $\alpha\in (0,\gamma) $, the functions $ \varphi (s_j, t_j) $ converge to 
$ \varphi (s, t) $ in the  \h{$C^{2,\alpha}$}  topology, and moreover
$ \varphi (s, t)\in$\h{$C^{2,\gamma}$}. Thus, $(s, t)\in B$. This completes the proof that $A_{S}\subset B$,
So we have shown that
$$
\cup_{S>1}A_{S}=(- \infty , 0)\times [0, 1]\subset B.
$$
Observe that this actually concludes the proof of Corollary \ref{AubinYauCor} whenever $\mu <0 $ thanks to the elliptic regularity results mentioned below. 

Second, we show that actually $A_\infty\subset B$. Indeed, $(0, 0)\in B,$
and by openness also $\{(0, t)\,:\, 0 <t <\epsilon\}\subset B $, for some $\epsilon>0. $    
Applying now Lemma \ref{CYLemma} instead of Lemma \ref{NegCzeroLemma}, we get just as in the previous paragraph 
$$
\sup_t\max_M||\vp(0,t)||_{ C^{2,\gamma} } <C=C(M,\o).
$$
Thus, as before it follows that $A_\infty\subset B$.
Observe that the solutions we constructed are continuous in the parameters 
$s, t\in A_\infty $, in particular even up to $s = 0, $ since we use the openness argument that relies on the implicit function 
theorem that necessarily produces solutions that depend continuously on the parameters.   
This actually concludes the proof of Corollary \ref{AubinYauCor} 
(again, thanks to the elliptic regularity results).

Third, we treat the remaining piece in $A$. First, by openness
$\{(s, 1)\,:\, 0 \le s <\epsilon\}\subset B $, for some $\epsilon=\eps(M,\o)>0. $
Therefore, by Lemma \ref{CzeroposLemma} together with the higher-order estimates (as in the preceding paragraphs)
$$
\sup_s\max_M||\vp(s,1)||_{ C^{2,\gamma} } <C=C(M,\o).
$$
Once again, this is enough to conclude that $[0,\mu]\times\{1\}\subset B$. Thus, 
$$
B=A,
$$
as desired.

Finally, by Corollary
\ref{CkgammaCor}
(standard elliptic regularity results), the  \h{$C^{2,\gamma}$}  solutions we constructed are actually smooth. Thus, $ \varphi (\mu, 1)\in\H^K $, and $\o_{ \varphi (\mu, 1)}$ is $K $-invariant K\"ahler--Einstein metric. This concludes the proof of Theorem \ref{propernessexistence}.

\section {A counterexample to Tian's first conjecture and a revised conjecture}
 \label{CounterexampleSection}

Theorem \ref{propernessexistence} shows that properness implies existence.
This is one direction of Tian's first conjecture (Conjecture \ref{TianConj}).
The special case when there are {\it no} automorphisms (by which we mean
ones homotopic to the identity, i.e.,
$\AutMJz=\{\id\}$) of the other, harder, direction of Conjecture \ref{TianConj} 
was established by Tian \cite{Tian97} under a technical assumption that
was removed by Tian--Zhu \cite{TZ00}. This gave considerable plausibality 
to the conjecture. We now explain another reason 
why the general case of the conjecture seems plausible.

\subsection{Why Tian's conjecture is plausible}
\label{}
First we explain why it is natural (in fact, necessary!) for this harder converse direction
to only try to establish properness on $\H^K$ and not on all of $\H$.
For this, observe first that $E$ is invariant under the action of $\AutMJz$
whenever a \KE metric exists:

\bclaim
\lb{FutClaim}
Suppose $(M,\JJJ,\o)$ 
is Fano \KE
with $\mu[\o]=2\pi c_1(M)$ and $\mu>0$.
Then $E(g^\star\ovp)=E(\ovp)$ for all $g\in \AutMJz$ and $\vp\in\H$.
\eclaim

\bpf
By \eqref{EderivEq},
\beq
\lb{}
\baeq
\frac{d}{dt}\Big|_{t=0}E((\exp_It X)^\star\ovp) 
&= 
-\V\int_M \psi^X_{\ovp}\Delta_\ovp f_\ovp \o_{\vp}^n
\cr
&= 
-\V\int_M \psi^X_{\ovp}(s_\ovp-n\mu) f_\ovp \o_{\vp}^n.
\eaeq
\eeq
By a theorem of Futaki the functional
$$
\eta\mapsto \int_M \psi^X_{\eta}(s_\eta-n\mu) f_\eta \eta^n
$$
is constant on $\H$ \cite{Fut,Calabi1984,Bourg}. Since it is zero at $\o$ 
($\Ric\o=\mu\o$ implies $s_\o=n\mu$), it is identically zero.
Thus,
\beq\lb{zeroFutEq}
\frac{d}{dt}\Big|_{t=0}E((\exp_It X)^\star\ovp) 
=0.
\eeq
Now, actually 
\beq
\lb{}
\baeq
\frac{d}{dt}\Big|_{t=s}E((\exp_It X)^\star\ovp) 
&=
0
\eaeq
\eeq
for every $s$. Indeed,
\beq
\lb{}
\baeq
\frac{d}{dt}\Big|_{t=s}E((\exp_It X)^\star\ovp) 
&= 
\frac{d}{dt}\Big|_{t=0}E((\exp_I(s+t) X)^\star\ovp) 
\cr
&=
\frac{d}{dt}\Big|_{t=0}E((\exp_ItX)^\star((\exp_IsX)^\star\ovp))
\cr
&=
0,
\eaeq
\eeq
by replacing $\ovp$ by $(\exp_IsX)^\star\ovp$ in \eqref{zeroFutEq}.
Since $\AutMJz$ is a covered by its one-parameter subgroup, the statement follows.
\epf

On the other hand, the Aubin functional is not invariant under the action of automorphisms. In fact, it might blow up along a one-parameter subgroup.
The following lemma is due to Bando--Mabuchi \cite[Lemma 6.2]{BM}.
\begin{lemma}
\label{BMpropProp}
Let $\o\in\H$ be arbitrary and suppose $\eta\in\H$ is \KE with $\mu>0$.
The function $F_\eta:\AutMJz\to\RR_+$,
$$
F_\eta:g\mapsto (I-J)(g^\star\eta) 
$$
is proper (when we identify $\AutMJz$ with its $\eta$-orbit in $\H$
and endow this subset of $\H$ with the $C^{2,\gamma}(M,\o)$-topology).

\end{lemma}

\def\ovpj{\o_{\vp_j}}
\def\ovpjn{\o_{\vp_j}^n}
\def\vpj{{\vp_j}}

\bpf Indeed, suppose
that $I-J$ is bounded on a sequence $\o_j=\o_{\vp_j}$ of \KE metrics in 
$\AutMJz.\eta\subset\H$.
By Claim \ref{FutClaim} and \eqref{E2ndEq},
$$
C=E(\vp_1)=
E(\vp_j)=
\Ent(e^{f_\o} \o^n,\o^n_\vpj)
-\mu (I-J)(\vpj).
$$
Now, normalize $\vpj$ so that (recall \eqref{KEEq})
\begin{equation*}
\ovpjn=\on e^{\fo-\mu\vpj} 
\end{equation*}
(this fixes $\vpj$ since $\mu>0$ and the right hand side must integrate to $V$).
Plugging back into the formula for $E(\vpj)$ gives
$$
C=
E(\vp_j)=
-\mu\int\vpj\ovpjn-\mu (I-J)(\vpj),
$$
or,
$$
-\mu\int\vpj\ovpjn=C+\mu (I-J)(\vpj)\le C',
$$
by assumption that $(I-J)(\vpj)$ is uniformly bounded.
But now
\beq
\lb{vpoEq}
\baeq
\V\int\vpj\on 
&=
- \V\int-\vpj\ovpjn
+
I(\vpj)
\cr
&\le
- \V\int-\vpj\ovpjn
+
\frac 
{n+1}n(I-J)(\vp_j)
\cr
&\le
-(I-J)(\vp_j)
+\frac{n+1}n(I-J)(\vp_j)
\cr
&=\frac1{n+1}(I-J)(\vp_j)\le \frac C{n+1}.
\eaeq
\eeq
Thus, by \eqref{MeanVal},
$$
\max\vpj\le \V\int\vpj\on+C'<C''.
$$
Now a Moser iteration argument just as in \S\ref{LinfintSubSec} applies (the Sobolev
and Poincar\'e constants of the \KE metrics of Ricci curvature equal to $\mu>0$
are all uniform) to give 
$$
-\min\vpj\le \frac CV\int-\vpj\ovpjn +C.
$$
Combining the last two equations,
$$
\osc\vpj=\max\vpj-\min\vpj\le   C + \frac CV\int-\vpj\ovpjn +C''\le C''',
$$
using the display prior to \eqref{vpoEq}. 
Since $\vpj$ must change signs (from the
normalization for $\vp_j$ inherent in 
$\ovpjn=\on e^{f_\o-\mu\vpj}$ and the one for $f_\o$ in Definition \ref{foDef}), we have showed
that
$$
||\vp_j||_{L^\infty}<C,
$$ 
and consequently
$$||\vp_j||_{C^{k,\gamma}}<C({k,\gamma}),$$ for all $k,\alpha$, which
when $k=2$ gives
$$
C^{-1}\o\le\o_j\le C\o.
$$
Thus, endowing $\AutMJz$ with, say, the $C^{2,\gamma}$-topology we see that 
the preimage under of $F_\eta$ of compact sets in $\RR_+$ are compact
in the $C^{2,\gamma}$-topology, i.e., by definition $F_\eta$ (the original
$F_\eta$ considered as a map on the group $\AutMJz$) is proper. 
\epf

\begin{corollary}
\label{}
Suppose $(M,\JJJ,\eta)$ 
is Fano \KE
with $\mu[\o]=2\pi c_1(M)$ and $\mu>0$ and that $\AutMJz$ is nontrivial.
Then $(I-J):\{g^\star\eta\,:\, g\in \AutMJz\}\ra\RR_+$ is unbounded 
from above.
\end{corollary}

\begin{proof}
Indeed, by Corollary \ref{CartanCor} below
$F_\eta$ descends to a function on $\isom(M,g)$, still denoted by $F_\eta$,
$$
F_\eta(X)=(I-J)\big((\exp_I\JJJ X)^\star\eta\big).
$$
Since this function is still proper and  $\isom(M,g)$ is a non-compact vector space,
$F_\eta$ must
be unbounded.
\end{proof}

\bremark
There is actually no particular need to look at the orbit of a \KE metric to show
unboundedness; the same is true for the orbit of {\it any } metric as long as a \KE 
exists. Indeed, if $\alpha,\omega,\eta\in\H$, with $\eta$ \KEno,
$$
E(g^\star\alpha)=
E(\o,g^\star \alpha)=
E(\o,g^\star\eta)
+
E(g^\star\eta,g^\star \alpha)
=
E(\o,g^\star\eta)
+
E(\eta,\alpha).
$$
Thus, $E(g^\star\alpha)$ is unbounded if and only if $E(g^\star\eta)$ is (as
$E(\eta,\alpha)$ is some fixed constant).
\eremark

\subsection{A counterexample}
\label{}

However, surprisingly, Tian's first conjecture 
(which was stated as a theorem in  \cite[Theorem 4.4]{Tian97})
was recently disproved by Darvas and the author by establishing the following optimal version of Tian's conjecture. 

\begin{theorem}
\label{KEGexistenceIntroThm}
Suppose $(M,\JJJ,\o)$ 
is Fano
with $\mu[\o]=2\pi c_1(M)$ and $\mu>0$,
 and that $K$ is a maximal compact subgroup of $\AutMJz$ with $\o\in\H^K$.
The following are equivalent:

\smallskip
\noindent (i) 
There exists a K\"ahler--Einstein metric in $\mathcal H^K$
and $\AutMJz$ has finite center.

\noindent (ii) 
There exists $C,D>0$ such that $E(\eta) \geq CJ(\eta)-D, \ \eta \in \mathcal H^K$.  \end{theorem}

Thus, restricting to the $K$-invariant potentials is necessary, but not sufficient, to guarantee properness.

\begin{remark} {\rm
\label{}
The estimate in (ii) gives a concrete version of the properness condition 
\eqref{PropernessEq}. The direction (i) $\,\Rightarrow\,$ (ii) is 
due to Phong et al. \cite[Theorem 2]{pssw}, 
building on earlier work of Tian \cite{Tian97} and Tian--Zhu \cite{TZ00}
in the case $\Aut(M,\JJJ)_0=\{\id\}$,
who obtained a weaker inequality in (ii) with $J$ replaced by $J^\delta$
for some $\delta\in(0,1)$ (for more details see
also the survey \cite[p. 131]{Tian2012}).
} \end{remark}

\bexample
\lb{MainExam}
\cite[Example 2.2]{DR2}
Let $M$ denote the blow-up of $\PP^2$ at three non colinear points.
It is well-known that it admits \KE metrics (see, e.g., \cite{WangZhu}).
In fact, one way to see this is by showing that Tian's invariant is equal 
to 1 for an appropriately chosen group of symmetries \cite{BatySeliv}
and then apply Corollary \ref{TianbigalphaKEThm} (with $\mu = 1 $). 
According to \cite[Theorem 8.4.2]{Dolg}, 
\begin{equation}
\begin{aligned}
\label{AutPtwoEq}
\AutMJz=(\CC^\star)^2.
\end{aligned}
\end{equation}
We will explain this fact in a moment. Given this, we see that
$\AutMJz$ is equal to its center which is clearly not finite.
Thus, Conjecture \ref{TianConj}  fails for $M$
by Theorem \ref{KEGexistenceIntroThm}.  
Following the appearance of \cite{DR2}, 
X.-H. Zhu informed the author that 
using toric methods
one can give
an alternative proof that
Conjecture
\ref{TianConj}  fails in the special case of toric Fano 
$n$-manifolds
that satisfy $\AutMJz=(\CC^\star)^n$.

To see   \eqref{AutPtwoEq}, observe that automorphisms homotopic to the identity map preserve the cohomology class of divisors. Thus, they preserve each of the three exceptional divisors. In particular, they descend to automorphisms of $\PP^2$ which preserve the three blowup points. By that we mean that if $f\in\AutMJz$
then $\pi \circ f \circ \pi^ {-1}\in\h {Aut}(\PP^2)$.
Now automorphisms of
$\PP^2$ are represented by invertible three-by-three matrices, up to a nonzero complex number. We may assume in this representation that the three points are then 
$[1:0:0], [0:1:0], [0:0:1]$ (since they are not collinear!). Thus, each such automorphism is represented by a diagonal matrix. Since the matrix is invertible, and determined up to a nonzero complex number, that matrix can be taken to be
$$
\begin{pmatrix}
a&0&0\cr
0&b&0\cr
0&0&1\cr
\end{pmatrix}, \q a,b\in\CC^\star.
$$
Conversely, the blow-up of $\PP^2$ at three non colinear points
is a toric manifold so its automorphism group contains a copy of
$(\CC^\star)^2$. 
Thus,  \eqref{AutPtwoEq} is established.

\eexample

These results motivate a reformulation of Tian's original conjecture.
To present this reformulation we first make an excursion to infinite dimensional metric geometry in the next sections. In Section \ref{RevisedConjSec} we return to state the reformulated conjecture, whose proof is described in Section \ref{ProofSec}. 

\section {Infinite dimensional metrics on $\H $}

Approaching problems in \K geometry through 
an infinite-dimensional
 perspective goes back to Calabi in 1953 \cite{Calabi54} and 
later Mabuchi in 1986 \cite{Mabuchi87}. These works proposed two different 
weak Riemannian metrics of $L^2$ type which have been studied extensively
since. 

The most widely studied such metric is the Mabuchi metric \cite{Mabuchi87},
\begin{equation}
\label{MabuchiMetricEq}
\gM(\nu,\eta)|_\vp:=\int_M\nu\eta\,\o_\vp^n,
\quad \nu,\eta\in T_\vp\calH_\o\cong  C^\infty(M),
\end{equation}
discovered independently also by Semmes \cite{Semmes1} and Donaldson
\cite{Don} (see, e.g., \cite[Chapter 2]{RThesis} for an 
exposition and further references). 

Calabi's metric is given by
\begin{equation}
\gC(\nu,\eta)|_\vp:=\int_M\D_\vp\nu\D_\vp\eta\,\frac{\o_\vp^n}{n!}.
\end{equation}
This metric was introduced by Calabi in the 1950s in talks
and in a research announcement \cite{Calabi54}. It
might seem a little less natural at first since it involves more derivatives than the Mabuchi metric. However, from a Riemannian geometric point of view it is actually more natural, since it is simply the 
$L ^2$ metric on the level of Riemannian metrics, as the following simple result shows.
To state this result we let 
$$\calM$$ denote
the infinite-dimensional space of all smooth Riemannian metrics on $M$.
The
Ebin metric, also called the $L^2$ metric \cite{E} 
is defined by
\begin{equation}
\gE(h,k)|_g:=\int_M \tr (\ginv h \ginv k) dV_g,
\end{equation}
where $g\in\calM$, $h,k\in T_g\calM$ and
$T_g\calM\cong \Gamma(\h{\rm Sym}^2T^\star\! M)$, 
the space of smooth, symmetric $(0,2)$-tensor fields on $M$.

\begin{prop}
\label{EbinCalabiProp}
{\rm
\cite[Proposition 2.1]{ClarkeR}}
Consider the inclusion $\iota_{\calH}:\calH\hookrightarrow\calM$.
Then, $\iota_{\calH}^\star\, \gE = 2\gC$.
\end{prop}

In other words, $(\calH,2\gC)$ is
isometrically embedded in $(\calM,\gE)$, or what is the same, the metric $\gC$ is induced by the metric $\gE$.

On the other hand, the Mabuchi metric is  more natural from a symplectic or complex 
geometry point of view.
As shown by Semmes and Donaldson, the Mabuchi metric can be considered as an infinite-dimensional analogue of the symmetric space metric structure on spaces of the form $G^\CC/G$ where $G $ is a compact Lie group, but where the group is now infinite-dimensional, more specifically the group of Hamiltonian diffeomorphisms of $(M,\omega) $. 
We refer the reader to  
\cite{Semmes1,Don}, \cite[Chapter 4] {Tianbook}, \cite{Thomas}.
In another vein, the Mabuchi metric is also natural from the point of view of semi-classical complex geometry, also referred to as \K quantization sometimes. We refer the reader to 
\cite{RThesis,PhongSturmSurvey,FKZ,RZ1}.

\section {Metric completions of $\H$} 

Historically, Calabi claimed that the completion
of his metric ``consists of the positive semidefinite K\"ahler metrics defining
the same principal class," i.e., of
\begin{equation*}
\begin{aligned}
\label{}
\{\omega_\vp:=\o+\i\ddbar\vp \,:\, \vp\in C^{\infty}(M), \,  \omega_\vp\ge0\}.
\end{aligned}
\end{equation*}
Except from this single line published in in his short talk 
abstract in 1953 \cite{Calabi54}, there has been no study or even 
conjectures in the literature concerning metric completions of $\H$.
The first article in this direction is due to
Clarke--Rubinstein in 2011 \cite{ClarkeR}, that we now turn to discuss.

\subsection {The Calabi metric completion}
\lb{CalabiSubSec}

Denote by $\dC:\H\times\H\ra \RR_+$ the distance function associated to metric $\gC$. 
It is defined as follows.
A curve $[0,1]\ni t \mapsto  \alpha_t \in \mathcal{H}$ 
is called smooth if $\alpha(t,z)$ is smooth in both $t $ and $z $.
Denote $\dot\a_t:=\del\a(t)/\del t$.
The length of a smooth curve $t \to \alpha_t$ is
\begin{equation}\label{curve_length_def}
\ellC(\alpha):=\int_0^1\sqrt{\gC(\dot \alpha_t,\dot \alpha_t)|_{\alpha_t}}dt.
\end{equation}

\bdefin
The path length distance of $(\H,\gC)$ is defined by
$$
\dC(\o,\eta):= 
\inf\{\ellC(\alpha)\,:\,\alpha:[0,1]\ra\H \h{\ is a smooth curve with \ }
\a(0)=\o,\, \a(1)=\eta\}.
$$
We refer to the pseudometric $\dC$ as the {\it Calabi metric}. 
\edefin

\begin{remark} {\rm
\label{}
As observed already by Calabi, the Calabi--Yau Theorem implies that 
$(\H,\gC)$ is isometric to a portion of a sphere in $L^2(M,\on)$, and therefore the Calabi (pseudo)-metric is actually a metric, justifying the above name (see, e.g., \cite[pp. 1488--1489] {ClarkeR} or \cite{Calamai}).  
Even though we refer to $\dC$ and to $\gC$ by the same name, we hope it 
will be clear below to which one we are referring to from the context.
} \end{remark}

The Calabi metric completion is given by the following theorem due to Clarke--Rubinstein
\cite[Theorem 5.6]{ClarkeR}.

\begin{theorem}
\label{calHdCCompletion} {\rm }
The metric completion of $(\calH,\dC)$ is given by
$$
\begin{aligned}
\overline{(\calH,\dC)}\cong
\{\vp\in \calE(M,\o)\,:\, \o^n_{\vp} &\h{\ \rm is  absolutely continuous with }
\cr &
\h{ \rm respect to $\o^n$ and } \o^n_\vp/\o^n\in L^1(M,\o^n)\},
\end{aligned}
$$
and is a strict subset of 
$$
\calE(M,\o):=
\left\{
\vp\in \PSH(M,\o)\,:\,
\lim_{j\ra\infty}\int_{\{\vp\le-j\}}(\o+\i\ddbar\max\{\vp,-j\})^n= 0\;
\right\}.
$$
Furthermore, convergence with respect to $\dC$ is characterized as follows.
A sequence $\{\o_{\vp_k}\} \subset \mathcal{H}$ converges to $\ovp \in
  \mathcal{H}$ with respect to $\dC$ if and only if $\o_{\vp_k}^n
  \rightarrow \ovpn$ in the $L^1$ sense, i.e.,
  \begin{equation*}
    \int_M \Big|\frac{\o_{\vp_k}^n}{\on}-\frac{\o_{\vp}^n}{\on} \Big|\on \rightarrow 0.
  \end{equation*}
\end{theorem}

\begin{remark} {\rm
\label{}
Observe that the metric completion turns out to be considerably larger than what Calabi claimed.
We also note that Theorem \ref{calHdCCompletion} was motivated by the computation of the metric completion of the ambient space  $(\calM,\gE)$ obtained in Clarke's thesis \cite{Clarke}. It is interesting to note that his result does not directly imply Theorem  \ref{calHdCCompletion} as one might suspect from Proposition \ref{EbinCalabiProp}.
} \end{remark}

\begin{remark} {\rm
\label{}
The space $\E (M,\omega) $ was introduced by Guedj--Zeriahi \cite[Definition 1.1]{gz}.
The statement of Theorem \ref{calHdCCompletion}
of course assumes that the measure $\ovpn$ can be defined for each $\vp\in\E(M,\o)$.
This is indeed the case, but requires considerable background from pluripotential theory. One
defines
$$
\ovpn:=\lim_{j\ra-\infty}
{\bf 1}_{\{\vp>j\}}
(\o+\i\ddbar\max\{\vp,j\})^n.
$$
By definition, ${\bf 1}_{\{\vp>j\}}(x)$ is equal to $1$ if $\vp(x)>j$ and zero otherwise,
and
the measure $(\o+\i\ddbar\max\{\vp,j\})^n$ is defined by the
work of Bedford--Taylor \cite{bt} since $\max\{\vp,j\}$ is bounded.
The limit is then well-defined as a Borel measure; 
for more details we refer to  \cite[p. 445]{gz}.
} \end{remark}

What is perhaps more interesting than computing the metric completion itself, is the fact that this computation yields nontrivial geometric information  \cite[Theorem 6.3]{ClarkeR}.

\begin{definition}
\label{CRDef}
We say that $(M,J)$ is Calabi--Ricci unstable (or CR-unstable)
if there exists a Ricci flow trajectory
that diverges
in $\overline{(\calH,\dC)}$.
Otherwise, we say $(M,J)$ is CR-stable.
\end{definition}

\begin{theorem}
\label{CRStabilityThm}
A Fano manifold $(M,J)$ is CR-stable if and only if it admits a \KE metric.
Moreover, if it is CR-unstable then any Ricci flow 
trajectory diverges in $\overline{(\calH,\dC)}$.
\end{theorem}

Theorem \ref{CRStabilityThm} might seem rather abstract, however it shows that convergence in the metric completion is fundamental geometrically. In addition, it can be stated entirely in terms of an a priori estimates without any reference to the metric completion \cite[Corollary 6.9]{ClarkeR}:

\begin{corollary}
\label
{LoneLtwoConvCor}
The Ricci flow (\ref{RFEq}) converges smoothly if
and only if
\begin{equation}
\label{LoneLtwoConvEq}
||s-n||_{L^1(\RR_+,L^2(M,\o(t)))}<\infty,
\end{equation}
where $s = s (t)$ denotes the scaler curvature of $(M,\o(t))$.
\end{corollary}

This improves a result of Phong et al.~\cite{pssw}, where
  \eqref{LoneLtwoConvEq} is replaced by 
$$||s-n||_{L^1(\RR_+,C^0(M))}<\infty,$$ which was 
proved by completely different methods. 
The novelty in Corollary \ref{LoneLtwoConvCor} is that it uses supposedly ``soft" infinite-dimensional geometry to prove actual ``hard" a priori estimates for a PDE. Of course, the catch is that some analysis does go into computing the metric completion and, aside from that, some PDE techniques are still needed in the proof of
Corollary \ref{LoneLtwoConvCor}. But, nevertheless, the idea that some PDE 
estimates can be explained using infinite-dimensional geometry seems attractive. 

\begin{exer} {\rm
\label{}
Show that the length of the curve $t\mapsto \o_{\vp(t)}$ with respect to the Calabi metric is equal to
$$
||s-n||_{L^1(\RR_+,L^2(M,\o(t)))}
$$
if $\o_{\vp(t)} $ satisfies the Ricci flow equation
\begin{equation}
\label{RFEq}
\frac{\del \o(t)}{\del t}=-\Ric \o(t)+\mu\o(t),\quad \o(0)=\o\in\calH.
\end{equation}
Also, show that any solution of  \eqref{RFEq} that starts in  $\H $
remains in $\H $ \cite{H}. Thus, it makes sense to write $\omega (t) =\omega_{\vp(t)}.$ 
} \end{exer}

Thus, Corollary \ref{LoneLtwoConvCor} shows that convergence of the flow is equivalent to having finite distance in the Calabi metric.

\begin{exer} {\rm
\label{KRFMAExer}
Rewrite  \eqref{RFEq} in the form of a complex Monge--Amp\`ere equation
\begin{equation}
\begin{aligned}
\label{RFMAEq}
\ovpn=\on e^{f_\o-\mu\vp+\dot\vp}, \q \vp(0)=\h{const}.
\end{aligned}
\end{equation}
} \end{exer}

We remark that, depending on the context, the choice of the constant $\vp(0)$ might involve some care
(see \cite[\S10.1]{CT},\cite[\S2]{PSS}).

\begin{exer} {\rm  
\label{}
Assuming the theory of short-time existence for  \eqref{RFEq} (which replaces the openness arguments for the continuity method)
show that for every  $\omega\in\H $ the equation  \eqref{RFMAEq}
admits a solution for all $t > 0 $ whenever $\mu < 0 $. To do this, use Exercise \ref{KRFMAExer} as well as the results of \S\ref{PropExistSec}. Moreover, show that as $t $ tends to infinity, the solutions $\o(t)$ converge to the  K\"ahler--Einstein metric. 
} \end{exer}

Recently, Darvas generalized Calabi's metric to a two-parameter family of Finsler metrics, given by
\begin{equation}
\lb{CpqEq}
||\eta||^{\h{\smlsev C}
,p,q}_\vp:=
\bigg(\int_M|\D_\ovp\eta|^p \,\Big(\frac{\o_\vp^n}{\on}\Big)^q\frac{\on}{n!}\bigg)^{1/q},
\end{equation}
and computed the corresponding metric completions, directly generalizing Theorem \ref{calHdCCompletion}.
Denote by $\dCpq:\H\times\H\ra \RR_+$ the path-length distance function associated to   \eqref{CpqEq}.

\begin{theorem}
\label{CalabipqThm} {\rm \cite[Theorem 1.1]{DarvasCal}}
Let $p,q\in (1, \infty)$ and $q\le p$.
The metric completion of $(\calH,\dCpq)$ is given by
$$
\begin{aligned}
\overline{(\calH,\dCpq)}\cong
\{\vp\in \calE(M,\o)\,:\, \o^n_{\vp} &\h{\ \rm is  absolutely continuous with }
\cr &
\h{ \rm respect to $\o^n$ and } \o^n_\vp/\o^n\in L^q(M,\o^n)\}.
\end{aligned}
$$
Furthermore, convergence with respect to $\dCpq$ is characterized as follows.
A sequence $\{\o_{\vp_k}\} \subset \mathcal{H}$ converges to $\ovp \in
  \mathcal{H}$ with respect to $\dCpq$ if and only if $\o_{\vp_k}^n
  \rightarrow \ovpn$ in the $L^q$ sense, i.e.,
  \begin{equation*}
    \int_M \Big|\frac{\o_{\vp_k}^n}{\on}-\frac{\o_{\vp}^n}{\on} \Big|^q\on \rightarrow 0.
  \end{equation*}
\end{theorem}

In particular, the metric completion is independent of $p$! This immediately yields, by the same results 
of \cite{ClarkeR} that lead to Corollary \ref{LoneLtwoConvCor}, 
the following improvement to Corollary \ref{LoneLtwoConvCor}  \cite[Theorem 1.1]{DarvasCal}.
 
\begin{corollary}
\label
{LoneLoneConvCor}
The Ricci flow (\ref{RFEq}) converges smoothly if
and only if
\begin{equation}
\label{LoneLtwoConvEq}
||s-n||_{L^1(\RR_+, L^1(M,\o(t)))}<\infty.
\end{equation}
\end{corollary}

\begin{exer} {\rm
\label{}
Show that the length of the curve $t\mapsto \o_{\vp(t)}$ with respect to $\dCinput{1,1}$ is equal to
$$
||s-n||_{L^1(\RR_+,L^1(M,\o(t)))}
$$
if $\o_{\vp(t)} $ satisfies the Ricci flow equation \eqref{RFEq}.
} \end{exer}

It would be interesting to obtain a proof of Corollary \ref{LoneLoneConvCor} using direct flow methods. At the same time, it is  remarkable that such metric completion techniques can lead to new estimates on geometric flows. We believe that this circle of ideas should find more applications in other geometric and analytic settings.

\subsection {The Mabuchi metric completion} 
 \label{MabuchiSubSec} 

As remarked earlier, the Calabi metric is more closely tied with the Riemannian geometry of $M$, and indeed convergence in the Calabi metric is related to convergence of the associated Riemannian volume forms. The Mabuchi metric, on the other hand, is more closely tied with the complex geometry of $M$, and so completely different methods would be needed to compute the Mabuchi metric completion. Using sophisticated techniques from pluripotential theory this was carried through by Darvas. A special case was also obtained around the same time by Guedj \cite{guedj}.
Define,
\begin{equation}
\label{HoEq}
\textstyle\calH_\o
=
\{\vp\,:\,\vp\in C^{\infty}(M), \,  \omega_\vp>0\},
\end{equation}
and
$$
\E_2:=\big\{\vp\in\E(M,\omega)\,:\, \int\vp^2\ovpn<\infty\big\}.
$$

A curve $[0,1]\ni t \mapsto  \vp(t) \in \mathcal{H}_\o$ 
is called smooth if $ \varphi (t,z)= \varphi(t)(z) \in C^\infty([0,1] \times M)$. 
Denote $\dot \varphi(t):=\del \varphi (t)/\del t$.
The length of a smooth curve $t \mapsto   \varphi(t)$ is
\begin{equation}\label{curve_length_def}
\ellM(\alpha):=\int_0^1\sqrt{\gM(\dot\varphi(t),\dot\varphi(t))|_{\varphi(t)}}dt.
\end{equation}

\bdefin
The path length distance of $(\Ho,\dM)$ is defined by
$$
\dM({ \varphi _0},{ \varphi _1}):= 
\inf\{\ellM( \varphi )\,:\, \varphi :[0,1]\ra\Ho \h{\ is a smooth curve with \ }
 \varphi (0)= \varphi _0,\,  \varphi (1)= \varphi _1\}.
$$
We call the pseudometric $\dM$ the {\it Mabuchi metric}. 
\edefin

The metric completion of the Mabuchi metric is given by the following theorem of Darvas \cite[Theorem 1]{da3}
which also justifies the name given to $\dM$ above.

\begin{theorem}
\label{d2CompletionThm}
$(\H_\o,\dM)$ is a metric space.
Moreover, the metric completion of $(\H_\o,\dM)$ equals $(\E_2,\dMinput2)$,
where 
\begin{equation}
\begin{aligned}
\label{d2Eq}
\dMinput2(\vp_0,\vp_1):=\lim_{k\ra\infty}
\dM(\vp_0(k),\vp_1(k)),
\end{aligned}
\end{equation}
for any smooth decreasing sequences $\{\vp_i(k)\}_{k\in\NN}\subset\H$
converging pointwise to $\vp_i \in \mathcal E_2, i=0,1$.

\end{theorem}

Of course, the statement should be understood as also 
including the claims that: (i)   \eqref{d2Eq} is well-defined independently of the choices
of the approximating sequences, (ii) convergence in the metric completion is characterized as follows: $\{\vp_j\}\subset \calE_2$ converges to $\vp\in\calE_2$ if $\lim_j \dMinput2(\vp_j,\vp)=0$.

\begin{remark} {\rm
\label{HHoRemark}
The space $\H $   \eqref{HEq} is the space of K\"ahler forms, while the space $\H_\omega $ \eqref{HoEq} is the space of  K\"ahler potentials. In many instances one can go back and forth between the two carelessly, however in some situations some care is needed. One may also identify the latter as a subspace of the former in several ways, but again some care is needed in doing so. For example,
\begin{equation}
\begin{aligned}
\label{HHoEq}
\Ho\cap\{\h{\rm AM}=0\}
\end{aligned}
\end{equation}
is a $\dM$-totally geodesic submanifold (hypersurface) of $\Ho$
 \cite[Proposition 2.6.1]{Mabuchi87},  \cite[\S3]{Don}. The submanifold   \eqref{HHoEq} can be naturally identified with $\H$. Sometimes, though, we will use identifications different from  \eqref{HHoEq}.
} \end{remark}

In the vein of Remark \ref{HHoRemark}, we distinguish between solutions of   \eqref{RFEq}, which we continue to refer to as solutions to the Ricci flow, and solutions of  \eqref{RFMAEq}, which we refer to as solutions to the  K\"ahler--Ricci flow.

\begin{exer} {\rm
\label{}
Does the map $\o(t)\mapsto \vp(t)$ that sends solutions of 
 \eqref{RFEq} to solutions of \eqref{RFMAEq}, come from the identification of $\H$ with
 \eqref{HHoEq}?
} \end{exer}

Theorem \ref{d2CompletionThm} has already found several geometric applications. The first is the following analogue of Theorem \ref{CRStabilityThm} for the Mabuchi metric, due to Darvas \cite[Theorem 6.1]{da3}. 

\begin{definition}
\label{MRDef}
We say that $(M,J)$ is Mabuchi--Ricci unstable (or MR-unstable)
if there exists a  K\"ahler--Ricci flow trajectory
that diverges
in $\overline{(\calH,\dM)}$.
Otherwise, we say $(M,J)$ is MR-stable.
\end{definition}

\begin{theorem}
\label{MRStabilityThm}
A Fano manifold $(M,J)$ is MR-stable if and only if it admits a \KE metric.
Moreover, if it is MR-unstable then any Ricci flow 
trajectory diverges in $\overline{(\calH,\dC)}$.
\end{theorem}

\begin{exer} {\rm
\label{foExer}
Show that the length of the curve $t\mapsto \vp(t)$ with respect to $\dM$ is equal to
\begin{equation}
\begin{aligned}
\label{fotEq}
||f_{\o_{\vp(t)}}||_{L^1(\RR_+,L^2(M,{\o_{\vp(t)}}))}
\end{aligned}
\end{equation}
if $\vp(t)$ satisfies \eqref{RFMAEq} (which by Exercise \ref{KRFMAExer}  implies that $\o_{\vp(t)} $ satisfies the Ricci flow equation \eqref{RFEq}). As observed by Darvas, Theorem \ref{MRStabilityThm} 
together with the arguments of \cite{ClarkeR} imply the following analogue of Corollary \ref{LoneLoneConvCor} first obtained by McFeron \cite{McFeron}: the flow \eqref{RFMAEq}
converges if and only if   \eqref{fotEq} is finite. 
} \end{exer}

In fact, the following improvement of the last statement in Exercise \ref{foExer} is due to Darvas. It follows from 
 \cite[Theorem 6.1]{da3} together with later work of Darvas surveyed in \S\ref{FinslerSec}: 
\begin{theorem}
\label{LoneLoneConvMabThm}
The K\"ahler--Ricci flow (\ref{RFMAEq}) converges smoothly if
and only if
\begin{equation}
\label{LoneLtwoConvEq}
||f||_{L^1(\RR_+, L^1(M,\o_{\vp(t)}))}<\infty,
\end{equation}
where $f = f_{\o_{\varphi (t)}}$ is the Ricci potential along the flow (recall Definition \ref{foDef}).
\end{theorem}

Other applications for Theorem \ref{d2CompletionThm} 
include the work of Streets \cite{St}, and more recently Berman--Darvas--Lu \cite{BDL}, who show that one gains new insight on the long time behavior of the Calabi flow by placing it in the context of the 
Mabuchi metric completion;
the work of Darvas--He \cite{dh}, where the asymptotic behavior of the K\"ahler-Ricci flow in 
the metric completion is related to destabilizing geodesic rays.
We refer the reader to the survey \cite{R14} for more references.

\section {The Darvas metric and its completion}
\lb{FinslerSec}

Perhaps surprisingly, 
a key observation of Darvas is 
that not a Riemannian, but rather a {\it Finsler} 
metric, encodes the asymptotic behavior of the Aubin functional $J$.
This is discussed in the Section  \ref{DarvasAubinSec}.
In this section we introduce the Darvas metric and survey some of its basic properties.
In later sections, through considerable more technical work, 
we survey later work of Darvas--Rubinstein that shows that
the Darvas metric also encodes 
the asymptotic behavior for
essentially all 
energy functionals on $\calH$
whose critical points are precisely various types of canonical
metrics in \K geometry. 
In fact, as pointed out in \cite[Remark 7.3]{DR2}, the same
kind of statement is in general false for the much-studied
Riemannian metrics of Calabi and Mabuchi. Thus, the Darvas metric turns out to be 
fundamental.

The Darvas metric is a weak 
Finsler metric on $\mathcal H_\o$ given by \cite{da4},
\begin{equation}\label{FinslerDef}
\|\nu\|^{\hbox{\smlsev D}}_\vp:=  V^{-1}\int_M |\nu| \ovpn, 
\q  \nu \in T_\vp \mathcal{H}_\o=C^\infty(M).
\end{equation}
As in \S\ref{MabuchiSubSec}, define
the length of a smooth curve $t \mapsto  \varphi (t)$,
\begin{equation}\label{curve_length_defD}
\ellD(\alpha):=\int_0^1\int_M|\dot  \varphi (t)|\o_{\vp(t)}^n\wedge dt.
\end{equation}

\bdefin
The path length distance of $(\Ho,\dD)$ is defined by
$$
\dD({\varphi_0},{\varphi_1}):= 
\inf\{\ell_1(\alpha)\,:\,\alpha:[0,1]\ra\Ho \h{\ is a smooth curve with \ }
\a(0)=\varphi_0,\, \a(1)=\varphi_1\}.
$$
We call the pseudometric $\dD$ the {\it Darvas metric}. 
\edefin

The following result of Darvas justifies this name.
To state the result,
consider 
$[0,1]\times\RR\times M$ as a complex manifold of dimension
$n+1$, and denote by $\pi_2:[0,1]\times\RR\times M\ra M$ the natural projection.
\bthm
\lb{d1Thm}
{\rm \cite[Theorem 3.5]{da4}}
$(\Ho, \dD)$ is a metric space.
Moreover, 
\begin{equation}
\label{distgeod}
{\dD}(\varphi_0,\varphi_1)=\|\dot \varphi_0 \|_{\varphi_0}\ge0,
\end{equation}
with equality iff $\varphi_0=\varphi_1$, 
where $\dot \varphi_0$ is the image of $(\varphi_0,\varphi_1)\in \Ho\times\Ho$ under the 
Dirichlet-to-Neumann map for the \MA equation,
\beq\label{MabuchiEq}
\vp\in\PSH(\pi_2^\star\o, [0,1]\times\RR\times M),\q
(\pi_2^\star\o+\i\ddbar \vp)^{n+1}=0, 
\q
 \vp|_{\{i\}\times\RR}=\varphi_i,\; i=0,1.
\eeq
\ethm
\begin{remark} {\rm
\label{}
(i) The Dirichlet-to-Neumann operator simply maps $(\varphi_0,\varphi_1)$
to the initial tangent vector of the curve 
$$
t\mapsto \vp(t)\equiv \varphi_t
$$
that solves \eqref{MabuchiEq}.\hfill\break
(ii) One needs to make sense of the expression $\dot\vp_0$ in   \eqref{distgeod}
since there is no guarantee that $\vp_t$ will be smooth in $t $.
Since $\vp$ (considered as a function on $[0,1]\times\RR\times M$) is $\pi_2^*\o$-psh and independent of the imaginary part
of the first variable, it is convex in $t$. Thus,
\beq
\lb{dotu0eq}
\dot \varphi_0(x):=\lim_{t\ra 0^+}\frac{\vp(t,x)-\varphi_0(x)}{t},
\eeq
with the limit well-defined since the difference quotient
is decreasing in $t$. 

} \end{remark}

The metric completion of the Darvas metric is given by 
the next result \cite[Theorem 2]{da4}.
The proof is similar in spirit to that of Theorem \ref{d2CompletionThm}, but involves considerable additional technicalities stemming, at least intuitively, from the fact that $x\mapsto x^2$
is a smooth function while $x\mapsto |x|$ is only Liphscitz; partly due to this dealing 
with an $L^1$ type metric is fundamentally harder in this setting.

\bthm
\lb{d1CompletionThm}
The metric completion of $(\H_\o,{\dD})$ equals $(\E_1,{{\dD}})$,
where 
$$
{\dD}(\varphi_0,\varphi_1):=\lim_{k\ra\infty}
{\dD}(\varphi_0(k),\varphi_1(k)),
$$
for any smooth decreasing sequences $\{\varphi_i(k)\}_{k\in\NN}\subset\Ho$
converging pointwise to $\varphi_i \in \mathcal E_1, i=0,1$.
Moreover, for each $t\in(0,1)$, define
\begin{equation}\label{EpGeodDef}
\varphi_t:= \lim_{k \to\infty}\varphi_t(k), \ t \in (0,1),
\end{equation}
where $\varphi_t(k)$ is the solution of (\ref{MabuchiEq})
with endpoints $\varphi_i(k), i=0,1$. Then $\varphi_t\in \E_1$,
and the curve $t \to \varphi_t$ is well-defined independently of the choices
of approximating sequences and is a ${\dD}$-geodesic.
\ethm

\section{The Aubin functional and the Darvas distance function}
\label{DarvasAubinSec}

Finally we come to the fact stated at the beginning of the previous section relating the Darvas metric to the Aubin functional.

The subspace
\beq
\lb{H0Eq}
\calH_0:=\h{\rm AM}^{-1}(0)\cap \Ho
\eeq
is isomorphic to $\calH$ \eqref{HEq}, the space of K\"ahler metrics
(recall Remark \ref{HHoRemark}). We use this isomorphism to endow
$\calH$ with a metric structure, by pulling 
back the Darvas metric defined on $\Ho$. 

\begin{proposition} 
\label{Jproperness}
{\rm  \cite[Remark 6.3]{da4}}
There exists $C>1$ such that
for all $ \varphi  \in \mathcal H_0$ (recall (\ref{H0Eq})),
$$
\frac{1}{C}J( \varphi ) -C \leq {\dD}(0, \varphi ) 
\leq  C J( \varphi ) + C.
$$
\end{proposition}
We refer the reader to \cite [Proposition 5.5] {DR2} for a proof.

Given the equivalence of $J$ and $\dD$ on $\H_0$ it is natural to expect that this should extend to the metric completion. This is indeed the case. This amounts to two things: (i) one can extend Aubin's functional $J$ to the metric completion in a continuous way with respect to the $\dD$-topology,
(ii) $\H_0$, considered as a submanifold of $\H$ endowed with the metric induced by $\dD$, is a totally geodesic metric space whose completion coincides with $\E_1\cap \h{\rm AM}^{-1}(0)$, which in turn requires verifying that the Aubin--Mabuchi functional $\h{\rm AM}$ can be extended to $\E_1$
in a continuous way with respect to the $\dD$-topology. 
These facts are contained in the following Lemma  \cite[Lemma 5.2]{DR2}.

\blem
\lb{E1capH0Lemma}
(i)
$\h{\rm AM},J: \mathcal H_\o \to \Bbb R$ each admit a unique ${\dD}$-continuous extension 
to $\E_1$ 
and these extensions still satisfy (\ref{AMdef}) and (\ref{AubinEnergyEq})
(in the sense of pluripotential theory).

\noindent 
(ii)
The subspace $(\E_1\cap \h{\rm AM}^{-1}(0),{\dD})$ is a complete geodesic metric space, coinciding with the metric completion of $(\mathcal H_0,{\dD})$ (recall 
(\ref{H0Eq})).
\elem

Consequently, from now on we denote by $\h{\rm AM},J$ the unique ${\dD}$-continuous extensions 
to $\E_1$ given by the previous Lemma.

\begin{corollary} 
\label{JpropernessCorollary}
There exists $C>1$ such that
for all $ \varphi  \in \mathcal \E_1\cap \h{\rm AM}^{-1}(0)$, 
$$
\frac{1}{C}J( \varphi ) -C \leq {\dD}(0, \varphi ) 
\leq  C J( \varphi ) + C.
$$
\end{corollary}

Next, we discuss a concrete formula for the ${\dD}$ metric relating it to the Aubin--Mabuchi energy and also give a concrete growth estimate for ${\dD}$. First we need to introduce the following rooftop type 
envelope for $u,v \in \mathcal E_1$:
$$
P(u,v)(z):= 
\sup\big\{w(z)\,:\, w \in \PSH(M,\o),\, w \leq \min\{u,v\}\big\}.
$$ 
Note that $P(u,v) \in \mathcal E_1$ \cite[Theorem 2]{da3}.
Darvas shows the following beautiful ``Pythagorean" formula for ${\dD}$, as well as a 
very useful growth estimate \cite[Corollary 4.14, Theorem 3]{da4}.
\bprop
\lb{PythProp}
Let $u,v\in\E_1$. Then,
\begin{equation}
\label{Pythagorean}
{\dD}(u,v)=\h{\rm AM}(u) + \h{\rm AM}(v) - 2\h{\rm AM}(P(u,v)).
\end{equation}
Also, there exists $C>1$ such that for all $u,v\in\E_1$,
\begin{equation}
\label{d1CharFormula}
C^{-1} {\dD}(u,v) \leq \int_M |u-v|\o_u^n + \int_M |u-v|\o_v^n
\le C {\dD}(u,v).
\end{equation}
\eprop

\section {Quotienting the metric completion by a group action}

We now incorporate automorphisms into the picture. 
Since automorphisms induce isometries of the various 
infinite-dimensional metrics we have studied so far 
it is natural to consider the associated quotient spaces from the metric geometry point of view. In addition, the various functionals we have studied also admit natural descents to the quotient spaces.

\subsection{The action of the automorphism group on $\H $}
\label{}

Let $ \Aut_0(M,\JJJ) $ denote the connected component of the complex Lie group of automorphisms (biholomorphisms , i.e., homeomorphisms that are holomorphic and admit a holomorphic inverse) of $(M,\JJJ)$. Denote by  $\aut(M,\JJJ)$ the Lie algebra of $ \Aut_0(M,\JJJ) $, consisting of infinitesimal automorphisms, i.e., real vector fields $X$ satisfying $\calL_X\JJJ=0$,
equivalently,
\beq
\lb{CxStrucEq}
\JJJ[X,Y]=[X,\JJJ Y], \q \forall\, X\in\aut(M,\JJJ), \; \forall\, Y\in\h{\rm diff}(M),
\eeq
where $\h{\rm diff}(M)$ denotes all smooth vector fields on $M$.
Thus $\aut(M,\JJJ)$ is a complex Lie algebra with complex structure $\JJJ$.

The automorphism group $\AutMJz$  acts on $\H$ by pullback:
\beq\lb{AutActionEq}
f.\eta:=f^\star\eta, \qq f\in\AutMJz, \q \eta\in\H.
\eeq
Given the one-to-one correspondence between $\mathcal H$ and $\mathcal H_0$, the group $\AutMJz$  also acts on $\H_0$. The action is described in the next lemma.
\blem
For $\vp \in \mathcal H_0$ and $f \in \AutMJz$ let $f.\vp\in\H_0$ be the unique element such that $f.\ovp=\o_{f.\vp}$.
Then,
\beq
\lb{factionAMEq}
f.\vp=f.0+\vp\circ f, \qq  f\in\AutMJz, \q \vp\in\H_0.
\eeq
\elem

\bpf
Note that \eqref{factionAMEq} is a \K potential for $f^\star\ovp$. Indeed,   $f\in\AutMJ$ implies that $f^\star\i\ddbar\vp=\i\ddbar \vp\circ f$.
That $\h{\rm AM}(f.0+\vp\circ f)=0$ follows from Exercise \ref{AMDifferenceExer} as we have,
$$\h{\rm AM}(f.0+\vp\circ f)=\h{\rm AM}(f.0+\vp\circ f) - \h{\rm AM}(f.0)=\int_M \vp \circ f \sum_{j=0}^n f^\star \o^{n-j}\wedge f^\star \o_\vp^j = \h{\rm AM}(\vp)=0.$$
(Of course, $\h{\rm AM}(f.0)=0$ since by definition $f.0\in\H_0$.)
\epf

\begin{exer} {\rm
\label{AMDifferenceExer}
Show that
\beq
\lb{AMDifferenceEq}
\h{\rm AM}(v)-\h{\rm AM}(u)
=
\frac{V^{-1}}{n+1}\int_M(v-u)\sum_{k=0}^n \o_u^{n-k}\w \o_v^k.
\eeq
Among other things, this formula shows that $\h{\rm AM}$ is monotone, i.e.,  
\beq\lb{AMMonEq}
u \leq v \q\Rightarrow \q \h{\rm AM}(u) \leq \h{\rm AM}(v).
\eeq 
} \end{exer}

\blem
\lb{dpIsomLemma}
The action of $\AutMJz$ on $\H_0$ is a $\dD$-isometry.
\elem

\bpf
From \eqref{factionAMEq}, 
$$
\frac{d}{dt}f.\vpt=\dot\vpt\circ f,
$$
for any smooth path $t \mapsto \vp_t$ in $\H_0$. Thus, the $\dD$-length of $t \mapsto f.\vpt$
is
$$
V^{-1}\int_{[0,1]\times M}|\dot\vp_t\circ f| f^\star \o_{\vp_t}^n\w dt
=
V^{-1}\int_{[0,1]\times M}|\dot\vp_t| \o_{\vp_t}^n\w dt,
$$
equal to the $\dD$-length of $\vpt$.
\epf

Suppose $G$ is a subgroup of $\AutMJz$. By the previous lemma $G$ acts on $\mathcal H$ by ${\dD}$-isometries, hence induces a pseudometric on the orbit space $\H/G$,
$$
\dDG(Gu,Gv):=\inf_{f,g\in G}{\dD}(f.u,g.v).
$$
Here, we denote by $G u $ the orbit of $u $ under the action of $G$. Naturally, $Gu$ is an element of the orbit space $\H/G$. Thus, $\dDG $ measures the distance between orbits.

It is natural to expect that the group action extends to the metric completion. This is indeed the case.

\begin{lemma}
\label{LipschitzExt} 
Let $(X,\rho)$ and $(Y,\delta)$ be two complete metric spaces, $W$ a dense subset of $X$ and $f:W \to Y$ a $C$-Lipschitz function, i.e.,
\begin{equation}\label{lipineq}
\delta(f(a),f(b)) \leq C \rho(a,b), \q \forall\, a,b \in W.
\end{equation}
Then $f$ has a unique $C$-Lipschitz continuous extension 
to a map $\bar f:X\ra Y$.
\end{lemma}
\begin{proof}Let $w_k \in W$ be a Cauchy sequence converging to some $w \in X$. Lipschitz continuity gives
$$\delta(f(w_k),f(w_l)) \leq C \rho(w_k,w_l),$$
hence $\bar f(w) := \lim_k f(w_k) \in Y$ is well defined and independent of the choice of approximating sequence $w_k$. Choose now another Cauchy sequence $z_k \in W$ with limit $z \in X$, plugging in $w_k,z_k$ in \eqref{lipineq} and taking the limit gives that $\bar f: X \to Y$ is $C$-Lipschitz continuous.  
\end{proof}

\blem
\label{ActionExtension}
The action of $\AutMJz$ on $\mathcal H_0$ has a unique $\dD$-isometric extension 
to the metric completion  $\overline{(\H_0,\dD)} =(\mathcal E_1 \cap \h{\rm AM}^{-1}(0),\dD)$.
\elem

\bpf Because $\AutMJz$ acts by $\dD$-isometries, each $f \in \AutMJz$ induces a $1$-Lipschitz continuous self-map of $\mathcal H_0$. By Lemma \ref{LipschitzExt}, such maps have a unique $1$-Lipschitz extension to the completion $\mathcal E_1 \cap \h{\rm AM}^{-1}(0)$ and the extension is additionally a $\dD$-isometry. 
By density, the laws governing a group action have to be preserved as well.
\epf

For any Lie subgroup $K$ of the isometry group of $(M,g_{\o})$
define the subspace
\beq
\lb{HKEq}
\H_\o^K:=\{\vp\in\H_\o\,:\, \vp \h{ is invariant under $K$}\},
\eeq
and similarly define $\H_0^K=\H^K \cap \textup{AM}^{-1}(0)$.
According to Theorem \ref{d1CompletionThm}, the $\dD$-metric completion
of $\H_\o^K$ is
$$
\E_1^K:=\{u\in\E_1\,:\, u \h{ is invariant under $K$}\}.
$$

The next result follows using the arguments in the proofs
of Lemmas \ref{E1capH0Lemma} and \ref{ActionExtension}.

\blem
\lb{H0KCompletionLemma}
The metric completion of
$(\H_0^K,\dD)$ is 
$
\E_1^K\cap \h{\rm AM}^{-1}(0).
$
\elem

\subsection{The Aubin functional on the quotient space}
\label{}

Let $G\subset\AutMJz$ be a subgroup. 
Following Zhou--Zhu \cite[Definition 2.1]{zztoric} 
and Tian \cite[Definition 2.5]{Tian2012}, define 
the descent of $J$ to $\H/G$,
$$
J_G(Gu):=\inf_{g\in G}J(g.u).
$$
By Lemma \ref{E1capH0Lemma} this functional can be extended to a functional 
$J_G:\mathcal E_1 \cap \h{\rm AM}^{-1}(0)/G\ra \RR$, still satisfying
\beq
\lb{JGEq}
J_G(Gu)= \inf_{g \in G}  {J}(g.u).
\eeq
We now see that the key inequality between the Aubin functional and the Darvas distance function
(Proposition \ref{Jproperness}) descends to the metric completion of the quotient space.
\blem For $u \in \mathcal E_1 \cap \h{\rm AM}^{-1}(0)$ we have 
\lb{JGPropernessLemma}
\begin{equation}
\label{JGdGEqv}
 \frac{1}{C} J_G(Gu) -C \leq \dDG(G0,Gu) 
\leq  C J_G(Gu) + C,
\end{equation}
where $\dDG$ is the pseudometric of the quotient $\mathcal E_1 \cap \h{\rm AM}^{-1}(0)/G$.
\elem

\bpf
By Lemma \ref{dpIsomLemma}, 
$$
\dDG(G0,Gu)=\inf_{f\in G} \dD(0,f.u).
$$
The result  now follows from Proposition \ref{Jproperness}.
\epf

\section {A modified conjecture: Tian's second properness conjecture}
\lb{RevisedConjSec}

At last, we return to Conjecture \ref{TianConj} and pick up the discussion from where we left it at the end of Section \ref{CounterexampleSection}.
Lemma \ref{JGPropernessLemma} motivates the following modification of Conjecture \ref{TianConj}.

\bdefin
Let $F:\H\ra \RR$ be $G$-invariant.

\noindent $\bullet$ 
We say $F$ is {\it $\dDG$-proper } if
for some $C,D >0$,
$$
F(u) \geq C \dDG(G0,Gu) - D.
$$

\noindent $\bullet$  
We say $F$ is {\it $J_{G}$-proper } if
for some $C,D >0$,
$$
F(u) \geq C J_{G}(Gu) - D.
$$

\edefin

\bconj
\lb{MainConj} {\rm (Tian's second properness conjecture)}
Let $(M,\JJJ,\o$) be a Fano manifold. Set $G:=\Aut(M,\JJJ)_0$.
There exists a  K\"ahler--Einstein metric in $\mathcal H$ if and only if
the descent of the Mabuchi energy $E$ 
to the quotient space
$\calH/G$ is $\dDG$-proper
(equivalently, $J_{G}$-proper).
\econj

Note that according to Lemma \ref{JGPropernessLemma} both notions of properness are indeed equivalent. 
Also, the $G$-invariance condition can be considered as 
a version of the Futaki obstruction~\cite{Fut}.

Albeit being a purely analytic criterion,
properness should be morally equivalent 
to properness in a metric geometry sense, namely, that the
Mabuchi functional should grow at least linearly relative to some metric
on $\calH$, and this is precisely the content of Conjecture \ref{MainConj}. 

\begin{remark} {\rm
\label{DirichletRemark}
We now come back to the analogy with the Dirichlet  energy alluded to in the Prologue. There we seek to minimize the Dirichlet energy, say on the unit ball in $\RR^n$,
$$
E(f):=\int_{B_1(0)} \sum_{i=1}^n(\del_{x_i}f)^2 dx^1\w\cdots dx^n.
$$
The space of competitors $\H$ is now the space of smooth functions with prescribed boundary values 
$g\in C^\infty(\del B_1(0))$,
$$
\H:=\{f\in C^\infty(B_1(0))\,:\, f|_{\del B_1(0)}=g\}.
$$
In some sense, the prescribed boundary values can be morally thought of as the analogue 
for fixing a  K\"ahler class. What is the analogue of the Aubin functional? In this case it is just $E$ itself, i.e., we put $J=E$, so an analogue of Conjecture  \ref{TianConj} is trivial here.
However, the direct method in the calculus of variations motivates replacing $J$ (which is the $W^{1,2}$ seminorm) with the $W^{1,2}$ norm. Namely, we consider the metric
$$
(h,k):=\int \sum_{i=1}^n\del_{x_i}h\del_{x_i}k dx^1\w\cdots dx^n +
\int hk dx^1\w\cdots dx^n.
$$
The path-length distance is then just the one coming from the norm $W^{1,2}$, 
and the properness inequality is a consequence of the Poincar\'e inequality. This then implies
that a minimizer exists in the $W^{1,2}$ completion of $\H$. The Euler--Lagrange equation
is precisely the Laplace equation with prescribed boundary data.
Elliptic regularity theory then
shows the minimizer must be an element of $\H$ itself, hence a smooth harmonic function
agreeing with $g$ on the boundary.

} \end{remark}

In the remainder of these notes, we sketch the resolution of Conjecture \ref{MainConj} due to Darvas--Rubinstein \cite{DR2}. 

\begin{theorem} 
\label{KEexistenceIntroThm}
Conjecture \ref{MainConj} holds.
\end{theorem}

The proof of this result is completed in Section \ref{ProofSec}.

\begin{remark}
The easier implication ``$J_G$-proper $\;\Rightarrow\;$ existence of K\"ahler--Einstein" 
is due to Tian \cite[Theorem 2.6]{Tian2012} and is a modification of the proof of Theorem  \ref{propernessexistence}. Our proof of Theorem \ref{KEexistenceIntroThm}
also furnishes a new proof of this fact.
In the special case of toric Fano manifolds, a variant of the converse direction 
is due to Zhou--Zhu \cite[Theorem 0.2]{zztoric}. 
\end{remark}

\section {A general existence/properness principle}
 \label{GeneralSec} 

Motivated by Remark \ref{DirichletRemark}, we approach Conjecture \ref{MainConj} using an abstract metric geometry framework. While seemingly abstract it turns out to be a powerful way of dealing with several different minimization problems in K\"ahler geometry.

\bnot
\label{MainNot}
The data $(\mathcal R,d,F,G)$ is defined as follows.

\begin{enumerate}[label = (A\arabic*)]
  \item\label{a1} $(\mathcal R,d)$ is a metric space with a 
distinguished 
element $0\in\mathcal R$,
        whose metric completion is denoted $(\overline{\mathcal R},d)$.
  \item\label{a2} $F : \mathcal R \to \Bbb R$ 
is lower semicontinuous (lsc). Let  
        $F: \overline{\mathcal R} \to \Bbb R \cup \{ +\infty\}$ 
        be the largest lsc extension
        of $F: \mathcal R \to \Bbb R$:
$$F(u) = \sup_{\varepsilon > 0 } \bigg(\inf_{\substack{v \in \mathcal R\\ d(u,v) \leq \varepsilon}} F(v) \bigg), \ \ u \in \overline{ \mathcal R}.$$

For each 
$u,v\in{\mathcal R}$ define also
$$
F(u,v):=F(v)-F(u).
$$
  \item\label{a3} 
The set of minimizers of $F$ on $\overline{\mathcal R}$ is denoted
$$
\mathcal M:= 
\Big\{ u \in \overline{\mathcal R} \ : \ F (u)= 
\inf_{v \in \overline{\mathcal R}} F(v)
\Big\}.
$$
  \item\label{a4} Let $G$ be a group  
  acting on ${\mathcal R}$ by 
  $G\times{\mathcal R}\ni(g,u) \to g.u\in {\mathcal R}$. Denote by
${\mathcal R} /G$ the orbit space, by $Gu\in{\mathcal R} /G$ the 
  orbit of $u\in{\mathcal R}$, 
  and~define~$d_G:{\mathcal R}/G\times {\mathcal R}/G\ra\RR_+$ by
$$
d_G(Gu,Gv):=\inf_{f,g\in G}d(f.u,g.v).
$$
\end{enumerate}
\enot

\bhyp
\label{MainHyp}
The data $(\mathcal R,d,F,G)$
satisfies the following properties.

\begin{enumerate}[label = (P\arabic*)]
  \item\label{p1}
 For any $\varphi_0,\varphi_1 \in \mathcal R$ there exists a $d$--geodesic segment $[0,1] \ni t \mapsto \varphi_t \in \overline{\mathcal R}$ for which
 $t\mapsto F(\varphi_t) \textup{ is continuous and convex on }[0,1].$
  \item \label{p2} 
If  $\{\varphi_j\}_j\subset \overline{\mathcal R}$  satisfies 
$\lim_{j\ra\infty}F(\varphi_j)= \inf_{\overline{\mathcal R}} F$, and
for some $C>0$, $d(0,\varphi_j) \leq C$ for all $j$, then there exists a $u\in \mathcal M$ and a subsequence $\{\varphi_{j_k}\}_k$ 
$d$-converging to $u$.
  \item\label{p3} 
  $\mathcal M \subset \mathcal R.$
  \item\label{p4} $G$ acts on ${\mathcal R}$ by $d$-isometries.
  \item \label{p5} $G$ acts on $\mathcal M$ transitively.
  \item \label{p6} If $\mathcal M \neq \emptyset$,  
then for any $u,v \in \mathcal R$ there exists $g \in G$
such that $d_G(Gu,Gv)=d(u,g.v)$.

  \item \label{p7} For all $u,v \in \mathcal R$ and $g \in G$,
$F(u,v)=F(g.u,g.v)$.

\end{enumerate}

\ehyp

The following result will provide the aforementioned framework
for dealing with many minimization problems.

\begin{theorem}
\label{ExistencePrinc}
Let $(\mathcal R,d,F,G)$ be as in Notation \ref{MainNot}
and satisfying Hypothesis \ref{MainHyp}.
Then $\mathcal M$ is nonempty if and only if
$F:{\mathcal R}\ra \RR$ 
is $G$-invariant, and
for some $C,D>0$,
\begin{equation}\label{Dproperness}
    F(u) \geq Cd_G(G0,Gu)-D,\q \h{for all\ } u \in \mathcal R.
\end{equation}
\end{theorem}

One direction in this theorem is easy. Namely, if  \eqref{Dproperness} holds,
then  $F$ is bounded from below. 
By \ref{a2}, 
\begin{equation}
\label{InfRealization}
\inf_{v \in \overline{\mathcal R}} F(v)
=
\inf_{v \in \mathcal R} F(v).
\end{equation}
This, combined with \eqref{Dproperness}, 
the $G$--invariance of $F$ and the definition of $d_G$ implies there
exists $\varphi_j \in \mathcal R$ such that 
$\lim_{j}F(\varphi_j) = \inf_{\overline{\mathcal R}} F$ 
and $d(0,\varphi_j) \leq d_G(G0,G\varphi_j) + 1<C$ for $C$ independent of $j$. 
By \ref{p2}, $\mathcal M$ is non-empty.
For the other direction we refer the reader to \cite [Theorem 3.4] {DR2}. 

We have set up things in such a way that the modified properness conjecture, Conjecture \ref{MainConj}, would become a corollary of 
Theorem \ref{ExistencePrinc} applied to the following data
\begin{equation}
\begin{aligned}
\label{DataEq}
\mathcal R=\H_0, \q d=d_1, \q 
F=E, \q G:=\Aut_0(M,\JJJ),
\end{aligned}
\end{equation}
{\it if} this data satisfies the hypothesis of
Theorem \ref{ExistencePrinc}.
In the next sections we verify that this is indeed the case. 
Property \ref{p4} has already been verified in Lemma \ref{ActionExtension}. In the next few sections we verify the remaining hypothesis of   Theorem \ref{ExistencePrinc}.

\section{Applying the general existence/properness principle}
\lb{applyingSec}

In this section we briefly motivate---in the context of the K\"ahler--Einstein 
problem---some of the key assumptions
in the general existence/properness principle.
The point is to convince the reader that this principle
fits naturally/seamlessly with classical/foundational results in K\"ahler geometry.

First, a seemingly harmless condition, tucked into the ``notation" part of Theorem \ref{ExistencePrinc},
is that the functional we are trying to minimize on the metric completion
should be the greatest lower semicontinuous extension (with respect to the path-length metric)
of the functional we are trying to study originally on the ``regular" objects
$\calR$.  
This turns out to be quite a technical thing to verify.
At first, this might cause confusion: indeed any functional admits such an extension by means of the abstract formula
\begin{equation}
\begin{aligned}
\label{EAbstractEq}
F(u) = \sup_{\varepsilon > 0 } \bigg(\inf_{\substack{v \in \mathcal R\\ d(u,v) \leq \varepsilon}} F(v) \bigg), \ \ u \in \overline{ \mathcal R}.
\end{aligned}
\end{equation}
However, the issue is to verify that this abstract formula, say in the case of the Mabuchi energy, coincides with the original defining formula \eqref{EEq} which initially only makes sense on the space of smooth potentials $\calR=\H$. This is because only then can we actually verify that this extended functional satisfies the other hypothesis in Theorem \ref{ExistencePrinc} (without an explicit formula it is not clear how to proceed).
Fortunately, condition \ref{a2} for  \eqref{DataEq} does hold by the following result  \cite[Proposition 5.21]{DR2}.

\bprop
\label{EbetaExt}
Formula (\ref{EEq}) coincides with formula  \eqref{EAbstractEq} on $\E_1$. In other words, formula  \eqref{EEq} gives the greatest $d_1$-lsc extension of $E:\mathcal H \to \Bbb R$ 
to $\E_1$.
\eprop

\begin{remark} {\rm
\label{}
The analogue of this result for the Mabuchi metric $\dM$ can be found in  \cite{BDL}.
} \end{remark}

Second, property \ref{p1} holds for the Mabuchi energy due
to a result of Berman--Berndtsson \cite[Theorem 1.1]{bb}. 
In fact, we remark that it is well-known that the geodesic between smooth endpoints has considerable regularity
(as compared to just being in $\overline{\cal R}$) \cite{Bl12,Ch00}. 
In \cite{bb} it is shown that the Mabuchi energy is convex along such 
partially regular geodesics.

Third, property 
\ref{p2} stipulates 
precompactness of sublevel sets of the Mabuchi energy
with respect to the Darvas metric. 
Pre-compactness with respect to other functionals is a key result in the 
works \cite{bbgz,bbegz}, and can be adapted to show  
the aforementioned  pre-compactness \cite[Proposition 5.28]{DR2}.

Fourth,  property \ref{p3} stipulates regularity of minimizers 
of the Mabuchi energy in the metric completion.
This follows from the regularity result of Berman \cite[Theorem 1.1]{brm1}
combined with the characterization of the metric completion
of Darvas (Theorem \ref{d1CompletionThm}).

Fifth,  property \ref{p5}, modulo property \ref{p3}, amounts to the classical
Bando--Mabuchi theorem on uniqueness of \KE metrics up to
automorphisms.

Sixth, property \ref{p7} says that
the Mabuchi functional is exact, or of ``Bott--Chern" type,
and this is precisely Mabuchi's original theorem on his functional  
\cite[Theorem 2.4]{Mabuchi1986}. For an expository treatment we referred to \cite[\S5]{R14}.

Finally, property \ref{p6} is a new ingredient, and so we 
go into more detail, sketching property \ref{p6} for  \eqref{DataEq}.
It fits nicely into our framework since it shows precisely the role of another classical result in K\"ahler geometry, namely, 
Matsushima's classical theorem about the automorphism group of
a K\"ahler--Einstein manifold.
The key result in showing \ref{p6} is the following  \cite [Proposition 6.8] {DR2}.

\begin{proposition} 
\label{p6Prop}
Let $(M,\JJJ,\o,g)$ be \KEno. Define $(\mathcal R,d,F,G)$ by   \eqref{DataEq},
and suppose that \ref{a1}-\ref{a4} and \ref{p4} hold. Finally, assume the following:\\
\noindent (i) For each $X\in\isom(M,g)$, $t\mapsto\exp_It\JJJ X.\o$ 
is a $\dD$-geodesic whose speed depends continuously on $X$.\\
\noindent (ii) $\AutMJz\times \AutMJz \ni (f,g)\mapsto d(f.u,g.v)$
is  a continuous map for every $u,v\in\mathcal H$. \\
\noindent 
Then property \ref{p6} holds.
\end{proposition}

Condition (i) is essentially a corollary of   \eqref{distgeod}, while 
(ii) follows from   \eqref{d1CharFormula}.
Property \ref{p6} stipulates that a certain infimum over the group $G$ is attained. 
Thus, for the proof of Proposition \ref{p6Prop} we decompose the group $G$ into a compact part and a non-compact part in such a way that the compact part acts by $d$-isometries while on the non-compact part (but finite-dimensional!) we have $d$-properness. Then, together with conditions (i) and (ii),
the existence of a minimizer is guaranteed.

The aforementioned decomposition of the group into a compact and a non-compact part is stated in Corollary \ref{CartanCor} below.  
It should be well-known and relies on classical results that we now recall.
First, we recall Matsushima's classical theorem
\cite[Th\'eor\`eme 1]{Mat}.
We refer to Gauduchon \cite{ga} for more details.
Let $g(\,\cdot\,,\,\cdot\,)=\o(\,\cdot\,,\JJJ\,\cdot\,)$ denote the Riemannian
metric associated to $(M,\JJJ,\o)$.
Denote by $\Isom(M,g)_0$ the identity component of the isometry group of 
$(M,g)$. Since $M$ is compact so is $\Isom(M,g)_0$ \cite[Proposition 29.4]{Post}.
Denote by $\isom(M,g)$ the Lie algebra of $\Isom(M,g)_0$.

\bthm
\lb{VectorFieldDecompEdgeThm}
Let $(M,\JJJ,\o,g)$ be a Fano \K manifold. Suppose 
$g$ is a \KE metric. Then,
\begin{equation}\label{gsplit}
\aut(M,\JJJ)=\isom(M,g)\oplus \JJJ\,\isom(M,g).
\end{equation}
\ethm

The following result is classical, and we only state its \KE version, whose proof we sketch.
\bthm
\lb{MatsIwasawaThm}
Let $(M,\JJJ,\o,g)$ be \KEno.
Then any maximally compact subgroup of $\AutMJz$
is conjugate to $\Isom(M,g)_0$.
\ethm

\begin{proof}By a Theorem of Iwasawa--Malcev \cite[Theorem 32.5]{st}, if $G$
is a connected Lie group then its maximal compact
subgroup must be connected and any two maximal compact subgroups are conjugate. But then by Theorem
\ref{VectorFieldDecompEdgeThm} $\Isom(M,g)_0$ has to be a maximal compact
subgroup of $\AutMJz$.
\end{proof}

Next, we need a version of the classical Cartan decomposition
\cite[Proposition 32.1, Remark 31.1]{Bump}.

\bthm
\lb{CartanThm}
Let $S$ be a compact connected semisimple Lie group. Denote
by $(S^\CC,\JJJ)$ the complexification of $S$, namely the unique connected
complex Lie group whose Lie algebra is the complexification of that
of $\fs$, the Lie algebra of $S$. Then
the map $C$ from $S\times\fs$ to $S^\CC$ given by 
\beq
\lb{CartanMapFirstCaseEq}
(s,X)\mapsto C(s,X):=s\exp_I\JJJ X
\eeq 
is a diffeomorphism.
\ethm

Combining Theorems \ref{VectorFieldDecompEdgeThm}, \ref{MatsIwasawaThm} and \ref{CartanThm} 
we obtain the decomposition of $\AutMJz$ into a compact and a non-compact part
that is needed for the proof of Proposition \ref{p6Prop}. For details on how the following
result yields Proposition \ref{p6Prop} we refer to \cite[\S6]{DR2}, where a more general 
result is proven in the constant scalar curvature setting (when the Cartan type decomposition is not
given by classical results and we construct instead a ``partial Cartan decomposition" 
that may only be surjective).

\begin{corollary}
\label{CartanCor}
Let $(M,\JJJ,\o,g)$ be \KEno.
Then
the map $C$ from $\Isom(M,g)_0\times \isom(M,g)$ to $\AutMJz$ given by 
\beq
\lb{CartanMapFirstCaseEq}
(s,X)\mapsto C(s,X):=s\exp_I\JJJ X
\eeq 
is a diffeomorphism.

\end{corollary}

\section {A proof of Tian's second properness conjecture}
\lb{ProofSec}

As already explained at the end of Section \ref{GeneralSec}, and
as we started to elaborate in the previous section, we prove
Theorem \ref{KEexistenceIntroThm}
by applying Theorem \ref{ExistencePrinc} to data
  \eqref{DataEq}. Thus, it only remains to verify that
 this data satisfies the hypothesis of
Theorem \ref{ExistencePrinc}.

First, we go over Notation \ref{MainNot}.
First, in \ref{a1}, 
$\overline{\mathcal R}=\E_1\cap \h{\rm AM}^{-1}(0)$ by Theorem \ref{d1CompletionThm}
and Lemma \ref{E1capH0Lemma}.
Observe that \ref{a2} holds by Proposition \ref{EbetaExt}.
In \ref{a3}, the minimizers of $F$ are denoted by $\calM$.
Finally, \ref{a4} holds since $G\subset \AutMJz$ implies
that if $g\in G$ and $\eta\in\H$ then 
$g.\eta$ is both \K and cohomologous to $\eta$, i.e., $g.\eta\in\H$.
Thus, it remains to verify Hypothesis \ref{MainHyp}.

Properties \ref{p1}--\ref{p7} were all verified in \S\ref{applyingSec}
 with the exception of
property \ref{p4}, that itself follows from Lemma \ref{dpIsomLemma}.

Finally, we need to justify why we did not state $E$ must be $\AutMJz$-invariant
in Theorem \ref{KEexistenceIntroThm}, while it is needed to apply 
Theorem \ref{ExistencePrinc}. This follows from Futaki's theorem
\cite[p. 437]{Fut}. Indeed,
as in the proof of Claim \ref{FutClaim}
$$
\frac{d}{dt}E((\exp_It X)^\star\ovp) 
= C_X,
$$
for some $\RR\ni C_X$ depending on $X$ but not on $\ovp\in\H$. Also,
$$
\frac{d}{dt}E((\exp_I-t X)^\star\ovp) 
= -C_X.
$$
Thus, unless this derivative, i.e., $C_X$, is zero for every $X\in\autMJ$
and $\ovp\in\H$, the functional
$E$ cannot be bounded from below.
Now, properness of $E$ with respect to any nonnegative functional 
implies $E$ is bounded from below.
 Thus, $J_G$-properness of $E$
implies it is $\AutMJz$-invariant.

\section {A proof of Tian's third conjecture: the strong\\ Moser--Trudinger inequality}
 \label{SecondSec}

We now explain the proof of Tian's third conjecture, namely
the strong Moser--Trudinger inequality for \KE manifolds. First, let us recall the statement.

Denote by $\Lambda_1$ the real eigenspace of
the smallest positive eigenvalue of $-\Delta_\o$, and set
$$
\calH_\o^\perp:=\{\vp\in\H\,:\, \int \vp\psi\on=0,\;\forall \psi\in\Lambda_1\}.
$$

\bconj
Suppose $(M,\JJJ,\o)$ is Fano K\"ahler-Einstein.
Then for some $C,D>0$,
\beq
\lb{EperpEq}
E(\vp) \geq C J(\vp) - D, \qq  \vp \in \H_\o^\perp.
\eeq
\econj

Observe that no invariance properties are assumed, and the functionals
are not taken on the quotient space. Instead, an orthogonality assumption is made.

Conjecture \ref{TianConj2} was originally 
motivated by results in conformal geometry related to 
the determination of the best constants in the 
borderline case of the Sobolev inequality.
By restricting to functions orthogonal to the first eigenspace of the Laplacian,
Aubin was able to improve the constant in the aforementioned inequality
on spheres \cite[p. 235]{Aubinbook}. This can be seen as the sort of coercivity
of the Yamabe energy occuring in the Yamabe problem,
and it clearly fails without the orthogonality assumption due to the presence
of conformal maps. Conjecture \ref{TianConj2} stands in clear analogy with
the picture in conformal geometry, by stipulating that coercivity of
the K-energy holds in `directions perpendicular to holomorphic maps'
(when $\o$ is \KEno, 
it is well-known that $\Lambda_1$ is in a one-to-one correspondence with
holomorphic gradient vector fields, in fact this is how Matsushima's
Theorem \ref{VectorFieldDecompEdgeThm} 
is proven  \cite{ga,CR}).
It can be thought of as a higher-dimensional fully
nonlinear generalization of the classical Moser--Trudinger inequality.

It is a rather simple 
consequence of the work of Bando--Mabuchi \cite{BM}
that when a \KE metric exists,
$J_G$-properness implies $J$-properness on $\H_\o^\perp$
\cite[Corollary 5.4]{Tian97},\cite[Lemma A.2]{zztoric},\cite[Theorem 2.6]{Tian2012}. 
We now explain how to carry this through.
The key is to study the Aubin functional restricted to orbits of $\AutMJz$
and identify the minimizers and relate them to the first eigenspace.

Fix $\eta\in\H$.
Let $F_\eta:\AutMJz\ra \RR_+$ be given by 
$$
F_\eta(g):=(I-J)(g^\star\eta)
=
\V\frac1{n+1}\int_M\i\del\vp_g\w\dbar\vp_g\w\sum_{l=0}^{n-1}
(n-l)\o^{n-l-1}\w(g^\star\eta)^{l},
$$
where $\vp_g\in\H_\o$ is such that $g^\star\eta=\o+\i\ddbar\vp_g$
(i.e., where the $I-J$ energy of $g^\star\eta$ 
with respect to  the reference form $\o$).

\begin{lemma}
\label{FetaCritLemma}
Suppose $(M,\JJJ,\eta=\o_\psi)$ is Fano K\"ahler--Einstein.
Then $h\in \AutMJz$ is a critical point of $F_\eta$ precisely if
$-\vp_h\in\H_{h^\star\eta}^\perp$.
\end{lemma}

\bpf
Using \eqref{IJIEq} and \eqref{AMVarEq},
\beq
\baeq
\lb{IJIEq}
\frac{d}{d\delta}(I-J)(\vp(\delta))
&=
\frac{d}{d\delta}\h{\rm AM}(\vp)
-
\frac{d}{d\delta}\V\int \vp(\delta)\o_{\vp(\delta)}^n
\cr
&=
\V\int\frac{d}{d\delta} \vp(\delta)\o_{\vp(\delta)}^n
-
\V\int\big(
\frac{d}{d\delta} \vp(\delta)
+
\vp(\delta)\Delta_{\o_{\vp(\delta)}}\frac{d}{d\delta} \vp(\delta)
\big)\o_{\vp(\delta)}^n
\cr
&=
-
\V\int
\vp(\delta)\Delta_{\o_{\vp(\delta)}}\frac{d}{d\delta} \vp(\delta)
\o_{\vp(\delta)}^n
.
\eaeq
\eeq
Writing $g_t=h\exp_ItX$ with $X\in\autMJ$, observe that
$$
\baeq
\i\ddbar\dot\vp_{h}
&=
\i\ddbar\dot\vp_{g_0}
\cr
&=
\frac {d}{dt}\Big|_0\big(\o+\i\ddbar\vp_{g_t} \big)
\cr
&=
\frac {d}{dt}\Big|_0g_t^\star\eta=
\frac {d}{dt}\Big|_0(\exp_ItX)^\star(h^\star\eta)
=
\i\ddbar\psi^X_{h^\star\eta}
.
\eaeq
$$
Therefore,
\beq
\baeq
\frac {d}{dt}\Big|_0F_\eta(g_t)
&=
-
\V\int \vp_{h}\Delta_{h^\star\eta}\dot \vp_{h}(h^\star\eta)^n
\cr
&=
-
\V\int \vp_{h}\Delta_{h^\star\eta}\psi^X_{h^\star\eta}(h^\star\eta)^n
\cr
&=
\V\int \vp_{h}\psi^X_{h^\star\eta}(h^\star\eta)^n.
\eaeq
\eeq
Since this holds for all $X\in\autMJ$, it follows that
$\vp_h\in \calH_{h^\star\eta}^\perp$.
\epf

\begin{lemma}
\label{FetaMinLemma}
Suppose $(M,\JJJ,\eta=\o_\psi)$ is Fano K\"ahler--Einstein.
By Theorem \ref{MatsIwasawaThm} then $\AutMJz=K^{\CCfoot}$ for a maximally compact
subgroup $K$. Suppose $\o\in\H^K$.
Then $F_\eta$ has a unique critical point which is a global minimum. 
\end{lemma}

\bpf
We start with the following observation.

\begin{exer} {\rm
\label{}
If $g\in\AutMJz$ preserves $\o$ then
$$(I-J)(g^\star\eta)=(I-J)(\eta).$$
} \end{exer}

Thus, using the Cartan decomposition (Corollary \ref{CartanCor}),
$F_\eta$ descends to a function on $\isom(M,g)$, still denoted by $F_\eta$,
$$
F_\eta(X)=(I-J)\big((\exp_I\JJJ X)^\star\eta\big).
$$
Now, we show that the function $(\exp_It\JJJ X)^\star\eta$ satisfies
a useful equation.

The Hodge decomposition implies that every $X\in\autMJ$
can be uniquely written as \cite{ga}
\beq
\lb{XDecompEq}
X=X_H+\nabla \psi^X_\o-\JJJ\nabla \psi^{\JJJsml X}_\o,
\eeq
where $\nabla$ is the gradient with respect to the Riemannian
metric associated to $\JJJ$ and $\o$,
and $X_H$ is the $g_{\o}$-Riemannian dual of a 
$g_{\o}$-harmonic $1$-form.

By \eqref{XDecompEq} 
and the fact that 
$X \in \isom(M,{g_{\eta}})$ (here $g_\eta$ denotes the Riemannian
metric associated to $\JJJ$ and $\eta$) it 
follows that  
\beq
\lb{JXGradEq}
\JJJ X=\nabla\psi^{\JJJsml X}_{\eta}
\eeq
is a gradient (with respect to $g_{\eta}$) vector field
\cite[Theorem 3.5]{Mabuchi87}. We set 
$$
\o_{\vp(t)}:=\o(t)=\exp_It\JJJ X.\eta.
$$ 
Thus, 
\beq
\lb{dototEq}
\dot \o(t)=
\frac d{dt} 
\exp_It\JJJ X.\eta = 
\calL_{\JJJsml X}\eta\circ \exp_It\JJJ X
=\i\ddbar\psi^{\JJJsml X}_{\eta}\circ\exp_It\JJJ X,
\eeq
and 
$$
\baeq
\ddot \o(t)
&=\i\ddbar\big((\JJJ X)(\psi^{\JJJsml X}_{\eta})\big)\circ\exp_It\JJJ X
\cr
&=\i\ddbar\big(d\psi^{\JJJsml X}_{\eta}(\JJJ X)\big) \circ\exp_It\JJJ X
\cr
&=\i\ddbar|\nabla\psi^{\JJJsml X}_{\eta}|^2\circ\exp_It\JJJ X
,
\eaeq
$$
since the $\eta$-Riemannian dual of $d\psi^{\JJJsml X}_{\eta}$ 
is $\nabla\psi^{\JJJsml X}_{\eta}$. Thus,
$$
\ddot \vp(t)-|\nabla\dot\vp(t)|^2_{\o_{\vp(t)}}=0.
$$

Next, we can generalize this computation slightly to obtain
an equation for the function 
$(\exp_I\JJJ((1-t)Y+tZ))^\star\eta$.
By \eqref{XDecompEq} and the fact that 
$X \in \isom(M,{g_{\eta}})$ (here $g_\eta$ denotes the Riemannian
metric associated to $\JJJ$ and $\eta$) it 
follows that  
\beq
\lb{JXGradEq}
\JJJ (Z-Y)=\nabla\psi^{\JJJsml (Z-Y)}_{\eta}
\eeq
is a gradient (with respect to $g_{\eta}$) vector field
\cite[Theorem 3.5]{Mabuchi87}. We set 
$$
\o_{\vp(t)}:=\o(t)=(\exp_I\JJJ((1-t)Y+tZ))^\star\eta.
$$ 
Thus, 
\beq
\lb{dototEq}
\baeq
\dot \o(t)
&=
\frac d{dt} 
(\exp_I\JJJ((1-t)Y+tZ))^\star\eta 
\cr
&= 
\calL_{\JJJsml (Z-Y)}\eta\circ \exp_I\JJJ((1-t)Y+tZ)
=\i\ddbar\psi^{\JJJsml (Z-Y)}_{\eta}\circ\exp_I\JJJ((1-t)Y+tZ),
\eaeq
\eeq
and 
$$
\baeq
\ddot \o(t)
&=\i\ddbar\big((\JJJ (Z-Y))(\psi^{\JJJsml (Z-Y)}_{\eta})\big)
\circ\exp_It\JJJ (Z-Y)
\cr
&=\i\ddbar|\nabla\psi^{\JJJsml (Z-Y)}_{\eta}|^2\circ\exp_I\JJJ((1-t)Y+tZ)
,
\eaeq
$$
 Thus, again,
\begin{equation}
\begin{aligned}
\label{GeodvpEq}
\ddot \vp(t)-|\nabla\dot\vp(t)|^2_{\o_{\vp(t)}}=0.
\end{aligned}
\end{equation}

Observe that 
$$
F_\eta((1-t)Y+tZ)=(I-J)\big((\exp_I\JJJ ((1-t)Y+tZ))^\star\eta\big).
$$
Therefore,
\beq
\baeq
\frac {d}{dt}\Big|_0F_\eta((1-t)Y+tZ)
&=
-
\V\int \vp(t)\Delta_{g_t^\star\eta}\dot \vp(t)(g_t^\star\eta)^n
\cr
&=
-
\V\int \vp(t)\Delta_{\o(t)}\dot \vp(t)\o(t)^n
\cr
&=
-
\V\int \dot \vp(t)\Delta_{\o(t)}\vp(t)\o(t)^n
\cr
&=
\V\int \dot \vp(t)n(\o-\o(t))\w\o(t)^{n-1},
\eaeq
\eeq
where $g_t:=\exp_I\JJJ ((1-t)Y+tZ)$, since $g_t^\star\eta=\o(t)$.
Also, using \eqref{GeodvpEq},
\beq
\baeq
\frac {d^2}{dt^2}\Big|_0F_\eta((1-t)Y+tZ)
&=
\V\int \ddot \vp(t)n(\o-\o(t))\w\o(t)^{n-1}
\cr
& 
\q
-
\V\int \dot \vp(t)n\i\ddbar\dot\vp(t)\w\o(t)^{n-1}
\cr
&
\q
+
\V\int \dot \vp(t)n(n-1)(\o-\o(t))\w\i\ddbar\dot\vp(t)\w\o(t)^{n-2}
\cr
&=
n\V\int |\nabla\dot\vp|^2\o\w\o(t)^{n-1}
-
n\V\int |\nabla\dot\vp|^2\o(t)^{n}
\cr
& 
\q
+
n\V\int \i\del\dot \vp(t)\w\i\dbar\dot\vp(t)\w\o(t)^{n-1}
\cr
&
\q
+
\V\int \dot \vp(t)n(n-1)\o\w\i\ddbar\dot\vp(t)\w\o(t)^{n-2}
\cr
&
\q
-
\V\int \dot \vp(t)n(n-1)\i\ddbar\dot\vp(t)\w\o(t)^{n-1}
\cr
&=
n\V\int |\nabla\dot\vp|^2\o\w\o(t)^{n-1}
-
n\V\int |\nabla\dot\vp|^2\o(t)^{n}
\cr
& 
\q
+
\V\int |\nabla\dot\vp|^2\o(t)^{n}
\cr
&
\q
+
\V\int \dot \vp(t)n(n-1)\o\w\i\ddbar\dot\vp(t)\w\o(t)^{n-2}
\cr
&
\q
+
(n-1)
\V\int |\nabla\dot\vp|^2\o(t)^{n}
\cr
&=
n\V\int |\nabla\dot\vp|^2\o\w\o(t)^{n-1}
\cr
&
\q
-
n(n-1)
\V\int \i\del\dot \vp(t)\w\i\dbar\dot\vp(t)\w\o\w\o(t)^{n-2}
\cr
&=
\frac nV\int \Big(
|\nabla\dot\vp|^2\o(t)  
-
(n-1)
\i\del\dot \vp(t)\w\i\dbar\dot\vp(t)
\Big)\w\o\w\o(t)^{n-2}
\cr
&
\ge
\frac nV\int 
\i\del\dot \vp(t)\w\i\dbar\dot\vp(t)
\w\o\w\o(t)^{n-2}
>0
,
\eaeq
\eeq
since if $\alpha,\beta$ are two positive (1,1)-forms then
$(\tr_\alpha\beta)\alpha-\beta\ge0$, in general, so have
$|\nabla\dot\vp|^2\o(t)  
-
n
\i\del\dot \vp(t)\w\i\dbar\dot\vp(t)\ge0.
$
Thus, $F_\eta$ is strictly convex on the vector space $\isom(M,g)$.
Now, observe that it is a proper function
by Lemma \ref{BMpropProp}.
Since
a proper strictly convex function attains a unique minimum,
 the proof is complete.
\epf

\begin{exer} {\rm
\label{IJIJExer}
Prove the formula (see, e.g., \cite[p. 140]{RThesis})
$$
(I-J)(\o,\eta)=J(\eta,\o),
$$
where $(I-J)(\o,\eta)$ is just $(I-J)(\vp)$ for any $\vp$ such
that $\eta=\ovp$, while $J(\eta,\o)$ is just $J$ (recall \eqref{AubinEnergyEq})
 ``of" $\o$ ``with
respect to" the reference $\eta$, in the sense that
$$
J(\eta,\o)=
V^{-1}\int_M\vp\eta^n
-
\frac{V^{-1}}{n+1}\int_M
\psi\sum_{l=0}^{n}\eta^{n-l}\w\o^{l},
$$ 
where $\psi$ satisfies $\o=\eta_\psi$.
} \end{exer}

\bprop
\lb{ProperEquivProp}
Suppose $(M,\JJJ,\eta)$ is Fano K\"ahler--Einstein.
If $E$ is $J_G$-proper then \eqref{EperpEq} holds.
\eprop

\bpf
According to Lemma \ref{FetaCritLemma},
the functional
$$
g\mapsto (I-J)(\o,g^\star\eta)
$$
has a critical point at the identity $g=\id$ if $\eta=\o-\i\ddbar\vp$
when $\vp\in \H_\eta^\perp$.
Now, by Exercise \ref{IJIJExer}, this is tantamount to
the functional
\beq
\lb{JgMapEq}
g\mapsto J(g^\star\eta,\o)
\eeq
having a critical point at the identity $g=\id$ if $\eta=\o-\i\ddbar\vp$
when $\vp\in \H_\eta^\perp$.

Suppose now that indeed $\vp\in \H_\eta^\perp$.
Then the functional \eqref{JgMapEq} has a critical point at $g=\id$.
By Lemma \ref{FetaMinLemma}, this is the unique minimum of this functional.
Thus, using also $\AutMJz$-invariance of $J$ yields
$$
J(\vp)=:J(\eta,\eta_\vp)=
J(\eta,\o)
=
\inf_{g\in G}
J(g^\star\eta,\eta_\vp)
=
\inf_{g\in G}
J(\eta,g^\star\eta_\vp).
$$
The last expression is precisely $J_G(\vp)$ (with respect to the
reference metric $\eta$ (not $\o$!)). By assumption
$E$ is $J_G$-proper, so, say, for concreteness,
$$
E(\vp)\ge C J_G(\eta)-D = CJ(\vp) - D,
$$
as desired.
(Observe that the proof also gives the converse, namely that 
if  \eqref{EperpEq} holds then  $E$ is $J_G$-proper.)
\epf

Therefore, Theorem \ref{KEexistenceIntroThm} and Proposition \ref{ProperEquivProp} 
confirm Tian's conjecture.

\begin{corollary} 
\lb{TianConj2Cor}
Conjecture \ref{TianConj2} holds.
\end{corollary}

\section*{Acknowledgments}

\addcontentsline{toc}{section}{Acknowledgments}

Thanks go to T. Darvas for many helpful discussions
and a stimulating collaboration on many of the results presented here;
to B. Clarke for a fruitful collaboration on the results presented
in \S\ref{CalabiSubSec} that stimulated some of the later developments;
to B. Berndtsson for help with Theorem \ref{HThm}; 
to J. Streets
for 
 organizing the Winter School
as well as to the numerous students who attended it and asked stimulating questions; 
and to the staff
at MSRI for providing an excellent research environment
(supported by NSF grant DMS-1440140), 
where these notes were completed in January 2016.
Finally, it is a real privilege to dedicate this article to 
Gang Tian whose theorems, conjectures and vision have
shaped so much of this field, and whose guidance and encouragement
have been invaluable in my own pursuits
over the years. 
This research was also supported by BSF grants 2012236, 
a Sloan Research Fellowship, and NSF grants DMS-1206284,1515703.

\def\bi{\bibitem}



{\sc University of Maryland} 

{\tt yanir@umd.edu}


\begin{thebibliography}{1}

\addcontentsline{toc}{section}{References}

\bibitem{Aubin1970} T. Aubin,
M\'etriques Riemanniennes et courbure, J. Differential Geom.
4 (1970), 383--424.

\bibitem{Aubin1976} \opcit, \'{E}quations du type {M}onge-{A}mp\`ere sur les vari\'et\'es k\"ahl\'eriennes
compactes, C. Rendus Acad. Sci. Paris 283 (1976), 119--121.

\bibitem{Aubin1978} \opcit,
\'{E}quations du type {M}onge--{A}mp\`ere sur les vari\'et\'es k\"ahl\'eriennes
compactes, Bull. Sci. Math. 102 (1978), 63--95.


\bibitem{Aubin84} \opcit,
R\'eduction du cas positif de l'\'equation de Monge-Amp\`ere sur les vari\'et\'es k\"ahl\'eriennes compactes \`a la d\'emonstration d'une in\'egalit\'e, J. Funct. Anal. 57 (1984), 143--153. 

\bi{Aubinbook}  \opcit,
Some nonlinear problems in Riemannian Geometry, 
Springer, 1998.


\bi{Bamler}R. Bamler, Cheeger--Colding--Naber theory, preliminary draft, 2016.

\bibitem{BK88} S. Bando, R. Kobayashi, Ricci--flat K\"ahler metrics on affine algebraic manifolds,
Lecture notes in Math. 1339 (1988), 20--31.
 
\bibitem{BM}S. Bando, T. Mabuchi, Uniqueness of K\"ahler--Einstein metrics
modulo connected group actions, in: Algebraic Geometry,
Sendai, 1985, Advanced Studies in Pure Mathematics 10, 1987, pp. 11--40.

\bi{BatySeliv} V.V. Batyrev, E.N. Selivanova, 
Einstein--K\"ahler metrics on symmetric toric Fano manifolds,
Crelle's J. 512 (1999), 225--236.

\bibitem{bt}E. Bedford, B.A. Taylor, 
A new capacity for plurisubharmonic functions, Acta Math. 149 (1982),  1--40.

\bibitem{brm1} R.J. Berman, A thermodynamical formalism for 
Monge--Amp\`ere equations, Moser--Trudinger inequalities and 
K\"ahler--Einstein metrics. Adv. Math. 248 (2013), 1254--1297.

\bi{Berman-zerotemp} \opcit,
From Monge--Amp\`ere equations to envelopes and geodesic rays in the zero temperature limit,
preprint, arxiv:1307.3008.

\bibitem{bb} R.J. Berman, R. Berndtsson,  Convexity of the K-energy on the 
space of K\"ahler metrics, preprint, arxiv:1405.0401.

\bi{BBouc}
R.J. Berman, S. Boucksom,
Growth of balls of holomorphic sections 
and energy at equilibrium, Invent. Math. 181 (2010),  337--394. 

\bibitem{bbgz} R.J. Berman, S. Boucksom, V. Guedj, A. Zeriahi, A variational approach to complex Monge-Amp\`ere equations,  Publ. Math. Inst. Hautes \'Etudes Sci. 117 (2013), 179--245.

\bi{BDL}R.J. Berman, T. Darvas, C.H. Lu,
Convexity of the extended K-energy and the large time behaviour of the weak Calabi flow,
 preprint, arxiv:1510.01260. 

\bi{BDL2}R.J. Berman, T. Darvas, C.H. Lu,
Regularity of weak minimizers of the K-energy and applications to properness and K-stability, preprint, arxiv:1602.03114.

\bibitem{Bl2000} Z. \Blocki, Interior regularity of the complex 
Monge--Amp\`ere equation in convex domains, Duke Math. J. 105 (2000), 167--181.

\bibitem{Bl} Z. \Blocki, The complex Monge--Amp\`ere equation on compact K\"ahler manifolds,
Lecture notes, February 2007.

\bibitem{begz} S. Boucksom, P. Eyssidieux, V. Guedj, 
A. Zeriahi, Monge-Amp\`ere equations
in big cohomology classes, Acta Math. 205 (2010), 199--262.

\bibitem{brn}B. Berndtsson, A Brunn--Minkowski type inequality for 
Fano manifolds and some uniqueness theorems in K\"ahler geometry, 
Invent. math. 200 (2015), 149--200.

\bi{Besse} A.L. Besse, Einstein manifolds, Springer, 1987.

\bibitem{Bl12} Z. B\l ocki, On geodesics in the space of K\"ahler metrics, 
in: Advances in Geometric Analysis (S. Janeczko et al., Eds.),  International Press, 2012, pp. 3--20.
 
\bibitem{bbegz} S. Boucksom, R. Berman, P. Eyssidieux, V. Guedj, 
A. Zeriahi, 
K\"ahler--Einstein metrics and the K\"ahler-Ricci flow on log Fano 
varieties, preprint, arxiv:1111.7158.

\bi{Bourg} J.-P. Bourguignon, Invariants int\'egraux fonctionnels pour des \'equations aux d\'eriv\'ees
partielles d'origine g\'eom\'etrique, 
Lecture Notes in Mathematics { 1209}, Springer, 1986, pp. 100--108. 

\bi{Bump}D. Bump, Lie groups, Springer, 2004.

\bibitem{Calabi54} E. Calabi,
The variation of \K metrics.
I. The structure of the space; II. A minimum problem,
Bull. Amer. Math. Soc. 60 (1954), 167--168.

\bi{Calabi1984} \opcit, Extremal metrics. II, in:{ Differential geometry and complex analysis} (I. Chavel et al., Eds.),
Springer, 1985, pp. 95--114.


\bi{Calamai} S. Calamai, The Calabi's metric for the space of K\"ahler metrics,
preprint, arxiv:1004.5482. 

\bibitem{CR} I.A. Cheltsov, Y.A. Rubinstein,
Asymptotically log Fano varieties, Adv. Math. 285 (2015), 1241--1300. 

\bibitem{Ch00}X.X. Chen, The space of K\"ahler metrics, J. Differential Geom. 56 (2000), 189--234.

\bibitem{CC} X.X. Chen, J. Cheng, On the constant scalar curvature K\"ahler metrics, II, existence results, preprint,  arxiv:1801.00656.

\bibitem{CC2} X.X. Chen, J. Cheng, On the constant scalar curvature K\"ahler metrics, III, general automorphism group, preprint, arxiv:1801.05907.

\bibitem{cds} X.-X. Chen, S.K. Donaldson, S. Sun, 
K\"ahler-Einstein metrics on Fano manifolds, 
J. Amer. Math. Soc. 28 (2015), 183--278.

\bibitem{CT} X.-X. Chen, G. Tian, Ricci flow on K\"ahler-Einstein surfaces,
Inv. Math. 147 (2002), 487--544.

\bi{Clarke}B. Clarke, The completion of the manifold of Riemannian metrics, J. Different. Geom. 93 (2013), 203--268.

\bi{ClarkeR}B. Clarke, Y.A. Rubinstein,
 Ricci flow and the metric completion of the space of 
K\"ahler metrics, Amer. J. Math. 135 (2013), 1477--1505.

\bibitem{da3} T. Darvas, The Mabuchi completion of the space of 
K\"ahler potentials, preprint, arxiv:1401.7318.

\bibitem{da4}  \opcit,
The Mabuchi geometry of finite energy classes, 
preprint, Adv. Math. 285 (2015), 182--219.

\bi{DarvasCal} \opcit,
Comparison of the Calabi and Mabuchi geometries and
applications to geometric flows, preprint, 2015.

\bi{Dsur} \opcit, Geometric pluripotential theory on K\"ahler manifolds, 
preprint on the author's webpage.

\bibitem{dh} T. Darvas, W. He, Geodesic rays and K\"ahler-Ricci trajectories on Fano 
manifolds, preprint, arxiv:1411.0774

\bi{DR1} T. Darvas, Y.A. Rubinstein,
 Kiselman's principle, the Dirichlet problem for the Monge--Amp\'ere equation, 
and rooftop obstacle problems,  J. Math. Soc. Japan 68 (2016), 773--796.


\bi{DR2}  \opcit,
Tian's properness conjectures and Finsler geometry of the space of K\"ahler metrics,
 J. Amer. Math. Soc. 30 (2017), 347--387. 

\bibitem{dembook} J.P. Demailly, Complex Analytic and Differential Geometry, \href{https://www-fourier.ujf-grenoble.fr/~demailly/manuscripts/agbook.pdf}{https://www-fourier.ujf-grenoble.fr/~demailly/manuscripts/agbook.pdf}.

\bi{Dervan} R. Dervan, 
Alpha invariants and coercivity of the Mabuchi functional on Fano manifolds,
preprint, arxiv:1412.1426.

\bibitem{Ding1988}
W.-Y. Ding, Remarks on the existence problem of positive {K}\"ahler-{E}instein 
metrics, Math. Ann. 282 (1988), 463--471.

\bibitem{dt} W.Y. Ding, G. Tian, The generalized Moser-Trudinger inequality, in: Nonlinear Analysis and Microlocal Analysis
(K.-C. Chang et al., Eds.), World Scientific, 1992, pp. 57--70.

\bi{Dolg}I.A. Dolgachev, Classical algebraic geometry: a modern view,
Cambridge University Press, 2012.

\bibitem{Don}S.K. Donaldson, 
Symmetric spaces, K\"ahler geometry and Hamiltonian dynamics, in: Northern
California Symplectic Geometry Seminar (Ya. Eliashberg et al., Eds.),
Amer. Math. Soc., 1999, pp. 13--33.

\bibitem{E} D.G. Ebin, The manifold of Riemannian metrics, in: Global analysis 
(S. S. Chern et al., Eds.), Proceedings of Symposia in Pure and Applied Mathematics 
15 (1970), pp. 11-40.

\bi{FKZ}
F. Ferrari, S. Klevtsov, S. Zelditch,
Random K\"ahler metrics, preprint, arxiv:1107.4575.

\bibitem{FujitaOdaka}
K.~Fujita, Y.~Odaka, {On the K-stability of Fano varieties and anticanonical divisors}, to appear in Tohoku Math. Journal.

\bi{Fut}A. Futaki, An obstruction to the existence 
of Einstein K\"ahler metrics, Invent. Math. 73 (1983), 437--443.

\bibitem{ga} P. Gauduchon, Calabi's extremal metrics: an elementary 
introduction, manuscript, 2010.

\bi{GT} D. Gilbarg, N.S. Trudinger, Elliptic partial differential equations of second order, 
Reprint of the 1998 edition, Springer, 2001.

\bibitem{guedj} V. Guedj, The metric completion of the Riemannian space of 
K\"ahler metrics, preprint, arxiv:1401.7857.

\bibitem{gz} V. Guedj, A. Zeriahi, The weighted Monge--Amp\`ere 
energy of quasiplurisubharmonic functions, 	
J. Funct. Anal. 250 (2007), 442--482.

\bibitem{hhswz}W. Hazod, K.H. Hofmann, H.P. Scheffler, M. Wustner, H. Zeuner, 
Normalizers of compact subgroups, the existence of 
commuting automorphisms, and applications to operator semistable measures, 
J. Lie Theory 8 (1998), 189--209.

\bi{Helg}S. Helgason,
Groups and geometric analysis,
Amer. Math. Soc., 2002.


\bi{H} L. H\"ormander,  An introduction to complex analysis in several variables, Third edition, North-Holland, 1990.

\bibitem{JMR}T. Jeffres, R. Mazzeo, Y.A. Rubinstein,
K\"ahler--Einstein metrics with edge singularities, (with an appendix
by C. Li and Y.A. Rubinstein), Annals of Math. 183 (2016), 95--176.

\bi{Kahler} E. K\"ahler, \"Uber eine bemerkenswerte Hermitesche Metrik, 
Abhandlungen aus dem 
Mathematischen Seminar der Universit\"at Hamburg  9 (1933), 173--186.

\bibitem{Lu68} Y.-C. Lu, Holomorphic mappings of complex manifolds,
J. Diff. Geom. 2 (1968), 299--312.

\bibitem{Mabuchi1986} T. Mabuchi, K-energy maps integrating Futaki
invariants, T\^ohoku Math. J. 38 (1986), 575--593.

\bibitem{Mabuchi87} \opcit,
Some symplectic geometry on compact K\"ahler manifolds I, Osaka J. Math. 24, 1987, 227--252.

\bi{Mat}Y. Matsushima, Sur la structure du groupe d'hom\'eomorphismes analytiques
d'une certaine vari\'et\'e k\"ahl\'erienne, Nagoya Math. J. 11 (1957), 145--150.

\bi{McFeron}
D. McFeron, The Mabuchi metric and the K\"ahler-Ricci flow,
Proc. Amer. Math. Soc. 142 (2014), 1005--1012.

\bibitem{PSS} D.H. Phong, N. Se\v sum, J. Sturm,
Multiplier ideal sheaves and the K\"ahler-Ricci flow,
Comm. Anal. Geom. 15 (2007), 613--632.


\bibitem{PhongSongSturm} D.H. Phong, J. Song, J. Sturm,
Complex Monge--Amp\`ere equations, in: Survey in Differential
Geometry XVII, Int. Press, 2012, pp. 327--410.

\bibitem{pssw} D.H. Phong, J. Song, J. Sturm, B. Weinkove, 
The Moser--Trudinger inequality on K\"ahler-Einstein manifolds, 
Amer. J. Math. 130 (2008), 651--665.

\bibitem{PhongSturmSurvey} D.H. Phong, J.  Sturm,
Lectures on stability and constant scalar curvature, in:
Current developments in mathematics 2007, Int. Press, 2009, pp. 101--176.

\bi{Post} M.M. Postnikov, Geometry VI. Riemannian geometry, Springer, 2001.

\bi {Ransford} T. Ransford, Potential theory in the complex plane, Cambridge University Press, 1995.

\bibitem{RJFA} Y.A. Rubinstein, On energy functionals, K\"ahler-Einstein metrics, and the Moser--Trudinger--Onofri neighborhood, J. Funct. Anal. 255, special issue dedicated to Paul Malliavin (2008), 2641--2660.

\bi{R08} \opcit, 
Some discretizations of geometric evolution equations and the Ricci iteration on the space of K\"ahler metrics, Adv. Math. 218 (2008), 
1526--1565.

\bibitem{RThesis} \opcit,
Geometric quantization and
dynamical constructions on the space of K\"ahler metrics,
Ph.D. Thesis, Massachusetts Institute of Technology, 2008.


\bibitem{R14} \opcit,
Smooth and singular K\"ahler--Einstein metrics, 
in: Geometric and Spectral Analysis, (P. Albin et al., Eds.), Contemp. Math. 630,
Amer. Math. Soc. and Centre de Recherches Math\'ematiques, 2014, pp. 45--138. 

\bibitem{RZ1} Y.A. Rubinstein, S. Zelditch, The Cauchy problem
for the homogeneous Monge--Amp\`ere equation, I. Toeplitz quantization,
J. Differential Geom. 90 (2012), 303--327.

\bi{Schouten1} J.A. Schouten, 
Ueber unit\"are Geometrie, { Proceedings Koninklijke Nederlandse Akademie van 
Wetenschappen Amsterdam} { 32} (1929), 457--465.


\bi{SchoutenVanDantzig}
J.A. Schouten, D. van Dantzig, 
Ueber die Differentialgeometrie einer Hermiteschen
Differentialform und ihre Beziehungen zu den Feldgleichungen der Physik, 
{ Proceedings Koninklijke Nederlandse
Akademie van Wetenschappen Amsterdam} { 32} (1929), 60--64. 

\bi{SchoutenVanDantzig2} \opcit, 
\"Uber unit\"are Geometrie, { Mathematische Annalen} { 103} (1930), 319--346.

\bibitem{Semmes1}
S. Semmes,  Complex Monge-Amp\`ere and symplectic manifolds,
Amer. J. Math. 114 (1992), 495--550.


\bibitem{Siu} Y.-T. Siu, Lectures on Hermitian--Einstein metrics for 
stable bundles and \KE metrics,  Birkh\"auser, 1987.  

\bibitem{St} J. Streets, Long time existence of minimizing movement solutions of Calabi flow, Adv. Math. 259 (2014) 688--729.

\bibitem{st} M. Stroppel, 
Locally compact groups,	EMS Textbooks in Mathematics, 2006.

\bibitem{Thomas} R.P. Thomas,
Notes on GIT and symplectic reduction for bundles and varieties, in: Surveys
in Differential Geometry: Essays in memory of S.-S. Chern (S.-T. Yau, Ed.),
Int. Press, 2006, pp. 221--273.

\bibitem{Tian1987} G. Tian, On K\"ahler--Einstein metrics on certain K\"ahler
manifolds with $C_1(M)>0$, Inv. Math. 89 (1987), 225--246.


\bibitem{T1994} \opcit,
The K-energy on hypersurfaces and stability,
Comm. Anal. Geom.  2 (1994), 239--265.

\bi{Tian97}  \opcit,
K\"ahler-{E}instein metrics with positive scalar curvature,
Invent. Math. 130 (1997), 1--37.

\bibitem{Tianbook}  \opcit,
Canonical Metrics in \K Geometry, Birkh\"auser, 2000.

\bibitem{Tian2012}  \opcit,
Existence of Einstein metrics on Fano manifolds, 
in: Metric and Differential Geometry (X.-Z. Dai et al., Eds.),
Springer, 2012, pp. 119--159.

\bi{T15} \opcit,
K-stability and K\"ahler-Einstein metrics,
Comm. Pure. Appl. Math. 68 (2015), 1085--1156.

\bi{TZ00} G. Tian, X. Zhu, A nonlinear inequality of Moser-Trudinger type, 
Calc. Var. PDE 10 (2000), 349--354.

\bi{Wanby}G. Wanby,
Subharmonic and strongly subharmonic functions in case of variable coefficients, Math. Scand. 22 (1968), 283--309.

\bibitem{WJ} R. Wang, J. Jiang, Another approach to the Dirichlet problem
for equations of Monge--Amp\`ere type, Northeastern Math. J. 1 (1985), 27--40.

\bi{WangZhu} X.-J. Wang, X.-H. Zhu, K\"ahler--Ricci solitons 
on toric manifolds with positive first Chern class, Adv. Math. 188 (2004),
87--103.

\bibitem{Wu} D. Wu, \KE metrics of negative Ricci curvature on general
quasi-projective manifolds, Comm. Anal. Geom. 16 (2008), 395--435.

\bibitem{Yau1978} 
S.-T. Yau,
On the Ricci curvature of a compact K\"ahler
manifold and the complex Monge--Amp\`ere equation, I, Comm. Pure
Appl. Math. 31 (1978), 339--411.


\bibitem{ze} A. Zeriahi, Volume and capacity of 
sublevel sets of a Lelong class of plurisubharmonic functions, 
Indiana Univ. Math. J. 50 (2001), 671--703.

\bibitem{zztoric} B. Zhou, X. Zhu, Relative K-stability and modified K-energy on toric manifolds, Adv. Math. 219 (2008) 1327--1362.


\end{thebibliography}
\end{document}